\documentclass{amsart}
    \usepackage{amsmath, amssymb, amsthm,graphicx}
    \usepackage{setspace}
    \usepackage{xspace}
    \usepackage[utf8]{inputenc}
    \usepackage[english]{babel}
    \usepackage{soul}
    \usepackage{mathrsfs}
    \usepackage{verbatim}
    \usepackage{tikz}
    \usepackage{xcolor}
    \usepackage{bm}
    \usepackage{url}
    \usetikzlibrary{fit,calc,positioning,decorations.pathreplacing,matrix}
    \usetikzlibrary{arrows,decorations.markings}
    \usetikzlibrary{shapes.geometric}
    \usetikzlibrary {arrows.meta}
     \usetikzlibrary{patterns}

    \numberwithin{equation}{section}
    \theoremstyle{plain}
    \newtheorem{thm}{Theorem}[section]
    \newtheorem{prp}[thm]{Proposition}
    \newtheorem{cor}[thm]{Corollary}
    \newtheorem{lm}[thm]{Lemma}

    \theoremstyle{definition}
    \newtheorem{df}[thm]{Definition}
    \newtheorem{ex}[thm]{Example}
    \newtheorem*{conj}{Conjecture}
    \newtheorem{nota}{Notation}
    \newtheorem{quest}{Question}
    \newtheorem*{ft}{Fact}

    \newcommand\NN{\mathbb{N}}
    
    \newcommand\DD{\Delta}
    \newcommand\GDD{\delta}

    \newcommand{\SD}[2]{\textrm{D}(#2,#1)}
    \newcommand{\SLD}[2]{\textrm{L}(#2,#1)}
    \newcommand{\SRD}[2]{\textrm{R}(#2,#1)}
    \newcommand{\SLDA}[3]{\textrm{L}_{#3}(#2,#1)}
    \newcommand{\SRDA}[3]{\textrm{R}_{#3}(#2,#1)}

    \newcommand{\ggn}{(g_n,\ldots, g_1)}
    \newcommand{\ggm}{(g_m,\ldots, g_1)}

    \newcommand{\hhn}{(h_n,\ldots, h_1)}
    
    \newcommand{\ND}[1]{[#1]}
    
    \newcommand{\dpg}{\textrm{\textrm{dpt}}(g)}

    \newcommand{\ii}{i}
    \newcommand{\jj}{j}
    \newcommand{\kk}{k}
    \newcommand{\kkp}{{k+1}}
    \newcommand{\kko}{{k-1}}
    \newcommand{\rr}{r}
    
    \newcommand{\rro}{{r-1}}
    \newcommand{\iip}{{i+1}}
    \newcommand{\iio}{{i-1}}
    
    \newcommand{\jjo}{{j-1}}
    \newcommand{\nn}{n}

    \renewcommand\gg{g}
    
    \newcommand\inv{^{-1}}
    
    \newcommand{\sig}[1]{\sigma_{#1}}
    \renewcommand{\ss}[1]{s_{#1}}
    \renewcommand{\tt}[1]{t_{#1}}
    \newcommand{\ssi}{s_{i}}
    \newcommand{\ssj}{s_{j}}
    \newcommand{\ssk}{s_{k}}

    \newcommand{\ssinv}[1]{s_{#1}\inv}

    \newcommand\oTheta{\widetilde{\Theta}}
    \newcommand\resp{\emph{resp.}\xspace}

    \DeclareMathOperator{\num}{num}
    \DeclareMathOperator{\den}{den}
    \DeclareMathOperator{\supp}{supp}
    \DeclareMathOperator{\dpt}{dpt}
    \DeclareMathOperator{\brd}{brd}

    \newcommand\numL{\num_L}
    \newcommand\numR{\num_R}
    \newcommand\denL{\den_L}
    \newcommand\denR{\den_R}
    
    \newcommand\rem[1]{\pi(#1)}
    \newcommand\siginc[2]{\sigma_{#1\nearrow#2}}
    \newcommand\sigdec[2]{\sigma_{#1\searrow#2}}
    \newcommand\Nearrow{\begin{tikzpicture}[x=1pt,y=1pt,baseline=1]\draw[double distance = 0.8,-{Classical TikZ Rightarrow[length=2pt]}](0,0) -- (6,6);\end{tikzpicture}}
    \newcommand\Searrow{\begin{tikzpicture}[x=1pt,y=1pt,baseline=1]\draw[double distance = 0.8,-{Classical TikZ Rightarrow[length=2pt]}](0,6) -- (6,0);\end{tikzpicture}}
    \newcommand\sigincsq[2]{\sigma_{#1\Nearrow#2}}
    \newcommand\sigdecsq[2]{\sigma_{#1\Searrow#2}}
    \newcommand\muinc[2]{\mu_{#1\nearrow#2}}
    \newcommand\nuinc[2]{\nu_{#1\nearrow#2}}
    \newcommand{\mymarginpar}[1]{\ifodd\thepage\marginpar{#1}\else\marginpar{\hfill #1}\fi} 
    \newcommand\tail[2]{\tau_{\raisebox{-2pt}{\scriptsize$#1$}}(#2)}
    
    \title{Alternating decomposition of a product in a spherical type Artin-Tits monoid}
    \author{Jean Fromentin  and Eddy Godelle}
     
    \date{\today}

\begin{document}
    \begin{abstract} A Dehornoy structure is a tool used to obtain a left order in the Garside group~$G$. Arcis and Paris introduced two conditions, known as Condition~$A$ and Condition~$B$. When they are both satisfied, they ensure that~$G$ possesses a Dehornoy structure.  These two conditions describes how the product in the corresponding monoid behaves with an alternating decomposition.  Here, we consider the case of a spherical type Artin group~$G$. We prove that Condition A almost always holds, but Condition~$B$ is rarely satisfied.
    \end{abstract}
    \maketitle
    \section{Introduction}
    A group is said to be left-ordered if it can be equipped with a total order that is invariant under multiplication on the left. 
    Left-orderings on groups are related to important open conjectures, such as Zimmer's Conjecture \cite{Brownetall2020} or the $L$-space Conjecture \cite{Boyeretall2013}. 
    We refer to~\cite{KoM}  for a general presentation of the theory of left-ordered groups.
    It is well-known that braid groups are left-ordered with Dehornoy's order \cite{Burckel97,Dehornoy94}.
    This order is of interest in its own right as it is related to self-distributivity~\cite{Dehornoy94} and to curve diagrams of braids~\cite{Fennetall99}.
    But it also leads to the first example of groups having infinitely many left-orderings but  isolated left-orderings~\cite{DubrovinaDubrovin2001,DDRW, Ito2013}.
    The subject is still in progress (see \cite{MannRivas2018,Malicetetall2019,Matsumoto2020} for instance). 
    In general this is not an easy problem to decide whether or not a given finitely presented group is left-ordered, and there are various attempts  to extend Dehornoy's order to families of groups that are related to braid groups.  
    Among them are the family of Artin-Tits groups and the family of Garside groups.  
    Indeed, Braid groups  belong to both families, and in~\cite{DDRW} (see page 302), the authors address the  two following questions : 
    
    \begin{quest}~
    
    (i)  Which Artin-Tits groups are left-orderable ?
    
    (ii)  Is every Garside group left-orderable ?
    \end{quest}  
    
    Here we focus on those groups that belong to both families, namely to Artin-Tits groups of spherical type.
    See Figure~\ref{F:CoxDiag} for a complete list of Artin-Tits groups of spherical type that are irreducible,  classified by their corresponding Dynkin's diagrams.
    Even in this specific cases few is known.  Dehornoy's order is built  on the notion of the~$\sigma$-\emph{positivity} of a braid.
    So, a first attempt consists on generalising this  approach to obtain new orders (see \cite{Sibert08} and~\cite{Ito2013}). 
    But, it is hopeless to expect to extend this method to other spherical type Artin-Tits groups (see \cite{Sibert08}), except for Artin-Tits groups of type B, that are subgroups of braids groups.  
    So, an alternative approach has to be find. 
    This is what is done  in \cite{ArP2019} by Arcis  and Paris (see also  \cite{Arcis2017}).  
    Using the notion of the alternating normal form introduced by Dehornoy (see \cite{Deh2008,Fromentin2011,FromentinParis2012}), the notion of a \emph{Dehornoy structure}  on a  Garside group is defined.
    It is proved that a Dehornoy structure can be used to build a left-ordering.  
    Then two conditions are introduced, namely \emph{Condition~A} and~\emph{Condition~B}.  When both conditions are satisfied in a Garside group, then the latter possesses a Dehornoy structure \cite[Theorem~3.2]{ArP2019}.  
     Even for Artin-Tits groups, the question of whether or not Condition~$A$ and Condition~$B$  hold is open except for only few cases.  
    
    Let us introduce  here the properties of Garside groups that we shall need to state the two Conditions and provide our main results. 
    We refer to the next section for precise definitions and to \cite{DDGKM} for a general presentation of Garside Theory.
    A Garside group is the group of fractions of a Garside monoid, which is a monoid equipped with several good properties regarding the divisibilities on the right and on the left. In particular a Garside monoid possesses a unique minimal generating set, called its \emph{atom set}, and  a family of particular elements called the \emph{Garside elements} of the monoid.  Moreover Garside monoids  (\emph{resp.} Garside groups) possess a family of special submonoids  (\emph{resp.} subgroups)  called \emph{parabolic submonoids} (\emph{resp.} parabolic subgroups) \cite{God5}.  
    Parabolic submonoids are themselves Garside monoids and their atom sets are some subsets of the atom set of the whole Garside monoid.  
    A key property in what follows is that, given a parabolic submonoid, the set of left (or right) divisors of any element  possesses  a unique  maximal element  (see Proposition~\ref{P:ExistenceTail} below). 
     
    Consider a Garside monoid $M$ with atom set $S$ and  Garside element~$\DD$. 
    Assume that $M_{S_1},M_{S_2}$ are two non-trivial proper  parabolic submonoids of $M$ with atom sets $S_1$ and $S_2$, respectively.  
    Assume moreover that  $S = S_1\cup S_2$ where $S_1$ and $S_2$ are proper subsets of $S$.
     Following~\cite{Deh2008}, when the latter property holds, we say that  $(S_2,S_1)$ is a proper covering of $M$. 
    Then, any element  $g$ of $M$ possesses (alternating)  decompositions $g_k\cdots g_1$  where  $g_1$  belong to $M_{S_1}$ and  $g_i$ lies in $M_{S_j}$ with $i+j$ even for any index~$i$.
    This decomposition becomes unique \cite{Deh2008} when one requires that each $g_i$ is maximal for the right-divisibility and $g_i\neq 1$ for $i\geq 2$, see Proposition~\ref{P:AlternatingNormalForm} below.
    Note that $g_1$ may be equal to $1$.
    The \emph{depth} of $g$, denoted by $\dpg$,  is the number of $g_i$ not belonging to $M_{S_1}$.
    
The covering $(S_2, S_1)$  is said to satisfy Condition~$A$ with respect to~$\DD^p$, for some $p\geq 1$,  when for all  $t\geq 1$ we have 
$\dpt(\DD^{pt}) - 1 = t\times (\dpt(\DD^p) - 1)$ (see Definition~\ref{D:ConditionA}) .   
    Note that here we follow the definition  stated in~\cite{ArP2019}, which is different  from (but equivalent to) the one given in~\cite{Arcis2017}. 
    Note also that we state the definition in a more general framework than in~\cite{ArP2019} as we do not assume here that~$\DD^p$ is central. 
    Clearly, if the covering~$(S_2, S_1)$ satisfies Condition~$A$ for~$\DD^p$ then it satisfies Condition~$A$ for all Garside elements~$\DD^{pq}$ for $q\geq 1$. 
     Condition~B  is too technical to be presented here, and we postpone it to Definition~\ref{D:ConditionB}.
    
   \subsection*{What is done in this paper}
   
 We focus on irreducible spherical type Artin-Tits monoids. This is not a restriction as explained in Section~\ref{S:left order}.  We also restrict our study to proper and irreducible covering $(S_2,S_1)$ of the generating set~$S$ of  a spherical type Artin-Tits monoid  $M$, \emph{i.e.}, $S_1$ and $S_2$ are proper subsets of $S$ and the corresponding parabolic submonoids $M_{S_1}$ and $M_{S_2}$ are also  irreducible. This hypothesis on $M_{S_1}$ and $M_{S_2}$  is not so restrictive in order to obtain a left ordering  on $M$  because it is necessary to repeat the process replacing $M$ with both $M_{S_1}$ and $M_{S_2}$. We address independently the question of whether Condition~$A$ and Condition~$B$ hold in \ref{S:ConditionA} and  \ref{S:ConditionB}, respectively. It is turn out that the answers are quiet different.  Write $\Phi(g) = \DD \cdot g\cdot \DD\inv$ for $G$ in $M$, where $\DD$ is the Garside element of $M$. Say that a proper and irreducible covering $(S_2,S_1)$ is $\Delta^p$- regular if $\Phi(\{S_2,S_1\}) = \{S_2,S_1\}$.  Say it  is $\Delta^p$-fixed  if $\Phi(S_2,S_1) = (S_2,S_1)$.  Note that for $p$ even, or $M$ not of type $A$, $D_{2r+1}$ or  $E_6$ then  $(S_2,S_1)$  must be $\Delta^p$-fixed.  
  \begin{thm} (Corollary~\ref{C:ConditionAIff})  \label{T:1er resultat} Assume $M$ is an irreducible Artin monoid of spherical type, with generating set $S$.
   Let~$p$ be a positive integer and $(S_2,S_1)$ be  a $\Delta^p$-regular covering of $M$. 
    Assume  $\SLDA{S}{\DD^p}{*} = \Phi^p(S_2\boxminus S_1)$.  Then,~$(S_2, S_1)$  satisfies Condition~$A$ with respect to~$\DD^p$ if and only if~$(S_2,S_1)$ is $\DD^{p}$-fixed.  
  \end{thm} 
 We postpone the definition of the condition $\SLDA{S}{\DD^p}{*} = \Phi^p(S_2\boxminus S_1)$ to~Section~\ref{S:ConditionA},  but we want to stress out the fact that the condition is easy to verify and often holds. For instance (see Corollary~\ref{C:ConditionA:GammaLine}),  for any proper and irrreducible covering it holds for any $p\geq 1$ in type $B$, $F_4$, $H_3$, $H_4$ and for any $p$ even in type $A$. We refer to~Section~\ref{S:ConditionA} for the  obtained results in types $D$ and $E$.\\

 The results for Condition~$B$ are not as positive. Say that a subset $T$ of $S$ is co-abelian when $S\setminus T$ generates an abelian monoid. Say that a covering~$(S_2,S_1)$ is co-abelian when both $S_1$ and $S_2$ are co-abelian. 
  \begin{thm} ( Proposition~\ref{P:SmallCoAbelian}) \label{T: 2eme resultat}
    Let $M$ be an irreducible spherical type Artin monoid of type $A, D$ or $E$ with generating set~$S$. Let $(S_2,S_1)$ be a proper and irreducible covering of $S$. 
    The covering $(S_2,S_1)$ may satisfy condition~$B(p)$ for some $p\geq 1$ only if $(S_2,S_1)$ is co-abelian.
    \end{thm}
For instance in type $A$, the only proper and irreducible covering that satisfies  Condition~$B$ is the one associated to the Dehornoy order.   We refer to~Section~\ref{S:ConditionB} for the obtained results in types $B$, $I(n)$,  $F_4$, $H_3$ and $H_4$.

We do not obtain new coverings for which condition~$B$ is proved, however, the above negative result associated with computation experiments  lead to identify some good candidates. For instance consider the co-abelian regular coverings in type $D$. Also,  consider the following intriguing case : take $M$ of type $B_3$ with the notations of  Figure~\ref{F:CoxDiag} with $S_1 = \{\sigma_0\}$ and $S_2 = \{\sigma_1,\sigma_2\}$. Then $(S_2, S_1)$ is a $\Delta$-regular covering which by Theorem~\ref{T:1er resultat} satisfies Condition~$A$.  None of  the results we proved nor our computing experiment contredict the fact that this covering may satisfies condition~$B$. We were unable to obtain a counterexample by hand of Condition B, nor a proof that Condition B is satisfied. However, we suspect that the latter is true. Finally, regarding the existence of a Dehornoy structure, Conditions A and B,  together,  form a sufficient but not necessary condition. Since Condition A often holds and Condition B often fails, one may wonder whether Condition B may be relax to provre the existence of new Dehornoy structure. We point out that a carreful reading of the proof of \cite[Theorem~3.2]{ArP2019} and our counter-example to condition B suggest this could be possible.  \\

 The paper is organised  as it follows. Section~\ref{S:preli} is a preliminary section, when we introduce the main objects and tools. In particular we recall the  notion of chain developped  in \cite{BrS}.  Section~\ref{S:alter} is devoted to the alternating normal from. Finally, in Section~\ref{S:ConditionA}  we focus on Condition~$A$;  in Section~\ref{S:ConditionB}  we focus on Condition~$B$.         
    
    \section{Preliminaries} \label{S:preli}

      \subsection{Artin--Tits monoids}
    
    A \emph{Coxeter graph} is a finite simple labelled graph $\Gamma = (S,E,m)$ with vertex set $S = \{\ss1,\ldots, \ss\nn\}$, edge set~$E$ and a label map~$m:~E\to \NN_{\geq 3}\cup\{\infty\}$. 
    For $s_i\not=s_j$ and $\{s_i,s_j\}\not\in E$ we put $m({s_i,s_j})=2$.
    To simplify notation, for $\{s_i, s_j\}$ in~$E$,  we write $m(i,j)$ instead of $m(\{s_i,s_j\})$.
    As usual, we write the label $m(e)$ of an edge $e$ only if $m(e)\geq 4$.
    
    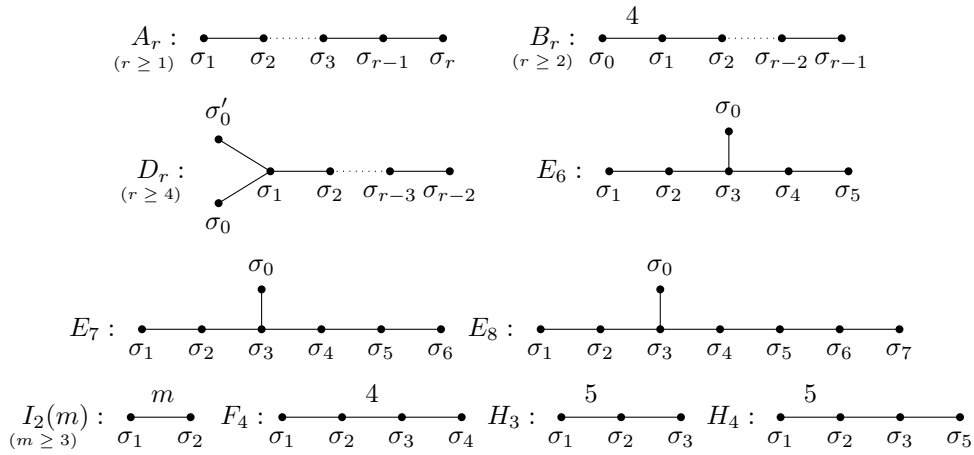
\begin{figure}[h]
     	\begin{center}
     		\begin{tikzpicture}[x=15pt,y=15pt]
     			\fill(0,0) circle (0.1) node[below=2pt]{$\sigma_1$};
     			\fill(1.5,0) circle (0.1) node[below=2pt]{$\sigma_2$};
     			\fill(3,0) circle (0.1) node[below=2pt]{$\sigma_3$};
     			\fill(4.5,0) circle (0.1) node[below=2pt]{$\sigma_{r-1}$};
     			\fill(6,0) circle (0.1) node[below=2pt]{$\sigma_{r}$};
     			\draw (0,0) -- (1.5,0);
     			\draw[dotted] (1.5,0) -- (3,0);
     			\draw (3,0) -- (6,0);
     			\draw (-1.25,0) node {$A_r: $};
     			\draw (-1.5,-0.65) node {\tiny $(r\geq 1)$};
     			\fill(10,0) circle (0.1) node[below=2pt]{$\sigma_0$};
     			\fill(11.5,0) circle (0.1) node[below=2pt]{$\sigma_1$};
     			\fill(13,0) circle (0.1) node[below=2pt]{$\sigma_2$};
     			\fill(14.5,0) circle (0.1) node[below=2pt]{$\sigma_{r-2}$};
     			\fill(16,0) circle (0.1) node[below=2pt]{$\sigma_{r-1}$};
     			\draw (10,0) -- (13,0);
     			\draw[dotted] (13,0) -- (14.5,0);
     			\draw (14.5,0) -- (16,0);
     			\draw (10.75,0) node[above=2pt] {$4$};
     			\draw (8.75,0) node {$B_r: $};
     			\draw (8.5,-0.6) node {\tiny $(r\geq 2)$};
     		\end{tikzpicture}
     	\end{center}
     	\begin{center}
     		\begin{tikzpicture}[x=15pt,y=15pt]
     			\fill(0.2,0.8) circle (0.1) node[above=2pt]{$\sigma'_0$};
     			\fill(0.2,-0.8) circle (0.1) node[below=2pt]{$\sigma_0$};
     			\fill(1.5,0) circle (0.1) node[below=2pt]{$\sigma_1$};
     			\fill(3,0) circle (0.1) node[below=2pt]{$\sigma_2$};
     			\fill(4.5,0) circle (0.1) node[below=2pt]{$\sigma_{r-3}$};
     			\fill(6,0) circle (0.1) node[below=2pt]{$\sigma_{r-2}$};
     			\draw (0.2,0.8) -- (1.5,0) -- (3,0);
     			\draw (0.2,-0.8) -- (1.5,0);
     			\draw[dotted] (3,0) -- (4.5,0);
     			\draw (4.5,0) -- (6,0);
     			\draw (-1.25,0) node {$D_r: $};
     			\draw (-1.5,-0.6) node {\tiny $(r\geq 4)$};
     			\fill(10,0) circle (0.1) node[below=2pt]{$\sigma_1$};
     			\fill(11.5,0) circle (0.1) node[below=2pt]{$\sigma_2$};
     			\fill(13,0) circle (0.1) node[below=2pt]{$\sigma_3$};
     			\fill(14.5,0) circle (0.1) node[below=2pt]{$\sigma_4$};
     			\fill(16,0) circle (0.1) node[below=2pt]{$\sigma_5$};
     			\fill(13,1) circle (0.1) node[above=2pt]{$\sigma_0$};
     			\draw (10,0) -- (14.5,0);
     			\draw(13,1) -- (13,0);
     			\draw (14.5,0) -- (16,0);
     			\draw (8.75,0) node {$E_6: $};
     		\end{tikzpicture}
     	\end{center}
     	\begin{center}
     		\begin{tikzpicture}[x=15pt,y=15pt]
     			\fill(0,0) circle (0.1) node[below=2pt]{$\sigma_1$};
     			\fill(1.5,0) circle (0.1) node[below=2pt]{$\sigma_2$};
     			\fill(3,0) circle (0.1) node[below=2pt]{$\sigma_3$};
     			\fill(4.5,0) circle (0.1) node[below=2pt]{$\sigma_4$};
     			\fill(6,0) circle (0.1) node[below=2pt]{$\sigma_5$};
     			\fill(3,1) circle (0.1) node[above=2pt]{$\sigma_0$};
     			\fill(7.5,0) circle (0.1) node[below=2pt]{$\sigma_6$};
     			\draw (0,0) -- (7.5,0);
     			\draw(3,1) -- (3,0);
     			\draw (-1.25,0) node {$E_7: $};
     			\fill(10,0) circle (0.1) node[below=2pt]{$\sigma_1$};
     			\fill(11.5,0) circle (0.1) node[below=2pt]{$\sigma_2$};
     			\fill(13,0) circle (0.1) node[below=2pt]{$\sigma_3$};
     			\fill(14.5,0) circle (0.1) node[below=2pt]{$\sigma_4$};
     			\fill(16,0) circle (0.1) node[below=2pt]{$\sigma_5$};
     			\fill(13,1) circle (0.1) node[above=2pt]{$\sigma_0$};
     			\fill(17.5,0) circle (0.1) node[below=2pt]{$\sigma_6$};
     			\fill(19,0) circle (0.1) node[below=2pt]{$\sigma_7$};
     			\draw (10,0) -- (19,0);
     			\draw(13,1) -- (13,0);
     			\draw (8.75,0) node {$E_8: $};
     		\end{tikzpicture}
     	\end{center}
     	\begin{center}
     		\begin{tikzpicture}[x=15pt,y=15pt]
     			\fill(0.2,0) circle (0.1) node[below=2pt]{$\sigma_1$};
     			\fill(1.7,0) circle (0.1) node[below=2pt]{$\sigma_2$};
     			\draw (0.2,0) -- (1.7,0);
     			\draw (1,0) node[above=2pt] {$m$};
     			\draw (-1.5,0) node {$I_2(m): $};
     			\draw (-2,-0.6) node {\tiny $(m \geq 3)$};
     			\fill(4,0) circle (0.1) node[below=2pt]{$\sigma_1$};
     			\fill(5.5,0) circle (0.1) node[below=2pt]{$\sigma_2$};
     			\fill(7,0) circle (0.1) node[below=2pt]{$\sigma_3$};
     			\fill(8.5,0) circle (0.1) node[below=2pt]{$\sigma_{4}$};
     			\draw (4,0) -- (8.5,0);
     			\draw (6.25,0) node[above=2pt] {$4$};
     			\draw (3,0) node {$F_4: $};
     			\fill(11,0) circle (0.1) node[below=2pt]{$\sigma_1$};
     			\fill(12.5,0) circle (0.1) node[below=2pt]{$\sigma_2$};
     			\fill(14,0) circle (0.1) node[below=2pt]{$\sigma_3$};
     			\draw (11,0) -- (14,0);
     			\draw (11.75,0) node[above=2pt] {$5$};
     			\draw (9.75,0) node {$H_3: $};
     			\fill(16.5,0) circle (0.1) node[below=2pt]{$\sigma_1$};
     			\fill(18,0) circle (0.1) node[below=2pt]{$\sigma_2$};
     			\fill(19.5,0) circle (0.1) node[below=2pt]{$\sigma_3$};
     			\fill(21,0) circle (0.1) node[below=2pt]{$\sigma_{5}$};
     			\draw (16.5,0) -- (21,0);
     			\draw (17.25,0) node[above=2pt] {$5$};
     			\draw (15.25,0) node {$H_4: $};
     		\end{tikzpicture}
     	\end{center}
     	
     	\caption{Coxeter diagrams of irreductible finite Coxeter groups}
     	\label{F:CoxDiag}
     \end{figure}

    The \emph{Artin--Tits monoid}~$M(\Gamma)$ associated with $\Gamma$ is given by the following presentation of a monoid:  
    \begin{equation} 
    \label{E:PresentationArtinMonoid}
    M(\Gamma) = \biggl\langle S\  \biggl | \begin{array}{cl}\ \ssi\ssj = \ssj\ssi & \textrm{ for } \{\ssi,\ssj\}\not\in E\\ \    \underbrace{\ssi\ssj\ssi\cdots}_{m(i,j) \textrm{ terms}}  = \underbrace{\ssj\ssi\ssj\cdots}_{m(i,j) \textrm{ terms}} &\textrm{ for } \{\ssi,\ssj\}\in E  \ \text{with}\ m(i,j)\not=\infty  \end{array}\biggr\rangle^+ 
    \end{equation}
    while the Artin--Tits braid group $G(\Gamma)$ of type $\Gamma$ is given by the same presentation, but considered as a presentation of a group:
    \begin{equation} 
    G(\Gamma) = \biggl\langle S\  \biggl | \    \begin{array}{cl}\ \ssi\ssj = \ssj\ssi & \textrm{ for } \{\ssi,\ssj\}\not\in E\\ \    \underbrace{\ssi\ssj\ssi\cdots}_{m(i,j) \textrm{ terms}}  = \underbrace{\ssj\ssi\ssj\cdots}_{m(i,j) \textrm{ terms}} &\textrm{ for } \{\ssi,\ssj\}\in E  \ \text{with}\ m(i,j)\not=\infty  \end{array} \biggr\rangle 
    \end{equation}
    For example, the Coxeter graph $\Gamma$  of  type $A_3$ in Figure~\ref{F:CoxDiag}  gives the monoid
    \[
    M(\Gamma) =  \left\langle \sigma_1,\sigma_2,\sigma_3\ \left|\ \begin{array} {rcl}\sigma_1\sigma_3 &=&\sigma_3\sigma_1\\\sigma_1\sigma_2\sigma_1&=&\sigma_2\sigma_1\sigma_2\\\sigma_2\sigma_3\sigma_2&=&\sigma_3\sigma_2\sigma_3\end{array}\right.\right\rangle^+.
    \]
    It is well-known that, for every Coxeter graph $\Gamma=(S,E,m)$,  the group $G(\Gamma)$ is the envelopping group of $M(\Gamma)$. Moreover, $M(\Gamma)$ embeds in $G(\Gamma)$ (see~\cite{Par2002}),  Therefore, in the sequel, we always see $M(\Gamma)$ as a submonoid of $G(\Gamma)$.  In particular, the monoid $M(\Gamma)$ is cancellative.  When adding to the presentation~\eqref{E:PresentationArtinMonoid}  the relations $\sigma^2 = 1$ for $\sigma$ in~$S$,  we obtain a Coxeter group, denoted by $W(\Gamma)$. 
    The group $G(\Gamma)$ and the monoid $M(\Gamma)$ are said of \emph{spherical type} whenever $W(\Gamma)$ is finite. In this case, $G(M)$ is indeed the group of fractions of $M(\Gamma)$. 
  
 In the sequel, a word over $S$ is called a $S$-word, and a word over $S\cup S^{-1}$ is called a $S^\pm$-word. The empty word is denoted by $\varepsilon$.  If $u$ is a $S$-word (\resp a $S^\pm$-word), by $\overline{u}$ we denote  the element of $M(\Gamma)$ (\resp of $G(\Gamma$)) it represents, 
 Two $S$-words $u$ and $v$ are said to be \emph{equivalent}, denoted $u\equiv v$, if they represent the same element, \emph{ie}, $\overline{u}=\overline{v}$. It follows from the presentation of  $M(\Gamma)$  that two equivalent words must have the same length.  As a consequence $M(\Gamma)$ is atomic  ( and therefore Noetherian for factors), with atom set $S$. 
     
     If the Coxeter graph $\Gamma$ can be decomposed into connected components as $\Gamma=\Gamma_1\cup\ldots\cup \Gamma_t$, then the braid group $G(\Gamma)$  is isomorphic  to $G(\Gamma_1)\times ... \times G(\Gamma_t)$ and a similar result holds for $M(\Gamma)$.  If $\Gamma$ is connected, its vertex set $S$, the group $G(\Gamma)$ and the monoid $M(\Gamma)$ are said to be irreducible.
   The list of  Coxeter graphs of irreducible Artin monoids of spherical type is given in Figure~\ref{F:CoxDiag}.

      \subsection{Left orderable groups and positive cones} 
      \label{S:left order}   
    A group $G$ is said to be \emph{left orderable} if it exists a total ordering $<$ on $G$ such that:
    \[
    \forall g,g',h \in G,\quad g < g' \Rightarrow hg < hg'.
    \]
    A subset $P$ of $G$ is a \emph{positive cone} whenever $P\cdot P \subseteq P$ and $G = P\inv \sqcup \{1\} \sqcup P$. There is a one-to-one correspondance between ordering and positive cones on $G$: a binary relation $<$ on $G$ is a left-ordering  if and only if $ \{g \in G,\ 1 < g\}$ is a positive cone. 
    
    Assume $G$ can be decomposed as $G \simeq G_1\times G_2$ with $G_1$ and $G_2$ two proper subgroups. Then, $G$ is left orderable if and only if $G_1$ and $G_2$ are. Cleary if $G$ is left orderable with respect to $<$, we define the orderings $<_1$ and $<_2$ on the groups $G_1$ and $G_2$, respectively, as the restrictions of $<$ on $G_1$ and $G_2$.
    Converlely, if $<_1$ and $<_2$ are left orders on $G_1$ and $G_2$, respectively, we define an order $<$ on $G$ by setting 
    \[
    (g_1,g_2) < (g'_1, g_2) \Leftrightarrow \begin{cases}
    g_1 <_1 g'_1, \\
    g_1 = g'_1\ \text{and}\ g_2<_2 g'_2.
    \end{cases}
    \]
    Hence we restrict our study of the left-orderability to irreducible Artin groups.

    \subsection{The Garside structure of an Artin-Tits Monoid} 
    \label{S:Garside}
     Let $M$ be a monoid.  For $g$ and $h$ in $M$, we say that $h$ is a \emph{right divisor} of $g$, or $g$ is a \emph{left multiple} of $h$, when there exists $g'$ in $M$ so that $g = g'h$.  
    In this case,  we write $h\preceq_R g$.
    Similarly, we say that $g'$ is a \emph{left divisor} of $g$, or $g$ is a \emph{right multiple} of $g'$, and we write $g'\preceq_L g$.
    The set of right divisors, \resp left divisors, of $g$ in $M$ is denoted by~$\SRD{M}{g}$, \resp $\SLD{M}{g}$. 
    Whenever the sets $\SRD{M}{g}$ and $\SLD{M}{g}$ are the same, we say that $g$ is \emph{balanced} and we denote it by~$\SD{M}g$.
    
    \begin{df}
    Let $M$ be a monoid and  assume $\GDD$ is a balanced element of $M$. We say that $M$ is Garside monoid with Garside structure $(M,\GDD)$ whenever :
    \begin{itemize}
    \item $M$ is cancellative and Noetherian for both $\preceq_L$ and $\preceq_R$;
    \item $(M,\preceq_L)$ and $(M,\preceq_R)$ are lattices;
    \item $\SD{M}\GDD$ is finite and generates $M$.
    \end{itemize}
    \end{df}
    Lattice operations of $(M,\preceq_R)$  are denoted $\vee_R$ and $\wedge_R$ while these of  $(M, \preceq_L)$ are denoted $\vee_L$ and $\wedge_L$.

    \begin{df}
    \label{D:DeltaUnmovable}
    Let $(M, \GDD)$ be a Garside structure and $g$ be an element of $M$.
    
    \begin{itemize}
    \item $g$ is said to be $\GDD$-unmovable if it it not left divisible by $\GDD$,
    \item the $\GDD$-form of $g$ is the unique decomposition $g=\GDD^t\cdot h$ where $t\in \NN$ and $h$ is $\GDD$-unmovable.
    \end{itemize}
    \end{df}

    The element $\GDD$ is called the Garside element of the Garside structure $(M, \GDD)$.
    For a given monoid $M$ there is no unicity of $\GDD\in M$ such that $(M, \GDD)$ is a Garside structure of $M$. 
    Indeed $(M, \GDD^p)$ is also a Garside structure of $M$ whatever $p\geq 1$.  Since~\cite{BrS}, we know that $M(\Gamma)$ is a Garside monoid whenever it is of spherical type.
    
    \begin{prp}
    Let $M$ be a spherical type Artin--Tits monoid with generating set $S$. Then, the left common multiple $\DD$ of $S$ is also the  right common multiple of $S$ and is a balanced element of $M$. Moreover  $(M,\Delta)$ is a Garside structure and  $\Delta$ is the \emph{minimal Garside element} of $M$.
    \end{prp}

    As discuted previously it is enough, for the existence of a left-ordering, to treat only the  case of irreducible  spherical type Artin--Tits groups, and so,  to consider the connected Coxeter graphs $\Gamma$  that appear in Figure~\ref{F:CoxDiag}, only.  One of the advantage, is that for each  Garside structure $(M(\Gamma), \GDD)$ it exists $p\in\NN^*$ such that $\GDD$ is equal to $\DD^p$.\\
    
 \begin{nota}\label{N: cadre}   
  For the rest of the paper, unless otherwise specified,  $\Gamma = (S,E,m)$ is assumed to be among the list given in Figure~\ref{F:CoxDiag}. In other words, $M(\Gamma)$ is an irreducible spherical type Artin-Tits monoid.  We write $M$ and $G$ for $M(\Gamma)$ and $G(\Gamma)$, respectively, and by $\Delta$ we denote the minimal Garside element of $M$. 
  \end{nota}
    \begin{df}
    \begin{itemize} 
    \item The \emph{Garside automorphism} $\Phi_\DD : G \to G$  is defined by $\Phi_\DD(g) = \DD \cdot g\cdot \DD\inv$. 
    \item For all~$g\in \SD{M}\DD$, we denote by $\alpha_{\DD}(g)$ the unique element of $M$ satisfying $\Delta=\alpha_\Delta(g)\cdot g$. 
    \end{itemize}
    \end{df}
    If the context is sufficiently clear we will write $\Phi$ and $\alpha$ instead of $\Phi_\Delta$ and $\alpha_\Delta$. 
    It is well known that~$\Phi_\DD(M) = M$. So, for any $g$ in $M$, the equality $\Phi_\DD(g)\cdot \DD = \DD \cdot g$ holds in $M$. 
    As a consequence, for any~$h\in M$ we have $\DD \cdot h=  \Phi_{\DD} (h) \cdot \DD$. 
    Therefore, an element $g$ of $M$ is $\DD$-unmovable if and only if it is not right divisible by $\DD$. 
    Moreover, for any $s\in S$, the element $\Phi_\DD(s)$ belongs to $S$. 
    So we can define $\Phi_\DD$ on $S$-words.   All along this paper we will use the following results on minimal Garside element of $\DD$.
    
    \begin{prp}
    \label{P:Garside}
    The following hold
    
    \begin{enumerate}
    \item $\DD$ is square free, \emph{i.e.}, $\DD\not=g s^2 h$ for any $s\in S$ and $g,h\in M$;
    \item $\Delta^2$ is central in $M$ and so $\Phi$ is an involution;
    \item For $g\in \SD{M}\DD$, we have $\Delta=  \alpha(g)\cdot g=  \Phi(\alpha(g))\cdot \Phi(g) =  \Phi(g)\cdot \alpha(g) = g\cdot \Phi(\alpha(g))$.
    \end{enumerate}
    \end{prp}
    
    \begin{proof}
    $(i)$ is Proposition 5.7 of \cite{BrS}. $(ii)$ is Lemma 5.2 of \cite{BrS}.
    Let us prove~$(iii)$. We have 
    \[
    \Phi(g) = \Delta \cdot g \cdot\Delta\inv = (\alpha(g)\cdot g) \cdot g\cdot  (\alpha(g)\cdot g)\inv = \alpha(g)\cdot g \cdot\alpha(g)\inv,
    \]
    and so $\Phi(g)\cdot \alpha(g) =\alpha(g)\cdot g = \Delta$. We conclude with $\Delta = \Phi(\Delta) = \Phi(\Phi(g))\cdot \Phi(\alpha(g)) = g\cdot \Phi(\alpha(g))$ together with $\Delta=\Phi(\Delta) =  \Phi(\alpha(g))\Phi(g)$.
    \end{proof}
    
       We do note recall here the value of $\DD$.  This will be done when needed.

 By the above result, $\Phi$ has either order $2$ or is the identity. The latter case occurs when $\DD$ is a central element of $M$.   
    
    \begin{prp}\label{P:DeltaCentral}
    	The minimal Garside element $\DD$ of $M$ is central if and only if the type of $\Gamma$ in one of the following list:
    	
    	\begin{itemize}
    		\item $B_r$;
    		\item $D_r$ and $r$ even ;
    		\item $E_7$, $E_8$, $F_4$, $H_3$, $H_4$;
    		\item $I_2(m)$ for $m$ even.
    	\end{itemize}
    	
    \end{prp}
    
    \begin{cor} 
    \label{C:PowerDelta}
    For $g\in \SD{M}\DD$ and $k\in \NN$, we have
    \[
    \DD^k = \Phi^{k-1}(\alpha(g))\cdot\ldots \cdot \Phi(\alpha(g))\cdot \alpha(g)\cdot g^k 
    \]
    \end{cor}
    
    \begin{proof}
    By induction on $k$.
    Case $k=0$ is immediate as left and right terms are trivial. The case $k= 1$ is also immediate as we have $\DD$ by $\alpha(g)\cdot g$ by defintion of $\alpha$.  
    Assume $k\geq 2$.
    We have $\DD^{k}=\DD\cdot \DD^{k-1}$, which by induction hypothesis is equal to 
    \[
    \DD\cdot \Phi^{k-2}(\alpha(g))\cdot\ldots\Phi(\alpha(g))\cdot \alpha(g) \cdot g^{k-1}.
    \]
    Using the relation $\DD\cdot y = \Phi(y) \cdot \DD$, we obtain 
    \[
    \DD^k= \Phi^{k-1}(\alpha(g))\cdot\ldots\Phi^2(\alpha(g))\cdot \Phi(\alpha(g)) \cdot \DD  \cdot g^{k-1}.
    \]
    Using the case $k = 1$, we conclude by replacing $\DD$ by $\alpha(g)\cdot g$ in the last term.
    \end{proof}
    
    \subsection{Complemented presentation}
    The presentation $\left<S\,|\,\mathcal{R}\right>^+$ of $M$ given in \eqref{E:PresentationArtinMonoid} is \emph{left complemented}, that is there exists a map $f_L:S\times S \to S^\ast$ satisfying
    \[
      \mathcal{R}=\left\{f_L(\ssi,\ssj)\ssi = f_L(\ssj,\ssi)\ssj\ |\ (\ssi, \ssj) \in S^2\right\}
    \]
    together with $f(\ssi,\ssi)=\varepsilon$ for all $\ssi\in S$.
    The map $f_L$ is called the \emph{left complement map} of the considered prensentation of $M$.
    Symmetricaly, the presentation~\eqref{E:PresentationArtinMonoid} is also \emph{right complemented} as it exists a \emph{right complement map} satisfying
    \[
     \mathcal{R}=\left\{\ssi\,f_R(\ssj,\ssi) = \ssj\,f_R(\ssi,\ssj)\ |\ (\ssi, \ssj) \in S^2\right\}.
    \]
    We refer to~\cite{DDGKM} for more details.
    An immediate verification gives 
    \[
    f_L(\ssi,\ssj) = \underbrace{\cdots\ssj\ssi\ssj}_{m(i,j)-1 \textrm{ terms}}\quad\text{and}\quad f_R(\ssj,\ssi) = \underbrace{\ssj\ssi\ssj\cdots}_{m(i,j)-1 \textrm{ terms}},
    \]
    since we have the following relations in $\mathcal{R}$
    \[
    f_L(\ssi,\ssj) \ssi = f_L(\ssj,\ssi) \ssj\quad\text{and}\quad
    \ssj f_R(\ssi,\ssj)= \ssi f_R(\ssj,\ssi).
    \]
    
    \begin{df}
    Let $w$ and $w'$ be two $S^{\pm}$-words. We say that
    
     \begin{itemize}
      \item   $w$ \emph{left reverses in one step} to $w'$, denoted by $w\curvearrowright_L^1 w'$, if we can obtain $w'$ from $w$ substituting a factor $xy^{-1}$ of $w$ by $f_L(x,y)^{-1}f_L(y,x)$.
    \item $w$ \emph{left reverses} to $w'$, denoted $w\curvearrowright_L w'$, if there exists a sequence $w=w_1,\ldots,w_\ell=w'$ of $S^{\pm}$-words satisfying $w_k\curvearrowright_L^1 w_{k+1}$, for all~$k\in\{1,\ldots,\ell-1\}$.
     \end{itemize}
    Symmetrically, we say that $w$ \emph{right reverses} to $w'$,  denoted $w\curvearrowright_R w'$, if $w'$ is obtained from $w$ by repeatedly replacing a factor $x^{-1}y$ of $w$ by $f_R(y,x)f_R(x,y)^{-1}$.
    \end{df}
    
    From the word equivalences 
    \[
    xy\inv \equiv f_L(x,y)\inv\,f_L(y,x)\quad\text{and}\quad x\inv y\equiv f_R(y,x)\,f_R(x,y)\inv,
    \]
    we obtain
    \[
    w \curvearrowright_L w' \Rightarrow w\equiv w'\quad \text{and}\quad  w \curvearrowright_R w' \Rightarrow w\equiv w'.
    \]
    
    Both left and right reversing processes can be illustrated by reversing diagrams made of $S$-labelled arrows.
    Let $w$ be a $S^{\pm}$-word. 
    Reading letters of $w$ from left to right, we draw a path of arrows :
    \begin{itemize}
    \item if we read a positive letter $x$, we draw a right-oriented arrow labelled $x$;
    \item if we read a negative letter $x^{-1}$, we draw a down-oriented arrow labelled $x$.
    \end{itemize}
    Note that a sequence of consecutive arrows with the same orientation may be depicted as a unique arrow labelled with the corresponding word.
    \begin{center}
        \begin{tikzpicture}[x=1.1cm,y=1.1cm]
        \draw[-latex] (0,0.5) -- (0,0) node [midway, left]{$a$};
        \draw[-latex] (0,0) -- (0.5,0) node [midway, below]{$b$};
        \draw[-latex] (0.5,0) -- (1,0) node [midway, below]{$c$};
        \draw[-latex] (1,1) -- (1,0.5) node [midway, right]{$e$};
        \draw[-latex] (1,0.5) -- (1,0) node [midway, right]{$d$};
        \draw (2,0.5) node{$=$};
        \begin{scope}[shift={(3,0)}]
        \draw[-latex] (0,1) -- (0,0) node [midway, left]{$a$};
        \draw[-latex] (0,0) -- (1,0) node [midway, below]{$bc$};
        \draw[-latex] (1,1) -- (1,0) node [midway, right]{$ed$};
        \end{scope}
        \end{tikzpicture}
    \end{center}
    The reversing process consists on the completion of the diagram using the following rules:
    
    \begin{center}
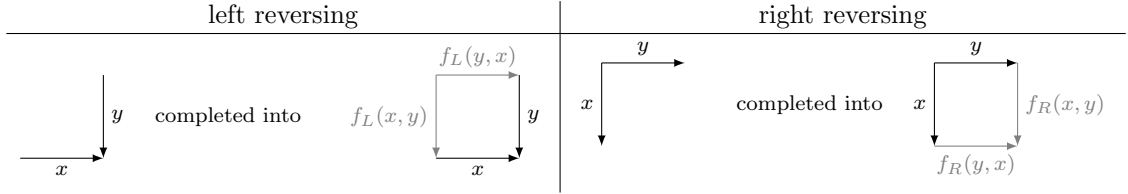
    \begin{tabular}{c|c}
    left reversing & right reversing \\
    \hline
    \begin{tikzpicture}[x=1.1cm,y=1.1cm]
    \footnotesize
    \draw[-latex] (0,0) -- (1,0) node [midway, below]{$x$};
    \draw[-latex] (1,1) -- (1,0) node [midway, right]{$y$}; 
    \draw(2.5,0.5) node{completed into};
    \begin{scope}[shift={((5,0)}]
    \draw[-latex] (0,0) -- (1,0) node [midway, below]{$x$};
    \draw[-latex] (1,1) -- (1,0) node [midway, right]{$y$}; 
    \draw[gray, -latex] (0,1) -- (0,0) node [midway, left]{$f_L(x,y)$};
    \draw[gray, -latex] (0,1) -- (1,1) node [midway, above]{$f_L(y,x)$};
    \end{scope}
    \end{tikzpicture}
    &
    \begin{tikzpicture}[x=1.1cm,y=1.1cm]
    \footnotesize
    \draw[-latex] (0,1) -- (0,0) node [midway, left]{$x$};
    \draw[-latex] (0,1) -- (1,1) node [midway, above]{$y$};
    \draw(2.5,0.5) node{completed into};
    \begin{scope}[shift={((4,0)}]
    \draw[-latex] (0,1) -- (0,0) node [midway, left]{$x$};
    \draw[-latex] (0,1) -- (1,1) node [midway, above]{$y$};
    \draw[gray, -latex] (0,0) -- (1,0) node [midway, below]{$f_R(y,x)$};
    \draw[gray, -latex] (1,1) -- (1,0) node [midway, right]{$f_R(x,y)$};
    \end{scope}
    \end{tikzpicture}
    \end{tabular}
    \end{center}
    It is possible to produce empty sequence of arrows during a reversing process.
    This is the case, for example, whenever $x=y$. In this case we use undirected $\varepsilon$-labelled edge in diagrams.
    
    \begin{ex}
    \label{E:Reversing1}
    Consider the word $w=\ss1\ssinv2\ss3\ssinv1\ss1$ in the braid group $G(A_3)$.
    
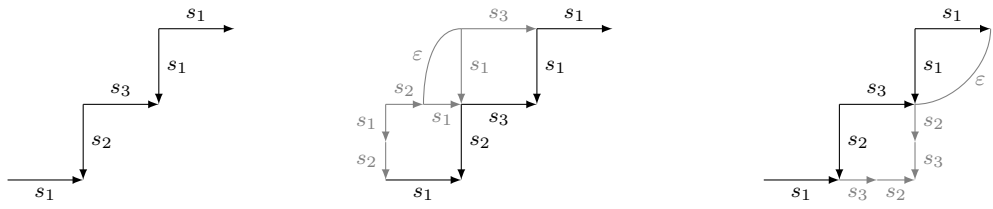
\begin{center}
    \begin{tikzpicture}
    \footnotesize
    \draw[-latex] (0,0) -- (1,0) node [midway, below]{$\ss1$};
    \draw[-latex] (1,1) -- (1,0) node [midway, right]{$\ss2$};
    \draw[-latex] (1,1) -- (2,1) node [midway, above]{$\ss3$};
    \draw[-latex] (2,2) -- (2,1) node [midway, right]{$\ss1$};
    \draw[-latex] (2,2) -- (3,2) node [midway,  above]{$\ss1$};
    \draw(1,-1) node {\textbf{a.} path associated with $w$};
    \begin{scope}[shift={(5,0)}]
    \draw[-latex] (0,0) -- (1,0) node [midway, below]{$\ss1$};
    \draw[-latex] (1,1) -- (1,0) node [midway, right]{$\ss2$};
    \draw[-latex] (1,1) -- (2,1) node [midway, below]{$\ss3$};
    \draw[-latex] (2,2) -- (2,1) node [midway, right]{$\ss1$};
    \draw[-latex] (2,2) -- (3,2) node [midway,  above]{$\ss1$};
    \draw[gray,-latex] (0,0.5) -- (0,0) node [midway, left]{$\ss2$};
    \draw[gray,-latex] (0,1) -- (0,0.5) node [midway, left]{$\ss1$};
    \draw[gray,-latex] (0,1) -- (0.5,1) node [midway, above]{$\ss2$};
    \draw[gray,-latex] (0.5,1) -- (1,1) node [midway, below]{$\ss1$};
    \draw[gray,-latex] (1,2) -- (1,1) node [midway, right]{$\ss1$};
    \draw[gray] (1,2) .. controls (0.6,2) and (0.5,1.4) ..  (0.5,1) node [midway, left]{$\varepsilon$};
    \draw[gray,-latex] (1,2) -- (2,2) node [midway, above]{$\ss3$};
    \draw(1,-1) node {\textbf{b.} left reversing diagram of $w$};
    \end{scope}
    \begin{scope}[shift={(10,0)}]
    \draw[-latex] (0,0) -- (1,0) node [midway, below]{$\ss1$};
    \draw[-latex] (1,1) -- (1,0) node [midway, right]{$\ss2$};
    \draw[-latex] (1,1) -- (2,1) node [midway, above]{$\ss3$};
    \draw[-latex] (2,2) -- (2,1) node [midway, right]{$\ss1$};
    \draw[-latex] (2,2) -- (3,2) node [midway, above]{$\ss1$};
    \draw[gray,-latex] (1,0) -- (1.5,0) node [midway, below]{$\ss3$};
    \draw[gray,-latex] (1.5,0) -- (2,0) node [midway, below]{$\ss2$};
    \draw[gray,-latex] (2,1) -- (2,0.5) node [midway, right]{$\ss2$};
    \draw[gray,-latex] (2,0.5) -- (2,0) node [midway, right]{$\ss3$};
    \draw[gray] (2,1) .. controls (2.5,1) and (3,1.5) .. (3,2) node [midway, right]{$\varepsilon$};
    
    \draw(1,-1) node {\textbf{c.} right reversing diagram of $w$};
    \end{scope}
    \end{tikzpicture}
    \end{center}
    The word $w$ left reverses in $\ssinv2\ssinv1\ss2\ss3\ss1$ while it right reverses in $\ss1\ss3\ss2\ssinv3\ssinv2$.
    \end{ex}

    \begin{df}
    For a $S^\pm$-word~ $w$, 
         \begin{itemize}
         \item by $\denL(w)$ and $\numL(w)$  we denote the $S$-words,  defined by  $w\curvearrowright_L \denL(w)\inv\numL(w)$ 
      \item  by  $\denR(w)$ and $\numR(w)$  we denote the  $S$-words, defined by  $w\curvearrowright_R \numR(w) \denR(w)\inv$.
    \end{itemize}
    \end{df}
    
    From \cite{DeP1999} we know that the pair $(\denL(w),\numL(w))$ of $S$-word  (and similarly the pair $(\denL(w),\numL(w))$)  exist and is unique for a given $w$ because $M$ is a Garside monoid.

    \begin{ex}
     Reconsidering the word $w$ defined in Example \ref{E:Reversing1}, we obtain :
    \begin{align*}
    \denL(w)  = \ss1\ss2,\  \numL(w) = \ss2\ss3\ss1\quad\text{and}\quad\numR(w)=\ss1\ss3\ss2, \ \denR(w)=\ss2\ss3.
    \end{align*}
    
    \end{ex}
    \begin{prp}[\cite{DeP1999}]
    \label{P:ReversingDivisibility}
    An element $g$ of $M$ is left-divisible, \resp right-divisible,  by $\ss \in S$ iff for any word representative $w$ of $g$ the word $\denR(s\inv w)$, \resp $\denL(w s\inv)$, is empty.
    \end{prp}
    
    \begin{prp}[\cite{DeP1999}, Lemma 4.3]\label{P:ReversingLcm}  Let $u$ and $v$ be two words.
    
    -- the words $u\cdot\numR(u\inv v)$ and $v\cdot\denR(u\inv v)$ are two representatives of $\overline{u}\vee_R\overline{v}$,
    
    -- the words $\numL(u v\inv)\cdot u$ and $\denL(u v\inv)\cdot v$ are two representatives of $\overline{u}\vee_L\overline{v}$.
    \end{prp}
    
    This is useful to note for the sequel that  for $\ssi,\ssj$ distinct in $S$, the left-lcm $\ssi\vee_L\ssj$  and the right-lcm $\ssi\vee_R\ssj$ are equal  and possess only two representative $S$-words, that can be written alternatively as $f_L(\ssi,\ssj)\,\ssi$  and $f_L(\ssj,\ssi)\,\ssj$, or $\ssj\,f_R(\ssi,\ssj)$ and $\ssi\,f_R(\ssj,\ssi)$.  We have $f_L(\ssi,\ssj)\,\ssi = \ssj\,f_R(\ssi,\ssj)$ when $m(i,j)$ is odd, and  $f_L(\ssi,\ssj)\,\ssi = \ssi\,f_R(\ssj,\ssi)$ when $m(i,j)$ is even.

    \subsection{Chains}
    
    Following  \cite{BrS}, we introduced left and right chains in our specific context.
    However, our definition not follows rigorously~\cite{BrS}. 
    The main difference is that our chains are made of words and not of elements. 
    This choice will allow us to use the power of the reversing tool and shall clarify the discussion.
    Here we focus on the case of spherical type  Artin--Tits monoid. However our definition may be introduce in a more general context.
    
    \begin{df}
    \label{D:Chain}
    Let $(\ssi,\ssj)$ be a pair of, possibly equal, elements of $S$, and $u$ a non-empty $S$-word,
    \begin{enumerate}
    \item $u$ is a left $(\ssj,\ssi)$-chain if $\denL(u \ssinv\ii)$ is not empty and ends with $\ssj$, \emph{ie}, $u \ssinv\ii \curvearrowright_L \ssinv\jj \ldots$,
    \item $u$ is a right $(\ssi,\ssj)$-chain if $\denR(\ssinv\ii u)$ is not empty and begins with $\ssj$, \emph{ie}, $\ssinv\ii u\curvearrowright_R \ldots \ssinv\jj$.
    \end{enumerate}
    \end{df}
    
    It is immediate that the property to be a chain can be checked on reversing diagram.
    A non-empty $S$-word is a left $(\ssj,\ssi)$-chain if and only if  the right reversing diagram depicted in \textbf{a} holds.
    It is a right $(\ssi,\ssj)$-chain if and only if the left reversing diagram depicted in \textbf{b} holds.
    
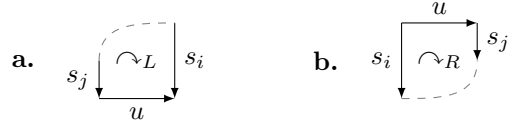
\begin{center}
    
    \begin{tikzpicture}
    \draw[latex-] (1,0) -- (1,1) node[midway, right]{$s_i$};
    \draw[-latex] (0,0) -- (1,0) node[midway, below]{$u$};
    \draw(0.5,0.5) node{$\curvearrowright_L$};
    \draw[latex-] (0,0) -- (0,0.5) node[midway, left]{$s_j$};
    \draw[gray, dashed](0,0.5) .. controls (0,1) and (0.5,1) .. (1,1);
    \draw(-1, 0.5) node {\textbf{a.}};
    \begin{scope}[shift={(4,0)}]
    \draw[latex-] (0,0) -- (0,1) node[midway, left]{$s_i$};
    \draw[-latex] (0,1) -- (1,1) node[midway, above]{$u$};
    \draw(0.5,0.5) node{$\curvearrowright_R$};
    \draw[latex-] (1,0.5) -- (1,1) node[midway, right]{$s_j$};
    \draw[gray, dashed](0,0) .. controls (0.5,0) and (1,0) .. (1,0.5);
    \draw(-1, 0.5) node {\textbf{b.}};
    \end{scope}
    \end{tikzpicture}
    \end{center}
    It will be common in the sequel to draw such chains using the following diagrams (on left, for left chain):
    
    
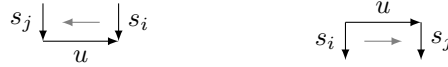
\begin{center}
    \begin{tikzpicture}
    \draw[-latex](0,0) -- (1,0) node[midway, below]{$u$};
    \draw[-latex](0,0.5) -- (0,0) node[midway, left]{$\ssj$};
    \draw[-latex](1,0.5) -- (1,0) node[midway, right]{$\ssi$};
    \draw[gray,latex-](0.25,0.25) -- (0.75,0.25);
    
    \begin{scope}[shift={(4,0.25)}]
    \draw[-latex](0,0) -- (1,0) node[midway,above]{$u$};
    \draw[-latex](0,0) -- (0,-0.5) node[midway, left]{$\ssi$};
    \draw[-latex](1,0) -- (1,-0.5) node[midway, right]{$\ssj$};
    \draw[gray,-latex](0.25,-0.25) -- (0.75,-0.25);
    \end{scope}
    \end{tikzpicture}
    \end{center}
    It make sense to depict the fact that $s_i$ left or right divides the element representing by $u$ with:
    
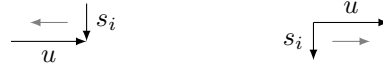
\begin{center}
    \begin{tikzpicture}
    \draw[-latex](0,0) -- (1,0) node[midway, below]{$u$};
    \draw[-latex](1,0.5) -- (1,0) node[midway, right]{$\ssi$};
    \draw[gray,latex-](0.25,0.25) -- (0.75,0.25);
    
    \begin{scope}[shift={(4,0.25)}]
    \draw[-latex](0,0) -- (1,0) node[midway, above]{$u$};
    \draw[-latex](0,0) -- (0,-0.5) node[midway, left]{$\ssi$};
    \draw[gray,-latex](0.25,-0.25) -- (0.75,-0.25);
    \end{scope}
    \end{tikzpicture}
    \end{center}
    which is equivalent to saying that the word $\denL(u \ssinv\ii)$, \emph{resp.}, $\denL(\ssinv\ii u)$ are empty.\\
    
    A crucial example of chain is given in the following lemma. 
    
    \begin{lm} \label{L:LcmToChain}
    Let $\ssk,\ss\ell$ be distinct in $S$.
    \begin{enumerate}
    \item Let $u$ be a proper prefix of the $S$-word $\ss\ell\,f_R(\ssk,\ss\ell)$. Write $\{i,j\} = \{k,\ell\}$ so that $\ssj$ is the last letter of $u$. Then  $u$ is a right $(\ssk,\ssi)$-chain.
    \item Let $u$ be a proper suffix of $f_L(\ss\ell,\ssk)\,\ss\ell$. Write $\{i,j\} = \{k,\ell\}$ so that $\ssj$ is the first letter of $u$. Then $u$ is a left $(\ssi,\ssk)$-chain.
    \end{enumerate}
    \end{lm}
    
    \begin{proof} By symmetry, it is enough to prove (i).
    By Proposition~\ref{P:ReversingLcm}, the $S$-word $\ss\ell\,f_R(\ssk,\ss\ell)$, which is equal to $u\,\denR(\ssk\inv \,u)$, is a representative of $\overline{u}\vee_R\ssk = \ss\ell\vee_R\ssk=\ssi\vee_R\ssj$ and so $\denR(\ssk\inv u)$ starts with the letter~$\ssi$.
     \end{proof}
    
    \begin{lm}\label{L:ChainComposition}
    If  $u$ is a left $(\ssi,\ssj)$-chain and $v$ is a left $(\ssj,\ssk)$-chain then $uv$ is a left $(\ssi,\ssk)$-chain. A similar result holds for right chains. 
    \end{lm}
    
    \begin{proof} It exist two $S^{\pm}$-words $u'$ and $v'$ such that $\ssi\inv\, u\curvearrowright_R u'\, \ssj\inv$ and $\ssj\inv\, v\curvearrowright_R v' \,\ssk\inv$.
   This implies $\ssi\inv\, u\, v\curvearrowright_R u'\,v' \, \ssk\inv$. So, $uv$ is a left $(\ssi,\ssk)$-chain.
    \end{proof}
    
    The previous lemma allows the following diagram compactification:
    
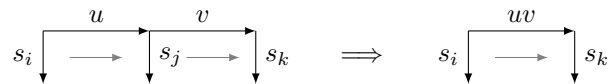
\[
    \begin{tikzpicture}[x=40pt,y=40pt]
    \draw[-latex](0,0) -- (1,0) node[midway, above]{$u$};
    \draw[-latex](0,0) -- (0,-0.5) node[midway, left]{$\ssi$};
    \draw[-latex](1,0) -- (1,-0.5) node[midway, right]{$\ssj$};
    \draw[-latex](1,0) -- (2,0) node[midway, above]{$v$};
    \draw[-latex](2,0) -- (2,-0.5) node[midway, right]{$\ssk$};
    \draw[gray,-latex](0.25,-0.25) -- (0.75,-0.25);
    \draw[gray,-latex](1.35,-0.25) -- (1.85,-0.25);
    \draw(3,-0.25) node{$\Longrightarrow$};
    \draw[-latex](4,0) -- (5,0) node[midway, above]{$uv$};
    \draw[-latex](4,0) -- (4,-0.5) node[midway, left]{$\ssi$};
    \draw[-latex](5,0) -- (5,-0.5) node[midway, right]{$\ssk$};
    \draw[gray,-latex](4.25,-0.25) -- (4.75,-0.25);
    \end{tikzpicture}
    \]

    The word homomorphism $\Phi$ is well compatible with the notion of a chain.
    
    \begin{lm}
    \label{L:ChainPhi}
    For $s,t$ in $S$, a $S$-word $u$ is a left $(t,s)$ chain if and only if $\Phi(u)$ is a left $(\Phi(t),\Phi(s))$-chain.
    \end{lm}
     
     \begin{proof}
     For any $s_i,s_j\in S$, we have $f_L(\Phi(s_i),\Phi(s_j)) = \Phi(f_L(s_i,s_j))$.
     Hence, for any $S$-word $w$, the left reversing diagram of word $\Phi(w)$ is obtain apply $\Phi$ to the labels of the left reversing diagram of $w$. 
     In particular, if $u$ is a left $(t,s)$-chain then $\Phi(u)$ is a left $(\Phi(t),\Phi(s))$-chain.
     As $\Phi$ is an involution, the converse holds.
     \end{proof}
    
     \begin{df}
     For a $S$-word $u$, the \emph{support} of $u$, denoted by $\supp(u)$, is the set of letters occurring in~$u$.
     \end{df} 
     
     Since relations of~\eqref{E:PresentationArtinMonoid} preserves support of words, we can define the support of an element $g$ of $M$ as the support of any of its representative word over $S$. 
     
    \begin{prp}
    \label{P:SuppChain} A non-empty word $S$-word $u$ whose support does not contain a letter $s\in S$ is both a left and a right $(s,s)$-chain.
    \end{prp}
    
    \begin{proof}
    
    We use an induction on the length of $u$.
    Assume $u$ is a letter $t\not= s$.
    We have the following left reversing diagram
    \[
    \begin{tikzpicture}
    \draw[-latex] (0,0) -- (1,0) node [midway, below]{$u=t$};
    \draw[-latex] (1,1) -- (1,0) node [midway, right]{$s$}; 
    \draw[-latex] (0,1) -- (0,0) node [midway, left]{$f_L(t,s)$};
    \draw[ -latex] (0,1) -- (1,1) node [midway, above]{$f_L(s,t)$};
    \draw(0.5,0.5) node{$\curvearrowright_L$};
    \draw (2,0.5) node{$=$};
    \begin{scope}[shift={(3,0)}]
    \draw[-latex] (0,0) -- (1,0) node [midway, below]{$u=t$};
    \draw[-latex] (1,1) -- (1,0) node [midway, right]{$s$}; 
    \draw[dotted] (0,1) -- (0,0.5);
    \draw[-latex] (0,0.5) -- (0,0) node [midway, left]{$s$};
    \draw[dotted] (0,1) -- (0.5,1);
    \draw[-latex] (0.5,1) -- (1,1) node[midway ,above]{$t$};
    \draw(0.5,0.5) node{$\curvearrowright_L$};
    \end{scope}
    \end{tikzpicture}
    \]
    implying that $u$ is a left $(s,s)$-chain.
    Assume $u$ has length $\ell \geq 2$. We put $u=v\,x$.
    We have  $x\not = s$ and $v$ does not contain $s$ since it is the case for $u$.
    By induction, $x$ and $v$ are two $(s,s)$-left chains and so is~$u$ by Lemma~\ref{L:ChainComposition}.
    A similar argument shows that $u$ is a $(s,s)$-right chain.
    \end{proof}

    As illustrated by the following result, chains are strongly related to the divisibility notion.
    
    \begin{prp}
    \label{P:ChainDiv}
    Let $u$ and $v$ be two $S$-words and $\ssi, \ssj$ be in $S$.
    
    \begin{enumerate}
    \item $u$ is a right $(\ssi,\cdot)$-chain if and only if $\ssi$ does not left-divide $\overline{u}$. 
    \item if $u$ is a right $(\ssi,\ssj)$-chain and $s_i$ left-divides $\overline{uv}$ then $\ssj$ left-divides $\overline{v}$.
    \item $u$ is a left $(\cdot,\ssi)$-chain if and only if $\ssi$ does not right-divide $\overline{u}$.
    \item if $u$ is a left $(\ssj,\ssi)$-chain and $\ssi$ right-divides $\overline{vu}$, then $\ssj$ right-divides $\overline{v}$.
    \end{enumerate} 
    \end{prp}
    
    \begin{proof}
    $(i)$ and $(iii)$ are immediate consequences of definition of chains and of Proposition~\ref{P:ReversingDivisibility}.
    For $(ii)$, as $\ssi$ left-divides $\overline{uv}$, Proposition~\ref{P:ReversingDivisibility} implies $\denR(\ssinv\ii uv)=\varepsilon$ :
    \begin{center}
    \begin{tikzpicture}[x=40pt,y=40pt]
    \draw[-latex](0,0) -- (1,0) node[midway, above]{$u\,v$};
    \draw[-latex](0,0) -- (0,-0.5) node[midway, left]{$\ssi$};
    \draw[gray,-latex](0.25,-0.25) -- (0.75,-0.25);
    \draw (1.5,-0.25) node{$=$};
    \begin{scope}[shift={(2.5,0)}]
    \draw[-latex](0,0) -- (1,0) node[midway, above]{$u$};
    \draw[-latex](0,0) -- (0,-0.5) node[midway, left]{$\ssi$};
    \draw[-latex](1,0) -- (1,-0.5) node[midway, right]{$\ssj$};
    \draw[-latex](1,0) -- (2,0) node[midway, above]{$v$};
    \draw[gray,-latex](0.25,-0.25) -- (0.75,-0.25);
    \draw[gray,-latex](1.35,-0.25) -- (1.85,-0.25);
    \end{scope}
    \end{tikzpicture}
    \end{center}
    and so we obtain $\denR(\ssj\inv v)=\varepsilon$ implying $t$ left divides $\overline{v}$.
    Point $(iv)$ is similar.
    \end{proof}
    
    \begin{cor}
    \label{C:Subchain}
    Let $u$ be a left $(\ssj,\ssi)$-chain and $g$ be a right divisor of $\overline{u}$.
    Any representative $v$ of $g$ is a left $(\ssk,\ssi)$-chain for some $\ssk\in S$. The letter $\ssk$ depends on the choice of $v$.
    \end{cor}
    
    \begin{proof}
    By $(iii)$ of Proposition~\ref{P:ChainDiv}, $\ssi$ is not a right divisor of $\overline{u}$ and so of $g$ in particular.
    Always by Proposition~\ref{P:ChainDiv}, it exists $\ssk\in S$ such that $v$ is a left $(\ssk,\ssi)$-chain.
    \end{proof}
    
    \begin{cor}
     \label{C:DivGarsideChain} Let $g$ be a divisor of $\DD$ and $w$ be a representative of $\alpha(g)$.
    
     \begin{enumerate}
     \item If $s\in S$ left divides $g$, then $w$ is a left $(t,s)$-chain with $t$ a right divisor of $\Phi(g)$.
     \item If $s\in S$ right divides $g$, then~$\Phi(w)$ is a right $(s,t)$-chain with $t$ a left divisor of $\alpha(g)$.
    \end{enumerate}
     \end{cor}
     
     \begin{proof}
     Let us prove $(i)$. 
    From Proposition~\ref{P:Garside} $(iv)$, we have $\DD = \alpha(g) \cdot g = \Phi(g)\cdot \alpha(g)$.
    Let $s\in S$ be a left divisor of~$g$. 
    Since $\DD$ is square free, the term $\alpha(g)$ is not right divisible by $s$.
    According to Proposition~\ref{P:ChainDiv} $(iii)$, the word $w$ is then a left $(t,s)$-chain for some $t\in S$.
    By Proposition~\ref{P:ChainDiv} $(iv)$, the relation $\DD= \Phi(g)\cdot \alpha(g)$ and the fact that $s$ is a right divisor of $\DD$ imply that $\Phi(g)$ is right divisible by $t$. 
    We establish $(ii)$ in the similar way.
      \end{proof}
      
    As we will see in Proposition~\ref{P:ChainRev}, in the particular context of spherical type  Artin-Tits monoids, a word is a left $(t,s)$-chain if and only it is a right $(s,t)$-chain. We then obtain a connection between left and right divisibility notions.
    This property is a consequence of the very particular form of least common multiples of elements of $S$ as already seen in Lemma~\ref{L:LcmToChain}. This leads to the following definition: 
    
    \begin{df}
    A left $(\cdot,\ssi)$-chain $u$ is said to be \emph{elementary} if it exists $\ssj\in S$ such that 
    $u$ is a  proper suffix of $f_L(\ssj,\ssi)\, \ssj $. A right $(\ssi,\cdot)$-chain $u$ is said to be \emph{elementary} if it exists $\ssj\in S$ such that $u$ is a  proper prefix of $\ssj\,f_R(\ssi,\ssj)$.
    \end{df}

    \begin{lm}
    \label{L:ExistElemChain}
    Consider a left $(\ssj,\ssi)$-chain $u$ and its last letter $s_\ell$.
    There exists a unique maximal proper suffix $v$ of $f_L(\ss\ell,\ssi)\,\ss\ell$  and a left $(\ssj,\ssk)$-chain $w$ so that $v$ is an elementary left $(\ssk,\ssi)$-chain and~$u\equiv wv$.
    A similar results holds for right chains.
    \end{lm}
    
    \begin{proof}
    Consider $u$ and $\ss\ell$ as in the statement. 
    Let $v'$ be the longest word $\ldots \ssi\ss\ell\,\ssi$ such that  $\overline{v'\ss\ell}$ right divides $\overline{u}$.
    Since $\ssi$ is not a right divisor of $\overline{u}$, the word $v'$ must have length at most $m_{\ii,\ell}-2$, and $v'\ss\ell$ is an elementary left $(\ssk,\ssi)$-chain with $\ssk\in\{\ssi,\ss\ell\}$ by Lemma~\ref{L:LcmToChain}.
    Define $u'$ and $v''$ so that $u=u' s_\ell$ and $v''\ssk v' = f_L(\ss\ell,\ssi)$.
    So $v''\ssk v'\ss\ell$ represents~$\ssi\vee_R\ss\ell$.
    As $\overline{v'}$ right divides $\overline{u'}$, there exists a word $w$ such that $u'{v'}^{-1}$ left reverses into $w$. By Proposition~\ref{P:ChainDiv}, the choice of $u$, imposes that $w$ has to be a left $(\ss{k'},\ssk)$-chain for some $\ss{k'}$.
    There exists $S$-words $w_1,w_2,w_3$ and $w_4$ such that we obtain the following left reversing diagram:
    
    
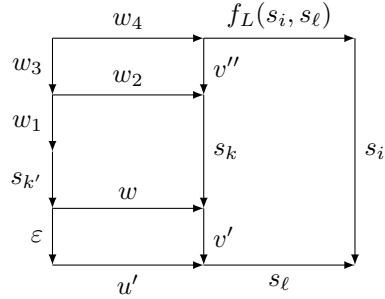
\[
    \begin{tikzpicture}
    \draw[-latex](0,0) -- (2,0) node[midway, below]{$u'$};
    \draw[-latex](2,0) -- (4,0) node[midway, below]{$\ss\ell$};
    \draw[-latex](2,3) -- (2,2.25) node[midway, right]{$v''$};
    \draw[-latex](2,2.25) -- (2,0.75) node[midway, right]{$\ssk$};
    \draw[-latex](0,2.25) -- (0,1.5) node[midway, left]{$w_1$};
    \draw[-latex](0,1.5) -- (0,0.75) node[midway, left]{$\ss{k'}$};
    \draw[-latex](2,0.75) -- (2,0) node[midway, right]{$v'$};
    \draw[-latex](4,3) -- (4,0) node[midway, right]{$\ssi$};
    \draw[-latex](0,0.75) -- (0,0) node[midway, left]{$\varepsilon$};
    \draw[-latex](0,3) -- (0,2.25) node[midway, left]{$w_3$};
    \draw[-latex](0,0.75) -- (2,0.75) node[midway, above]{$w$};
    \draw[-latex](0,3) -- (2,3) node[midway, above]{$w_4$};
    \draw[-latex](0,2.25) -- (2,2.25) node[midway, above]{$w_2$};
    \draw[-latex](2,3) -- (4,3) node[midway, above]{$f_L(\ssi,\ss\ell)$};
    \end{tikzpicture}
    \]
    We have then obtain that $u'\ss\ell\ssi\inv$ left reverses into $(w_3\,w_1\,\ss{k'})\inv (w_4\,f_L(\ssi,\ss\ell))$.
    Since $u=u'\,\ss\ell$ is a left $(\ssj,\ssi)$-chain we have $k' = j$ and so $w$ is a left $(\ssj,\ssk)$-chain.
    \end{proof} 
    
    \begin{lm}
    \label{L:ElemChainRev}
    A $S$-word is an elementary left $(\ssk,\ssi)$-chain if and only if it is an elementary right $(\ssk,\ssi)$-chain.
    \end{lm}
    
    \begin{proof}
    Assume $u$ is an elementary left $(\ssk,\ssi)$-chain and let $j,\ell$ be as in Lemma~\ref{L:LcmToChain}(i).
    By assumption $u$ is a proper prefix of $\ss\ell\,f_R(\ssk,\ss\ell)$.
    But this is also a proper suffix of $f_L(\ssj,\ssi)\,\ssj$ that begins with $s_\ell$.
    So by Lemma~\ref{L:LcmToChain}(ii), this is an elementary right $(\ssk,\ssi)$-chain.
    We conclude the proof by the left/right symmetry.
    \end{proof}
    
    \begin{prp}\label{P:ChainRev}
    A $S$-word is a left $(\ssj,\ssi)$-chain if and only if it is equivalent to a right $(\ssj,\ssi)$-chain.
    \end{prp}
    
    \begin{proof}
    Let $u$ be a left $(\ssj,\ssi)$-chain and let us prove that $u$ is equivalent to a right $(\ssj,\ssi)$-chain. 
    If $u$ is an (possible empty) elementary left $(\ssj,\ssi)$-chain, we conclude with Lemma~\ref{L:ElemChainRev}.
    Assume now that $u$ is not an elementary left $(\ssj,\ssi)$-chain. 
    We use an induction on the length of~$u$.
    Let $v$ and $w$ be the two $S$-words given by Lemma~\ref{L:ExistElemChain}: $v$ is a maximal elementary left $(\ssk,\ssi)$-chain and $w$ is a left $(\ssj,\ssk)$-chain such that $u\equiv w\,v$. 
    By the induction hypothesis, the word~$w$ is equivalent to a right $(\ssj,\ssk)$-chain $w'$. 
    On the other hand, Lemma~\ref{L:ElemChainRev} implies that $v$ is a right $(\ssk,\ssi)$-chain. 
    Therefore, by Lemma~\ref{L:ChainComposition}, $w'v$ is a right $(\ssj,\ssi)$-chain. 
    We conclude with $u\equiv w\,v\equiv w'\,v$.
    Thanks to left/right symmetry the converse is immediate.
    \end{proof}

    \section{Alternating normal forms in Artin monoids}
    
\label{S:alter}    
    
Recall that $\Gamma$, $S$, $M$, $G$ and $\Delta$ are defined once for all in the previous section (see Notation~\ref{N: cadre}). 
    
    \subsection{covering and normal forms}
    For a subset $T$ of $S$, the full subgraph~$\Gamma(T)$ of~$\Gamma$ generated by $T$ is also a Coxeter graph whom label map is the restriction of the map~$m$.  The \emph{parabolic} submonoid of $M$ generated by $T$ is denoted by~$M_T$.  The \emph{parabolic} subgroup of $G$ generated by $T$ is denoted by~$G_T$ Clearly,~$M_T$ is the Artin--Tits monoid $M(\Gamma(T))$ and~$G_T$ is the Artin--Tits group $G(\Gamma(T))$.
Since $M$ is of spherical type, so are $M_{T}$ and $M_{T}$ for any subset $T$ of $S$.  
\begin{df}   
    A \emph{covering} of $\Gamma$, is a couple $(S_2,S_1)$ satisfying~$S = S_1 \cup S_2$.
    It is said to be proper if $S_1$ and $S_2$ are proper subsets of $S$.\end{df}

Throughout Section~\ref{S:alter}, we assume $ (S_2, S_1)$ be a proper covering of $S$. In the next sections an additional assumption will be needed, namely that the covering is irreducible (see Definition~\ref{D:irred et stab}),  in order to prove our results.  But this restriction is not useful here.  Recall the sets~$\SRD{M}{g}$ and $\SLD{M}{g}$, for $g$ in $M$, are defined in Section~\ref{S:Garside}.   For a subset $X$ of $M$ and $g\in M$ we define :
    \[
    \SRD{X}{g} = \SRD{M}{g}\cap X\quad\text{and}\quad \SLD{X}{g} = \SLD{M}{g}\cap X. 
    \]
    
    Note that this notation is coherent : if $g$ belongs to $M_X$ for some $X\subseteq S$, then $R(g,M_X)$ is the set of right divisors of $g$ in $M_X$ ; the equivalent result holds for $L(g,M_X)$.
     The existence of the alternating decomposition on $M$ rests on the following result:
     
    \begin{prp}\label{P:ExistenceTail}  
    \cite[Lemma~1.6]{Deh2008}  Let $T$ be a subset of $S$.
    For any element $g$ in $M$, the set $\SRD{M_T}{g}$ admits a unique maximal element for the right-divisibility~$\preceq_R$.
    \end{prp}

    The maximal element in the above proposition is called the right~$T$-tail of $g$. We denote it by~$\tail{T}g$. 
   One can also consider left~$T$-tails. But, following~\cite{Deh2008} and \cite{ArP2019}, we work only with right-tails.    So, we write \emph{$T$-tail} for \emph{right~$T$-tail} in the sequel. 
    
    Note that in \cite{ArP2019}, the authors  introduced the notion of $N$-tail where $N$ is a parabolic submonoid of $M$. Since for each parabolic submonoid $N$ of $M$ it exists a unique subset $T$ of $S$ satisfying $N=M_T$, both notions coincide.
    \begin{lm}
    \label{L:Tail}
    Let $T$ be a subset of $S$. For $g\in M$ and $h\in M_T$, we have $\tail{T}{gh} = \tail{T}{g} \cdot h$.
    \end{lm}
    
    For a positive integer $i$, we denote by $\rem\ii$ the unique integer in $\{1,2\}$ satisfying $i\equiv \rem\ii \mod 2$.
    
    \begin{df} \label{D:AlternatingDecomposition}
    Let $g$ be an element of $M$.
    A sequence $\ggn$ of elements of $M$ is a \emph{$(S_2,S_1)$-alternating decomposition} of $g$ if
    \begin{enumerate}
    \item $g = g_n\cdot\ldots\cdot g_1$,
    \item $g_i$ belongs to $M_{S_{\rem\ii}}$.
    \end{enumerate}
    We say that a $(S_2,S_1)$-alternating decomposition of $g$ is the \emph{$(S_2,S_1)$-alternating normal form} of $g$ if we have in addition 
    \begin{enumerate}
      \setcounter{enumi}{2}
    \item  $g_n\neq 1$ except for $g=1$. In the latter case, $n = 1$.
    \item $g_i = \tail{S_{\rem\ii}}{g_n\cdots g_i}$ for all $i\in\{1,\ldots,n\}$,
    \end{enumerate}
    \end{df}
    
     \begin{prp} \cite[Proposition 1.16]{Deh2008} \label{P:AlternatingNormalForm} Every element of $M$ admits a unique $(S_2,S_1)$-alternating normal form.
     \end{prp}
     
    Whenever the covering $(S_2,S_1)$ is fixed, we denote by $\ND{g}$ the unique $(S_2,S_1)$-alternating normal form $\ggn$ of an element $g$ of $M$. 
    The length of $\ND{g}$ is the \emph{breadth} of $g$ and is denoted  by $\brd(g)$.
    We also define the \emph{depth} of $g$, denoted by $\dpg$, as the number of entries of $\ND\gg$ not belonging to $M_{S_1}$ :
    \begin{equation}
    \label{E:dpt:bh}
      \dpg = \text{card}\left\{ i\mid  g_i\not\in M_{S_1}\right\} = \left\lfloor \frac{\brd(g)}{2} \right\rfloor
    \end{equation}
    Let us remark that $\dpg$ can not be defined by $\text{card}\left\{ i\mid  g_i\in M_{S_2}\right\}$ since we can have $g_1\in M_{S_1}\cap M_{S_2}$. 
    The following fact is an immediate consequence of the definition of the alternating normal form.
    
    \begin{ft}
    For a covering $(S_1,S_2)$ and $g\in M$, we have: $\dpg = 0 \iff\brd(g) = 1\iff g\in M_{S_1}$. 
    \end{ft}

    \subsection{Dehornoy structure}
    
    For each $g\in G$ there exists a unique  pair $(k,h)$ with $k\in \NN$ and $h\in M$ such that $g=h\DD^{-k}$ with $k$ minimal for this equality. This decompostion is called   the $\DD$-fraction of $g$.
    
    \begin{df}
    Let $g$ be an element of $G$ with $\DD$-fraction $h\DD^{-\kk}$.
    We say that $g$ is $(S_2,S_1)$-\emph{negative} if $k\geq 1$ and $\dpt(h)<\dpt(\DD^k)$.
    We say that $g$ is $(S_2,S_1)$-\emph{positive} if $g\inv$ is $(S_2,S_1)$-negative.
    By~$P_{S_2,S_1}$ we denote  the set of all $(S_2, S_1)$-positive elements.
    \end{df}
    
    Following the work D. Arcis and L. Paris, we introduce the notion of a Dehornoy covering which is strongly connected to the notion of a \emph{Dehornoy structure} introduced in \cite{ArP2019}.
    \begin{df}
    We say that the covering $(S_2,S_1)$ is a \emph{Dehornoy covering}, and that $(M_2,M_1)$ is a   \emph{Dehornoy structure},   if $P=P_{S_2,S_1}$ satisfies the following conditions:
    \begin{enumerate}
    \item $P\, P \subseteq P$,
    \item $G_{S_1}\, P\cdot G_{S_1} \subseteq P$,
    \item $G=P\inv \sqcup G_{S_1} \sqcup P$.
    \end{enumerate}
    \end{df}
    In \cite{ArP2019} it is proved that a covering $(S_2,S_1)$ which satisfies  Condition~$A$ (see Defintion~\ref{D:ConditionA}) and Condition~$B$ (see Definition~\ref{D:ConditionB}), has to be a Dehornoy covering. 
    Roughly speaking, Condition~$A$ relies on the depth of powers of Garside element while Condition~$B$ is about the depth of the  product of two elements. 
 
    \subsection{Rigidity of the alternating normal form}
        
    Left and right divisors of the first and the last entries of the $(S_2, S_1)$-alternating normal form $\ND{g}$ of an element~$g\in M$ will play a key role our proofs. 
    For $g\neq 1$ in~$M$, we fix the following notations
    \begin{align}
    \SLDA{S}{g}{*}&\text{ stands for }\SLD{S}{g_n};\\
    \SRDA{S}{g}{*}&\text{ stands for } \SRD{S}{g_1}\ \text{when $g_1\neq 1$ and for $\SRD{S}{g_2}$ otherwise}.
    \end{align}
    Note that $\SLDA{S}{g}{*}\subseteq \SLD{S}{g}$ and  $\SRDA{S}{g}{*}\subseteq \SRD{S}{g}$ hold, but we do not have equalities in general.  
      
    \begin{lm}
    \label{L:RightSetDecomposition} 
     For a non-trivial element~$g$  of $M$ we have $\SRDA{S}{g}{}= \SRDA{S}{g}{*} \cup \SRDA{S_2}{g}{}$.
     Moreover, whenever $g_1$ is trivial, we have $\SRDA{S}{g}{} = \SRDA{S}{g_2}{} = \SRDA{S}{g}{*}$.  
     \end{lm}
    
    \begin{proof} Since $(S_2, S_1)$ is a covering of $S$ we have $\SRDA{S}{g}{}= \SRDA{S_1}{g}{} \cup \SRDA{S_2}{g}{}$.  
    Assume $g_1$ is not trivial. 
    Then $\SRDA{S}{g}{*}$ is defined to be $\SRDA{S}{g _1}{} = \SRDA{S_1}{g_1}{} \cup \SRDA{S_2}{g_1}{}$.
    As right divisors of $g$ in $S_1$ are exactly those of~$g_1$, we have~$\SRDA{S_1}{g_1}{}=\SRDA{S_1}{g}{}$, implying
    \[
     \SRDA{S}{g}{*} \cup \SRDA{S_2}{g}{} = \SRDA{S_1}{g}{} \cup \SRDA{S_2}{g_1}{} \cup \SRDA{S_2}{g}{} = \SRDA{S_1}{g}{} \cup \SRDA{S_2}{g}{} = \SRDA{S}{g}{}.
    \]
    Assume $g_1=1$ and so $\SRDA{S}{g}{*} = \SRDA{S}{g_2}{}$.
    Since $g_1$ is trivial, the set $\SRDA{S_1}{g}{}$ is empty and we obtain $\SRDA{S}{g}{} = \SRDA{S_2}{g}{}$.
    As right divisors of $g$ in $S_2$ are exactly those of~$g_2$, we have~$\SRDA{S_2}{g_2}{}=\SRDA{S_2}{g}{}$ and we obtain the desired equalities.
    \end{proof}
    
    The following lemma partially describes the becoming of the alternating normal form under multiplication.
     
    \begin{lm} 
    \label{L:ProductANF}
    Let $g$ and $h$ be two non-trivial elements of $M$ satisfying $\SRDA{S}{g}{}\subseteq \SLDA{S}{h}{*}$.
    We denote by $\ggm$ and $\hhn$ the respective $(S_2,S_1)$-alternating normal forms of $g$ and $h$.
    \begin{enumerate}
    \item  If $n$ is even and $g_1\in M_{S_2}$ then  $\ND{gh} = (g_m,\ldots,g_3,g_2\,g_1\,h_n,h_{n-1},\ldots, h_1)$. In particular in this case, $\brd(gh) = m + n- 2$ and $\dpt(gh) = \dpt(g) + \dpt(h) - 1$.
    \item If $n$ is odd and $g_1\neq 1$ then $\ND{gh} = (g_m,\ldots,g_2,g_1h_n,h_{n-1},\ldots, h_1)$. In particular  in this case, $\brd(gh) = m + n - 1$ and $\dpt(gh) = \dpt(g)+\dpt(h)$.
    \end{enumerate}
    \end{lm}
    
    \begin{proof} 
    Consider the hypotheses of case (i) or of case (ii), and by $(r_\ell, \ldots, r_1)$ denote the expression provided by the statement for $\ND{gh}$.
    It is immediate that $(r_\ell,\ldots,r_1)$ is an alternating decomposition of $g h$ and that all entries are non-trivial, except possibly for $r_1$.
    To prove, in both cases, that $(r_\ell,\ldots,r_1)$ is an alternating normal form we have to establish
    \begin{equation}
    \label{E:ProductANF:tau}
    \tail{S_{\rem{i}}}{r_\ell\ldots r_i} = r_i\quad\text{for all $i$ in $\{1,\ldots,\ell\}$.}
    \end{equation}
    Since $(g_m,\ldots,g_1)$ is itself an alternating normal form, \eqref{E:ProductANF:tau} is established for $i \geq n+1$ in both cases.
    Let us establish \eqref{E:ProductANF:tau} for $i = n$.
    Using Lemma~\ref{L:Tail}, we obtain 
    \[
    \tail{S_{\rem{n}}}{r_\ell\ldots r_n} = \tail{S_{\rem{n}}}{g h_n} = \begin{cases} \tail{S_2}{g h_n} = \tail{S_2}{g} h_n& \text{if $n$ is even,}\\ \tail{S_1}{g h_n} = \tail{S_1}{g} h_n& \text{if $n$ is odd.} \end{cases}
    \]
    In case (i), $n$ is even and $g_1\in M_{S_2}$. So by Lemma~\ref{L:Tail} we get $\tail{S_2}{g}=g_2\,g_1$, as expected.
    Case (ii) is direct since $n$ is odd and $\tail{S_1}{g}=g_1$.
    We now deal with the cases $i<n$.
    We have
    \[
    \tail{S_{\pi(i)}}{r_\ell\cdots r_i} = \tail{S_{\pi(i)}}{g h_n\,\cdots h_i}.
    \]
    As $h_i$ belongs to $M_{S_{\pi(i)}}$, there exists $h'_i \in M_{S_{\pi(i)}}$ such that $\tail{S_{\pi(i)}}{r_\ell\cdots r_i} = h'_i h_i$.
    Assume, for a contradiction, that $h'_i$ is not trivial. Then it exists $s\in S_{\pi(i)}$ right divising $g h_n\, \ldots h_{i+1}$.
    Since $\tail{S_{\pi(i)}}{h_n \ldots h_i} = h_i$, the element $s$ is not a right divisor of $h'=h_n\,\ldots h_{i+1}$.
    Let $u$ and $v$ be word representatives of $g$ and $h'$ respectively. 
    By $(iii)$ and $(iv)$ of Proposition~\ref{P:ChainDiv}, $v$ is a left $(t,s)$-chain and $t$ right divides $g$.
    
    \[
    \begin{tikzpicture}
    \draw[-latex](0,0) -- (1,0) node[midway,below]{$u\,v$};
    \draw[-latex](1,0.5) -- (1,0) node[midway, right]{$s$};
    \draw[gray,-latex](0.75,0.45) -- (0.25,0.45);
    \draw(2,0.25) node{$=$};
    \begin{scope}[shift={(3,0)}]
    \draw[-latex](0,0) -- (1,0) node[midway,below]{$u$};
    \draw[-latex](1,0) -- (2,0) node[midway,below]{$v$};
    \draw[-latex](1,0.5) -- (1,0) node[midway, right]{$t$};
    \draw[-latex](2,0.5) -- (2,0) node[midway, right]{$s$};
    \draw[gray,-latex](0.75,0.45) -- (0.25,0.45);
    \draw[gray,-latex](1.75,0.45) -- (1.25,0.45);
    \end{scope}
    \end{tikzpicture}
    \]
    Hence $t$ belongs to $\SRDA{S}{g}{}$ and so, by hypothesis, to $\SLDA{S}{h}{*} = \SLDA{S}{h_n}{}$.
    It follows that $t$ left divides~$h_n$ and so $h'$. 
    By Proposition~\ref{P:ChainRev}, $v$ is equivalent to a right $(t,s)$-chain $v'$, which with Proposition~\ref{P:ChainDiv} implies that $t$ does not left divide $\overline{v}=h'$.
    We obtain a contradiction. 
    Therefore  $\tail{S_{\pi(i)}}{r_\ell\cdots r_i}$ is equal to $h_i$, which is $r_i$, and we are done.
    \end{proof}

     \begin{lm}
     \label{L:SuppComm}
     Let $g_1, g_2$ and $g_3$ be element of $M$ satisfying $\supp(g_1)\cap\supp(g_2)=\emptyset$ and $\SRD{S}{g_3g_2} = \SRD{S}{g_2}$.
    If $s\in S$ left divides $g_3g_2g_1$ and $s\not\in \supp(g_3g_2)$, then $s$ left divides $g_1$ and commutes with all the elements of $\supp(g_3g_2)$.
     \end{lm}
     
     \begin{proof}
     Let $s\in S$ be a left divisor of $g_3g_2g_1$ not belonging to $\supp(g_3g_2)$.
    By proposition~\ref{P:SuppChain}, any representative of $g_3g_2$ is a right $(s,s)$-chain.
     Since $s$ left divides $g_3g_2g_1$, Proposition~\ref{P:ChainDiv} $(ii)$ implies that $s$ left divides $g_1$.
     Let $su'_1$, $u_2$ and $u_3$ be  representative words of $g_1$, $g_2$ and $g_3$, respectively. 
     We prove the last statement by induction on the length of $u_3 u_2$.
     There is nothing to prove if  $u_3u_2$ is empty. 
     Assume $u_3u_2$ is of length~$\ell\geq 1$.
     Hypothesis $\SRD{S}{g_3g_2} = \SRD{S}{g_2}$ implies $g_2\not=1$ and so $u_2\not=\varepsilon$.
     Write $u_2=t_k\cdots t_1$ and $u_3= t_{\ell}\cdots t_{k+1}$ as $S$-word. Note that $u_3$ can be the empty word.
     Assume for a contraction that $s$ does not commute with all elements of $\supp(u_3u_2)$.
     Then, it exists an minimal $i\in\{1,\ldots,\ell\}$ such that $s\, t_i \not\equiv t_i\, s$ and $ s\, t_j \equiv t_j\,s$ for all $j<i$.
     We have 
     \begin{equation*}
     u_3u_2u_1=  \tt\ell\cdots \tt\kkp \tt\kk\cdots \tt1 u_1 = \tt\ell\cdots \tt\iip \tt\ii\cdots  \tt1 \cdot s\cdot u'_1 \equiv \tt\ell \cdots \tt\ii \cdot s\cdot \tt\iio\cdots \tt1 u'_1.
    \end{equation*}
     Since $s\not\in\supp(\tt\ell\cdots\tt\iip)\subseteq \supp(g_3g_2)$, Proposition~\ref{P:SuppChain} implies that $\tt\ell\cdots \tt\iip$ is a right $(s,s)$-chain.
     From $s\,\tt\ii\not\equiv\tt\ii\, s$ we obtain the following right reversing diagram with $t\in S$: 
     \begin{center}
    \begin{tikzpicture}
    \draw[latex-] (2,0) -- (2,0.5) node[midway, right]{$w'$};
    \draw[latex-] (2,0.5) -- (2,1) node[midway, right]{$t$};
    \draw[-latex] (0,0) -- (2,0) node[midway, below]{$w$};
    \draw(1,0.5) node{$\curvearrowright_R$};
    \draw[latex-] (0,0) -- (0,1) node[midway, left]{$s$};
    \draw[-latex] (0,1) -- (1,1) node[midway, above]{$\tt\ii$};
    \draw[-latex] (1,1) -- (2,1) node[midway, above]{$s$};
    \end{tikzpicture}
    \end{center}
    By Proposition~\ref{P:ReversingLcm}, the word $sw$ is a representative of $\overline{\tt\ii}\,s\vee_R \overline{s}=\overline{\tt\ii}\vee_R \overline{s}$, which also admits $\tt\ii\,s\,\tt\ii\cdots$ as representative.
    We have necessarily $t=\tt\ii$ and $\tt\ii\,s$ is a right $(s,\tt\ii)$-chain. We obtain that $\tt\ell\cdots \tt\iip\,\tt\ii \,s$ is a right $(s,\tt\ii)$-chain. 
    By assumption, $s$ commutes with $\tt1,\ldots,\tt\iio$ and not with $\tt\ii$ and so 
    \begin{equation}
     \label{E:SuppComm:1}
    \tt\ii\not\in\supp(\tt\iio\cdots\tt1).
    \end{equation}
    Using Proposition~\ref{P:SuppChain} we obtain that $\tt\iio\cdots\tt1$ is a right $(\tt\ii,\tt\ii)$-chain. We can summarize all these informations in the following diagram:
    \begin{equation}
     \label{E:SuppComm:2}
    \begin{tikzpicture}
    \draw[-latex](0,1) -- (1,1) node[midway,above]{$\tt\ell$};
    \draw[dotted](1,1) -- (2,1);
    \draw[-latex](2,1) -- (3,1) node[midway,above]{$\tt\iip$};
    \draw[-latex](3,1) -- (4,1) node[midway,above]{$\tt\ii$};
    \draw[-latex](4,1) -- (5,1) node[midway,above]{$\ss{}$};
    \draw[-latex](5,1) -- (6,1) node[midway,above]{$\tt\iio$};
    \draw[dotted](6,1) -- (7,1);
    \draw[-latex](7,1) -- (8,1) node[midway,above]{$\tt1$};
    \draw[-latex](8,1) -- (9,1) node[midway,above]{$u'_1$};
    \draw[-latex](0,1) -- (0,0) node[midway,left]{$s$};
    \draw[-latex](3,1) -- (3,0) node[midway,left]{$s$};
    \draw[-latex](5,1) -- (5,0) node[midway,left]{$\tt\ii$};
    \draw[-latex](8,1) -- (8,0) node[midway,left]{$\tt\ii$};
    \draw[gray,-latex](1.25,0.5) -- (1.75,0.5);
    \draw[gray,-latex](3.75,0.5) -- (4.25,0.5);
    \draw[gray,-latex](6.25,0.5) -- (6.75,0.5);
    \draw[gray,-latex](8.25,0.5) -- (8.75,0.5);
    \end{tikzpicture}
    \end{equation}
    In particular, the element $\overline{u'_1}$ is left divisible by $t_i$ and so $t_i\in\supp(u'_1)\subseteq \supp(u_1) = \supp(g_1)$.
    By hypothesis, the latter gives $t_i\not\in\supp(g_2)=\supp(u_2)$ and so $i\geq \kkp$.
    We write $u'_3 = \tt\iio\ldots\tt\kkp$, $u'_2=u_2$ and $s'=\tt\ii$.
    From \eqref{E:SuppComm:2}, we know that $s'$ left divides $\overline{u'_3u'_2u'_1}$ and  $s'\not\in\supp(u'_3u'_2)$ by \eqref{E:SuppComm:1}.
    We also have $|u'_3u'_2|<|u_3u_2|$, $\supp(\overline{u'_1})\cap\supp(\overline{u'_2})\subseteq \supp(g_1)\cap \supp(g_2) =\emptyset$ and
     \[
     \SRD{S}{\overline{u'_3u'_2}} \subseteq \SRD{S}{\overline{u_3u_2}}  = \SRD{S}{g_2}.
     \]
    By the induction hypothesis, we get that $\tt\ii = s'$ commutes with all elements of $\supp(\overline{u'_3}\,\overline{u'_2})=\{\tt\iio,\ldots,\tt1\}$.
    This implies that $t_i$ right divides $g_3g_2=\overline{t_\ell \cdots t_{i+1} \cdot t_i \cdot t_{i-1} \cdots t_1}$ and so $t_i$ is a right divisor of $g_2$ since~$R(g_3g_2, S)=R(g_2, S)$.
    In particular $t_i$ belongs to the support of $g_2$, which is impossible by~\eqref{E:SuppComm:1}.
    \end{proof}
     
    \subsection{Power of Garside element}
     
    The behavior of the alternating normal form on powers of the Garside element plays a key role in the existence of the ordering

    For a non-empty subset~$T$ of $S$ we denote by~$\DD_T$ the left common multiple of the elements of $T$: it is the minimal Garside element of the Garside monoid~$M_T$ (see~\cite{God5,DDGKM}).
    By construction, $\Delta_T$ left and right divides $\Delta$ and we have 
    \begin{equation}
    \DD= \alpha(\DD_T) \cdot \DD_T \quad \text{together with} \quad \Phi(\Delta_T)=\Delta_{\Phi(T)}.
    \end{equation}
     
     \begin{lm}\label{L:DeltaDecomposition}
     Let $T \not=\emptyset$ be a subset of $S$ and $w$ be a representative of $g = \alpha(\DD_T)$.
     The following properties hold:
    \begin{enumerate}   
    \item for any $p\geq 1$, $\DD^p = \Phi^{p-1}(g)\cdots \Phi(g)\,g\,\DD_T^p$;
    \item  for any $p\geq 1$, $\SLD{S}{\Phi^{p-1}(g)\cdots \Phi(g)\,g} = S\setminus \Phi^p(T)$ and $\SRD{S}{\Phi^{p-1}(g)\cdots \Phi(g)g} = S\setminus T$;
    \item for any $p\geq 0$ and  $s\in T$, the word $\Phi^{p-1}
    (w)\cdots \Phi(w)\,w$ is a left $(t,s)$-chain for some $t\in \Phi^p(T)$;
    \item for any $p\geq 0$ and $t\in\Phi^p(T)$, the word $\Phi^{p-1}(w)\cdots \Phi(w)\,w$ is a right $(t,s)$-chain for some $s\in T$.
    \end{enumerate}    
    \end{lm}
    
    \begin{proof}
    Point (i) directly follows from Corollary~\ref{C:PowerDelta}. 
    By Corollary~\ref{C:DivGarsideChain} applied on $\DD_T$, for $s\in T$, the word~$w$ is a left $(t,s)$-chain with $t\in \Phi(T)$.
    Point $(iii)$ follows using Lemma~\ref{L:ChainPhi}.
    Point $(iv)$ is a consequence of $(iii)$ and Proposition~\ref{P:ChainRev}.
    Assume $p\geq 1$ and put 
    \[
    g_p = \Phi^{p-1}(g)\cdots \Phi(g)\,g.
    \]
    Thanks to Proposition~\ref{P:ChainDiv}, Points $(iii)$ and $(iv)$ give $R(g_p) \subseteq S \setminus T$ and $L(g_p)\subseteq S\setminus \Phi^p(T)$.
    Any element~$s$ of~$S\setminus T$ left-divides $\DD^p$ whereas any word representing $\DD^p_T$ is an $(s,s)$-chain by Proposition~\ref{P:SuppChain}. 
    Since $\DD^p =g_p \DD_T^p$, it follows from Proposition~\ref{P:ChainDiv} that all elements of $S\setminus T$ right-divides $g_p$, giving $\SRD{S}{g_p}= S\setminus T$.  
    Replacing the right-divisibility by the left divisibility and $T$ by $\Phi^p(T)$, we can repeat the argument to get that $\SLD{S}{g_p} = S\setminus \Phi^p(T)$. Point (ii) follows. 
    \end{proof}

     \begin{df}
    We denote by
     \begin{itemize}
    \item  $S_1^\perp$ the set of elements of $S$ commuting with all elements of $S_1$;
     \item $S_2\ominus S_1$ the set $(S_2\setminus S_1) \cap S_1^\perp$;
     \item $S_2\boxminus S_1$ the set $S_2 \setminus (S_1 \cup S_1^\perp)$.
     \end{itemize}
     \end{df}
     
    It may be helpful to visualize the different sets on a diagram:
    
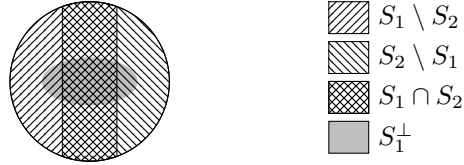
\begin{center}
    \begin{tikzpicture}[x=30pt,y=30pt]
    \draw (0,0) circle (1);
    \fill[lightgray] (0,0) ellipse (0.6 and 0.3);
    
    \draw (110:1) -- (-110:1);
    \draw (70:1) -- (-70:1);
    
    \draw[pattern=north west lines] (110:1) arc (110:-110:1);
    \draw[pattern=north east lines] (70:1) arc (70:290:1);
    \begin{scope}[shift={(3,0.1)}]
    \draw[pattern=north east lines] (0,0.5) rectangle ++ (0.5,0.4);
    \draw (0.5,0.7) node[right]{$S_1\setminus S_2$};
    \draw[pattern=north west lines] (0,0.0) rectangle ++ (0.5,0.4);
    \draw (0.5,0.2) node[right]{$S_2\setminus S_1$};
    \draw[pattern=north west lines] (0,-0.5) rectangle ++ (0.5,0.4);
    \draw[pattern=north east lines] (0,-0.5) rectangle ++ (0.5,0.4);
    \draw (0.5,-0.3) node[right]{$S_1\cap S_2$};
    \draw[fill=lightgray] (0,-1) rectangle ++ (0.5,0.4);
    \draw (0.5,-0.8) node[right]{$S_1^\perp$};
    \end{scope}
    \end{tikzpicture}
    \end{center}
†    In particular, we remark the relation
    \begin{equation}
    \label{E:OminusBoxminus}
     S_2\setminus S_1=(S_2\ominus S_1)\sqcup (S_2\boxminus S_1)
    \end{equation}
    In the same way, we define $S_2^\perp$, $S_1\ominus S_2$ and $S_1\boxminus S_2$.

    \begin{lm} 
    \label{L:DeltaAlternatingNormalForm}
    Let $p$ be a positive integer.
    Writing $\ND{\DD^p}=(g_n,\ldots,g_1)$, we have $g_1 = \DD^p_{S_1}$ and $g_2$ is right-divisible by $(\DD_{S_2\setminus S_1}\DD_{S_2\ominus S_1}^{-1})\DD_{S_2\ominus S_1}^p$. 
    In particular, $g_2$ is right-divisible by $\DD_{S_2\setminus S_1}$.
    \end{lm}
    \begin{proof}
    By Lemma~\ref{L:DeltaDecomposition} $(i)$ we have $\DD^p = \Phi^{p-1}(g) \cdot \ldots \cdot \Phi(g) \cdot g \cdot \DD_{S_1}^p$ with $g =\alpha(\DD_{S_1})$.
    Moreover, $\DD_{S_1}^p$ lies in $M_{S_1}$ and by Lemma~\ref{L:DeltaDecomposition} $(ii)$,  $\Phi^{p-1}(g) \cdot \ldots \cdot \Phi(g) \cdot g$ is not right-divisible by any element of $S_1$.
    So $g_1 = \tail{S_1}{\DD^p} = (\DD_{S_1})^p$.  We now establish the result on $g_2$.
    From $S_2\ominus S_1 \subseteq S_1^\perp$, we get 
    \[
    \DD_{S_2\ominus S_1}\cdot \DD_{S_1}=  \DD_{S_1}\cdot \DD_{S_2\ominus S_1}= \DD_{(S_2\ominus S_1) \cup S_1}.\]
    In particular $(\DD_{S_2\ominus S_1}\cdot \DD_{S_1})^p = \DD_{S_2\ominus S_1}^p\cdot \DD_{S_1}^p$.
    Applying Lemma~\ref{L:DeltaDecomposition} for $T = (S_2\ominus S_1) \cup S_1$, we get that
    \[
    \DD^p = \Phi^{p-1}(h) \cdot \ldots \cdot \Phi(h) \cdot h \cdot \DD_{T}^p =\Phi^{p-1}(h) \cdot \ldots \cdot \Phi(h) \cdot h \cdot \DD_{S_2\ominus S_1}^p\cdot \DD_{S_1}^p,
    \]
     where $h = \alpha(\Delta_T)$.  
     Note that $\DD  = g\cdot\DD_{S_1} = h\cdot \DD_{S_2\ominus S_1}\cdot \DD_{S_1}$.
     In particular $\DD_{S_2\setminus S_1}$ right-divides~$h\cdot \DD_{S_2\ominus S_1}$, by Lemma~\ref{L:DeltaDecomposition} $(ii)$.
     Hence $\DD_{S_2\setminus S_1}\DD_{S_2\ominus S_1}\inv$, which lies in $M$, right-divides $h$.
     Finally,
     \[
     g'_2=\DD_{S_2\setminus S_1}\cdot \DD_{S_2\ominus S_1}\inv \cdot \DD_{S_2\ominus S_1}^p
     \]
     lies in $M_{S_2}$ and right-divides $g_n\cdots g_2$, which implies that $g'_2$ is a right divisor of $g_2$.
    \end{proof}
    
    \begin{lm}
    \label{L:LeftPowerDelta}
    Let $p\geq 1$ such that $\brd(\DD^p) \geq 3$.
    Then $\SLDA{S}{\DD^p}{*}$ is included in $\Phi^p(S_2\boxminus S_1)$
    \end{lm}
    \begin{proof} 
    Consider $h=\alpha(\Delta_T)$ with $T = (S_2\ominus S_1) \cup S_1$ as in the proof of Lemma~\ref{L:DeltaAlternatingNormalForm}.
    We have
    \[
    \DD^p = \Phi^{p-1}(h) \cdot \ldots \cdot \Phi(h) \cdot h \cdot \DD_{T}^p =\Phi^{p-1}(h) \cdot \ldots \cdot \Phi(h) \cdot h \cdot \DD_{S_2\ominus S_1}^p\cdot \DD_{S_1}^p.
    \] 
    Let $\ggn$ be the alternating normal form of $\DD^p$. 
    Then $g_1 = \DD_{S_1}^p$ and $\DD_{S_2\ominus S_1}$ right-divides $g_2$. 
    Since $n\geq 3$, it follows that $g_n$ is a left divisor of $g = \Phi^{p-1}(h) \cdot \ldots \cdot \Phi(h) \cdot h $,  that implies the inclusion $\SLDA{S}{\DD^p}{*} \subseteq \SLD{S}{g}$. 
    By Lemma~\ref{L:DeltaDecomposition}, the latter is equal to $S\setminus\Phi^p(T)$.
    We have \begin{align*}
    S\setminus T &= S\setminus \left ((S_2\ominus S_1)\cup S_1\right)\\
    &= \big(S\setminus (S_2\ominus S_1)\big)\cap (S\setminus S_1)\\
    &= \big(S\setminus (S_2\ominus S_1)\big)\cap (S_2\setminus S_1) & \text{since $S=S_1\cup S_2$}\\
    &= (S_2 \setminus S_1)\setminus (S_2\ominus S_1) = S_2\boxminus S_1  &\text{by \eqref{E:OminusBoxminus}.}
    \end{align*}
    If $p$ is even, we have  $S\setminus\Phi^p(T) =  S\setminus T = S_2\boxminus S_1 = \Phi^p(S_2\boxminus S_1)$.
    If $p$ is odd, we have $\Phi^p(T) = \Phi(T)$ and then $S\setminus\Phi^p(T)=S\setminus\Phi(T) = \Phi(S\setminus T) = \Phi(S_2\boxminus S_1)= \Phi^p(S_2\boxminus S_1)$.
    \end{proof}
    
    \section{Condition~$A$} 
    \label{S:ConditionA}
    
    $\Gamma$, $S$, $M$, $G$ and $\Delta$ are still as defined in Notation~\ref{N: cadre}.
    We fix a proper covering $(S_2,S_1)$ of $S$. 
     
        \begin{df}[Condition~$A$]\label{D:ConditionA} Let $p\geq 1$, an integer.  The covering $(S_2, S_1)$  is said to satisfy Condition~$A$ with respect to~$\DD^p$  when for all  $t\geq 1$ we have 
    \begin{equation}
    \label{E:ConditionA}
    \dpt(\DD^{pt}) - 1 = t\times (\dpt(\DD^p) - 1).
    \end{equation}
    \end{df} 

  \begin{df} \label{D:irred et stab}
    For $p\geq1$, the covering $(S_2, S_1)$ is said to be
    \begin{enumerate}
    \item irreducible whenever $M_{S_1}$ and $M_{S_2}$ are also irreducible,
        \item $\DD^p$-stabilized if $\{\Phi^p(S_1),\Phi^p(S_2)\}=\{S_1,S_2\}$.
    \item $\DD^p$-fixed if $\Phi^p(S_1)=S_1$ and $\Phi^p(S_2)=S_2$,
    \end{enumerate}
    \end{df}

  As by Proposition~\ref{P:Garside}, $\DD^2$ is central. Therefore,  the covering $(S_2,S_1)$ is $\DD^p$ fixed for any $p$ even; when $p$ is odd,   $(S_2,S_1)$ is $\DD^p$-stabilized or  $\DD^p$-fixed iff it  $\DD$-stabilized or  $\DD$-fixed, respectively. 
  
\begin{df}
    	For $p\geq 1$, the covering $(S_2, S_1)$ of $S$ is said to be \emph{$\DD^p$-regular} if it is proper, irreducible and $\DD^p$-stabilized.
    	\end{df}   
    
       The element $\DD$ is crucial in the definition of  the $\DD$-fraction $h\DD^{-k}$ of an element  $g$of $G$. Since  the $\DD$-fraction of $\DD g \DD^{-1}$ is  $(\DD h \DD^{-1})\DD^{-k}$ with $\DD h \DD^{-1}$ in $M$, one may want some compatibility between the alternating normal form and the action of $\Delta$ by conjugacy.  This is why, in the sequel, we  assume  the covering $(S_2, S_1)$ is stabilized by~$\DD$:
 
 \begin{nota} During Section~\ref{S:ConditionA}, we fix $p\geq 1$, an integer and we assume that $(S_2,S_1)$ is  a $\DD^p$-regular covering  of $S$. \label{N:cov regular}
 \end{nota}     
 
 \subsection{Condition A and $\DD^p$-fixed coverings}
    \begin{lm}
    \label{L:DeltaBreathForFixedCovering}
  
   $\brd(\DD^p)$ is even iff $(S_2,S_1)$ is $\DD^p$-fixed.
    \end{lm}
    
    \begin{proof}
    We put $[\DD^p]=(g_n,\ldots,g_1)$.
    As $S\not=S_1$ we have $n\geq 2$.
    Proposition~\ref{P:Garside}(iv) gives
    \[
     \DD^p = g_{n-1}\cdots g_1 \Phi_{\DD^p}(g_n) = g_{n-1}\cdots g_1 \Phi^p(g_n).
     \]
    Hence $\Phi^p(g_n)$ can not lie in $M_{S_1}$.
    Otherwise $g_1\Phi^p(g_n)$ should right-divides $g_1$, which is impossible as~$g_n\neq 1$.
    So we must have $\Phi^p(S_{\pi(n)}) = S_2$.
    If $n$ is even we get $\pi(n) = 2$ and $\Phi^p$ fixes the covering.
    If $n$ is odd, $\pi(n) = 1$ and the covering is not fixed by $\Phi^p$.
    \end{proof}

    \begin{lm}  \label{L:Breadth2Reducible} The element $\DD^p$ cannot have  breadth $2$.

    \end{lm}

    \begin{proof}
    We prove the result by contradiction. Assume $[\DD^p]=(g_2,g_1)$. 
    Define $T_2$ to be $\supp(g_2)$ and $T_1$ to be $S\setminus T_2$.
    We have $S\setminus S_1 \subseteq T_2\subseteq S_2$ and $S\setminus S_2 \subseteq S\setminus T_2 = T_1 \subseteq S_1$.
    Hence $T_1$ and $T_2$ are both not empty and, by construction, $T_1\sqcup T_2 = S$ holds.
    Lemma~\ref{L:DeltaAlternatingNormalForm} gives $g_1=\DD^p_{S_1}$ and $g_2$ is right divisible by $h_2=\DD_{S\setminus S_1}$. 
    Let $h_3\in M_{S_2}$ be such that $g_2=h_3\,h_2$.
    Putting $h_1=g_1$ we obtain $\DD^p=h_3\,h_2\,h_1$ with $\supp(h_1)\cap \supp(h_2)= S_1\cap (S\setminus S_1) = \emptyset$ and $R(h_3h_2,S) = R(g_2,S) = S\setminus S_1= R(h_2,S)$.
    Thus, from Lemma~\ref{L:SuppComm}, any element $s \in T_1$ commutes with all elements of $T_2$, giving $M=M_{T_2}\times M_{T_1}$.  So, $\Gamma$ is not irreducible, a contradiction.
    \qedhere
    
    \end{proof}
    
    \begin{prp}
    \label{P:LeftOfPowerDeltaBoxMinus}
    Assume $\SLDA{S}{\DD^p}{*}$ is equal to $\Phi^p\left(S_2\boxminus S_1\right)$.
    Then, for any $t\geq 1$, the set~$\SLDA{S}{\DD^{pt}}{*}$ is equal to $\Phi^{pt}(S_2\boxminus S_1)$ and  
    \[
    \brd(\DD^{pt})-2 = t\times (\brd(\DD^{p})-2).
    \]
    Moreover,  
    \begin{enumerate}
    \item  if $\DD^p$ fixes the covering $(S_2,S_1)$  then
     $$\dpt(\DD^{pt})  - 1= t \times (\dpt(\DD^p) -1)$$
     \item  if $\DD^p$ does not  fix the covering $(S_2,S_1)$  then
     \begin{enumerate}
     \item when $t$ is even, one has $\dpt(\DD^{pt}) - 1 = t \times (\dpt(\DD^p) -\frac{1}{2})$; 
     \item when $t$ is odd, one has $\dpt(\DD^{pt}) - \frac{1}{2}= t \times (\dpt(\DD^p) -\frac{1}{2})$.
      \end{enumerate}
      \end{enumerate}
    \end{prp}
    
    \begin{proof} 
    We prove the assertions on $\SLDA{S}{\DD^{pt}}{*}$ and on $\brd(\DD^{pt})$ by induction on $t$.
    For $t = 1$ there is nothing to prove.   
    Assume $t\geq 2$,  $\brd(\DD^{p(t-1)})-2 = (t-1)\times (\brd(\DD^{p})-2)$ and that $\SLDA{S}{\DD^{p(t-1)}}{*} $ is equal to~$\Phi^{p(t-1)}(S_2\boxminus S_1)$. 
    Denote by~$(g_n,\ldots,g_1)$ and $(h_m,\ldots,h_1)$ the alternating form of~$\DD^p$ and $\DD^{p(t-1)}$ with respect to~$(S_2, S_1)$.
    Since~$M$ is irreducible, Lemma~\ref{L:Breadth2Reducible} implies that $n$ and $m$ are greater or equal to~$3$.
    As Lemma~\ref{L:DeltaAlternatingNormalForm} implies $g_1= \DD_{S_1}^p$ and $g_2=g'_2 \cdot (\DD_{S_2\setminus S_1}\DD_{S_2\ominus S_1}\inv) \DD_{S_2\ominus S_1}^p$, with $g'_2\in M_{S_2}$, we obtain
    \begin{align}
    \label{E:P:LeftOfPowerDeltaBoxMinus:1}
    \begin{split}
    \DD^{pt} &= \DD^p \DD^{p(t-1)} = g_n\cdot\ldots\cdot g_2\cdot g_1\cdot \DD^{p(t-1)}\\
    &= g_n\cdot \ldots\cdot g_3 \cdot g'_2\cdot (\DD_{S_2\setminus S_1}\DD_{S_2\ominus S_1}\inv) \cdot \DD_{S_2\ominus S_1}^p \cdot\DD_{S_1}^p \cdot \DD^{p(t-1)} \\
    &= g_n\cdot \ldots\cdot g_3 \cdot g'_2\cdot (\DD_{S_2\setminus S_1}\DD_{S_2\ominus S_1}\inv) \cdot \DD^{p(t-1)} \cdot\Phi^{p(t-1)} (\DD_{S_2\ominus S_1}^p\cdot \DD_{S_1}^p)
    \end{split}
    \end{align}
    
    Set $x= g_n\cdot \ldots\cdot g_3 \cdot g'_2\cdot \DD_{S_2\setminus S_1}\DD_{S_2\ominus S_1}\inv$ and $y =\DD^{p(t-1)} \cdot  \DD_{S_2\ominus S_1}^p\cdot \DD_{S_1}^p$.
    We first determine the alternating normal forms of $x$ and $y$ with respect to~$(S_2,S_1)$ in order to obtain the one of $\DD^{pt}$, using~\ref{E:CAProp:1} together with Lemma~\ref{L:ProductANF}.
    By Lemma~\ref{L:DeltaAlternatingNormalForm}, we have $h_1=\DD_{S_1}^{p(t-1)}$.
    As $S_2\ominus S_1$ is a subset of $S_1^\perp$ the two elements $\DD_{S_1}^{p(t-1)}$ and~$\DD_{S_2\ominus S_1}^p$ commute and we obtain  
    \begin{align*}
    y=\DD^{p(t-1)} \cdot\DD_{S_2\ominus S_1}^p\cdot \DD_{S_1}^p &= h_m\cdot \ldots \cdot h_3 \cdot h_2 \cdot \DD_{S_1}^{p(t-1)}\cdot \DD_{S_2\ominus S_1}^p\cdot \DD_{S_1}^p\\
    &=h_m\cdot \ldots \cdot h_3 \cdot \left( h_2 \cdot \DD_{S_2\ominus S_1}^{p}\right) \cdot \DD_{S_1}^{pt}.
    \end{align*}
    The $(S_2,S_1)$-alternating normal form of $\DD^{p(t-1)}\cdot \DD_{S_1}^{-p(t-1)}$ is $(h_m,\ldots,h_2,1)$.
    Let $w$ ba an $(S_2\ominus S_1)$ word representing $\DD_{S_2\ominus S_1}^p$.
    As $\supp(w)=S_2\ominus S_1$ has empty intersection with $S_1$, the word $w$ is a left $(s,s)$-chain for any $s\in S_1$, thanks to Proposition~\ref{P:SuppChain}.
    Since $h_2$ is not divisible by any element of~$S_1$, Proposition~\ref{P:ChainDiv} guarantees that the $S_1$-tail of $h_2 \DD_{S_2\ominus S_1}^p$ is trivial.
    Hence the $(S_2,S_1)$-alternating normal form of $\DD^{p(t-1)}\cdot \DD_{S_2\ominus S_1}^p$ is $(h_m,\ldots,h_2 \DD_{S_2\ominus S_1}^{p},1)$ and then the $(S_2,S_1)$-alternating normal form of $y$ is 
    \begin{equation}
    \label{E:P:LeftOfPowerDeltaBoxMinus:2}
    \left[y\right] = \left(h_m,\ldots,h_2 \DD_{S_2\ominus S_1}^{p},\DD_{S_1}^{pt}\right)
    \end{equation}
    We now focus on $x$. 
    We recall that $(g_n,\ldots, g_1)$ is an $(S_2,S_1)$-alternating normal form with $n\geq 3$.
    In particular no element of~$S_1$ right divides $g_3g_2$.
    However, since $S_2\ominus S_1$ is a subset of $S_1^\perp$, we have $s\cdot \DD_{S_2\ominus S_1}^p = \DD_{S_2\ominus S_1}^p \cdot s$ for any $s$ in $S_1$.
    From 
    \[
    g_3g_2=g_3g'_2\DD_{S_2\setminus S_1}\DD_{S_2\ominus S_1}\inv\DD_{S_2\ominus S_1}^p,
    \]
    we obtain that $g_3g'_2\DD_{S_2\setminus S_1}\DD_{S_2\ominus S_1}\inv$ is not right-divisible by any element of $S_1$.
    Hence we obtain 
    \begin{equation}
    	\label{E:P:LeftOfPowerDeltaBoxMinus:3.1}
    \text{$g'_2\DD_{S_2\setminus S_1}\DD_{S_2\ominus S_1}\inv$ is not trivial}
    \end{equation}
     and we get that the $(S_2,S_1)$-alternating normal form of $x$ is
    \begin{equation}
    \label{E:P:LeftOfPowerDeltaBoxMinus:3}
    \left[x\right] = (g_n,\ldots,g_3,g'_2 \DD_{S_2\setminus S_1}\DD_{S_2\ominus S_1}\inv, 1).
    \end{equation}
    Moreover, we have $R(x,S_1)=\emptyset$ together with
    \[
    R(x, S_2)=R(g'_2 \DD_{S_2\setminus S_1}\DD_{S_2\ominus S_1}\inv) = (S_2\setminus S_1) \setminus (S_2\ominus S_1).
    \]
    Using \eqref{E:OminusBoxminus} we obtain 
    \begin{equation}
    	\label{E:P:LeftOfPowerDeltaBoxMinus:3.5}
    	R(x,S)= R(x,S_1)\cup R(x,S_2) =  R(x,S_2) = S_2\boxminus S_1.
    \end{equation}
    
    \textbf{$\bullet$ Assume $\DD^{p(t-1)}$ fixes $(S_2,S_1)$.}
    In this case, decomposition~\eqref{E:P:LeftOfPowerDeltaBoxMinus:1} becomes
    \begin{equation}
    \label{E:CAProp:1}
    \DD^{pt}= \underbrace{g_n\cdot \ldots\cdot g_3 \cdot g'_2\cdot \DD_{S_2\setminus S_1}\DD_{S_2\ominus S_1}\inv}_x \,\cdot\, \underbrace{\DD^{p(t-1)} \cdot  \DD_{S_2\ominus S_1}^p\cdot \DD_{S_1}^p}_y
    \end{equation}
    By the induction hypothesis we have  
    \begin{equation}
    \label{E:P:LeftOfPowerDeltaBoxMinus:4}
    \SLDA{S}{y}{*}  = \SLDA{S}{\DD^{p(t-1)}}{*} = \Phi^{p(t-1)}(S_2\boxminus S_1).
    \end{equation}
    and the latter is equal to $S_2\boxminus S_1$ since $(S_2,S_1)$ is assumed to be fixed.
    As, by Lemma~\ref{L:DeltaBreathForFixedCovering}, the breadth~$m$ of $\DD^{p(t-1)}$ is even, Lemma~\ref{L:ProductANF}~(i) with \eqref{E:P:LeftOfPowerDeltaBoxMinus:3.5} implies
    \begin{equation*}
    \left[\DD^{pt}\right]=\left[xy\right]=\left(g_n,\ldots,g_3,g'_2\DD_{S_2\setminus S_1}\DD_{S_2\ominus S_1}\inv\,h_m,\ldots,h_3,h_2\DD_{S_2\ominus S_1}^{p},\DD_{S_1}^{pt}\right).
    \end{equation*}
    together with $\brd\left(\DD^{pt}\right) = n+m-2 = \brd(\DD^p) + \brd\left(\DD^{p(t-1)}\right)-2$.
    An immediate verification, using the induction hypothesis on $\brd\left(\DD^{p(t-1)}\right)$, gives the result on $\brd\left(\DD^{pt}\right)$.
    We conclude this case with
    \[
    \SLDA{S}{\DD^{pt}}{*}=\SLDA{S}{g_n}{}=\SLDA{S}{\DD^{p}}{*}=\Phi^p(S_2\boxminus S_1)=\Phi^p(\Phi^{p(t-1)}(S_2\boxminus S_1))=\Phi^{pt}(S_2\boxminus S_1).
    \]

    \textbf{$\bullet$ Assume $\DD^{p(t-1)}$ does not fix $(S_2,S_1)$.} This case can occur only if $\DD^p$ does not fix $(S_2,S_1)$.
    
    From decompoisiton~\eqref{E:P:LeftOfPowerDeltaBoxMinus:1} we obtain
    \begin{align*}
    \DD^{pt}&= \DD^p \DD^{p(t-1)} = \Phi^p(\DD^p) \DD^{p(t-1)} = \Phi^p(g_n)\cdot\ldots\cdot\Phi^p(g_2)\cdot\Phi^p(g_1)\cdot\DD^{p(t-1)}\\
    &=\Phi^p(g_n)\cdot \ldots\cdot \Phi^p(g_3) \cdot \Phi^p(g'_2)\cdot \Phi^p(\DD_{S_2\setminus S_1}\DD_{S_2\ominus S_1}\inv) \cdot \DD^{p(t-1)} \cdot\Phi^{p(t-1)} \left(\Phi^p(\DD_{S_2\ominus S_1}^p\cdot \DD_{S_1}^p)\right)
    \end{align*}
    As $\Phi^p$ and $\Phi^{p(t-1)}$ swap the sets $S_1$ and $S_2$ we have
    \[
    \DD^{pt}= \underbrace{\Phi^p(g_n)\cdot \ldots\cdot \Phi^p(g_3) \cdot \Phi^p(g'_2)\cdot \DD_{S_1\setminus S_2}\DD_{S_1\ominus S_2}\inv}_z \,\cdot\, \underbrace{\DD^{p(t-1)} \cdot \DD_{S_2\ominus S_1}^p\cdot \DD_{S_1}^p}_y
    \]
    Recall that the $(S_2,S_1)$-alternating normal form of $y$ is given by~\eqref{E:P:LeftOfPowerDeltaBoxMinus:2}.
    Remarking that $z$ is equal to~$\Phi^p(x)$, we obtain from \eqref{E:P:LeftOfPowerDeltaBoxMinus:3} that the $(S_2,S_1)$-alternating normal form of $z$ is
    \[
    \left[z\right] = \left(\Phi^p(g_n),\ldots,\Phi^p(g_3),\Phi^p(g'_2) \DD_{S_1\setminus S_2}\DD_{S_1\ominus S_2}\inv\right).
    \]
    We have $R(z, S)= R(\Phi^p(x), S)=\Phi^p(R(x, S)) = \Phi^p(S_2\boxminus S_1) = S_1\boxminus S_2$. 
    Since $\DD^{p(t-1)}$ does not fix $(S_2, S_1)$, the induction hypothesis gives
    \[
    \SLDA{S}{\DD^{p(t-1)}}{*} = \Phi^{p(t-1)}(S_2\boxminus S_1)=S_1\boxminus S_2,
    \]
    which is equal to $R(z,S)$.
    From \eqref{E:P:LeftOfPowerDeltaBoxMinus:3.1}, we obtain that $\tail{S_1}{z} = \Phi^p(g'_2 \DD_{S_2\setminus S_1}\DD_{S_2\ominus S_1}\inv)$ is not trivial. 
    Thanks to Lemma~\ref{L:DeltaBreathForFixedCovering}, the breadth $m$ of $\DD^{p(t-1)}$ is odd.
    Applying Lemma~\ref{L:ProductANF}~(ii), we obtain
    \begin{equation*}
    \left[\DD^{pt}\right]=\left(\Phi^p(g_n),\ldots,\Phi^p(g_3),\Phi^p(g'_2)\DD_{S_2\setminus S_1}\DD_{S_2\ominus S_1}\inv\cdot h_m,\ldots,h_3,h_2\DD_{S_2\ominus S_1}^{p},\DD_{S_1}^{pt}\right).
    \end{equation*}
    together with $\brd\left(\DD^{pt}\right) = (n - 1)+m-1 = \brd(\DD^p) + \brd\left(\DD^{p(t-1)}\right)-2$.
    An immediate verification, using the induction  hypothesis on $\brd\left(\DD^{p(t-1)}\right)$, gives the result on $\brd\left(\DD^{pt}\right)$.
    To complete this case we compute
    \[
    \SLDA{S}{\DD^{pt}}{*}=\Phi^p(g_n)=\Phi^p(\SLDA{S}{\DD^{p}}{*})=\Phi^p(\Phi^p(S_2\boxminus S_1)) = S_2\boxminus S_1.
    \]
    On the other hand, both $\Phi^p$ and $\Phi^{p(t-1)}$ swap the sets $S_1$ and $S_2$ so 
    \[ 
    \Phi^{pt}(S_2\boxminus S_1) =\Phi^{p(t-1)}(\Phi^p(S_2\boxminus S_1)) = \Phi^{p(t-1)}(S_1\boxminus S_2) = S_2\boxminus S_1.
    \]
    
    It remains to establish the statements on $\dpt(\DD^{pt})$.
    If $\DD^p$ fixes the covering $(S_2,S_1)$, this is also the case for $\DD^{pt}$ and then, by Lemma~\ref{L:DeltaBreathForFixedCovering}, the breadths of these two elements are even, giving :
    \[
    \dpt(\DD^{pt}) - 1 =\frac12 \left(\brd(\DD^{pt}) -2 \right)= t \times \frac12 \left(\brd(\DD^p)-2\right) = t\times(\dpt(\DD^p)-1) 
    \]
    Assume $\DD^p$ does not fix the covering $(S_2,S_1)$.
    In this case $\DD^{pt}$ fixes the covering $(S_2,S_1)$ whenever $t$ is even and does not fix it whenever $t$ is odd.
    Hence, by Lemma~\ref{L:DeltaBreathForFixedCovering}, we obtain that $\brd(\DD^p)$ is odd and that $\brd(\DD^{pt})$ has the same parity as $t$.
    We get
    \[
    \brd(\DD^{pt})-2 =t\times (\brd(\DD^p)-2) = t\times (2\,\dpt(\DD^p) - 1)
    \]
    together with 
    \[
    \brd(\DD^{pt}) -2 = \begin{cases}
    2\,\dpt(\DD^{pt}) -2 & \text{if $t$ is even,}\\
    2\,\dpt(\DD^{pt}) -1 & \text{if $t$ is odd.}
    \end{cases}
    \qedhere
    \]
      \end{proof}
    
    \begin{cor}\label{C:ConditionAIff}
   
    Recall that $\DD^p$ stabilizes $(S_2,S_1)$. Assume  $\SLDA{S}{\DD^p}{*} = \Phi^p(S_2\boxminus S_1)$. 

    The  covering $(S_2, S_1)$  satisfies Condition~$A$ with respect to $\DD^p$ if and only if $\DD^{p}$ fixes  $(S_2,S_1)$  .
    \end{cor}
    
    \begin{proof}
    The \emph{if} part is Proposition~\ref{P:LeftOfPowerDeltaBoxMinus} $(i)$.
    Assume $\DD^{p}$ does not fix the covering $(S_2,S_1)$. By Proposition~\ref{P:LeftOfPowerDeltaBoxMinus} $(ii)$, we have
    \[
    \dpt((\DD^p)^t)=t\times (\dpt(\DD^p) -1) +
    \begin{cases}
     \frac{t}2 & \text{if $t$ is even},\\
     \frac{t-1}2 & \text{if $t$ is odd},
    \end{cases}
    \]
    and so $(S_2, S_1)$ does not satisfy Condition~$A$ with respect to $\DD^p$.
    \end{proof}

    \subsection{Results on Condition~A}
    
    The case of $\Gamma$ of type $I_2(m)$ is treated in \cite{ArP2019}, so we focus on the other types.  We recall that $(S_2,S_1)$ is assumed to be $\DD^p$-regular for some fixed $p\geq$.

    As a special case of Corollary~\ref{C:ConditionAIff} we have:
      
    \begin{prp}
    \label{P:ConditionASingleLeft}
    Assume  $S_2\boxminus S_1$ is a singleton.
    Then $(S_2,S_1)$ satisfies Condition~$A$ with respect to $\DD^p$ if and only if $(S_2,S_1)$ is $\DD^p$-fixed.
    \end{prp}
    
    \begin{proof}
    We have $\emptyset \not=\SLDA{S}{\DD^p}{*} \subseteq \Phi^p(S_2\boxminus S_1)$ by Lemma~\ref{L:LeftPowerDelta}. 
    As by hypothesis the latter is a singleton we deduce that $\SLDA{S}{\DD}{*}=\Phi(S_2\boxminus S_1)$.
    We conclude using Corollary~\ref{C:ConditionAIff}.
    \end{proof}

    \begin{cor}
    \label{C:ConditionA:GammaLine}
 
    \begin{enumerate}
    \item  Assume  $\Gamma$ is of type $B_r, F_4, H_3$ or $H_4$. Then,  $(S_2,S_1)$ satisfies Condition~$A$ with respect to $\DD^p$ whatever $p\geq 1$.
    \item Assume  $\Gamma$ is of type $A_r$, Then,  $(S_2,S_1)$ satisfies Condition~$A$ with respect to $\DD^p$ iff $p$ is even.  
    \end{enumerate}
    \end{cor}
    
    \begin{proof}
    By the  hypothesis, the graph $\Gamma$, $\Gamma(S_1)$ and $\Gamma(S_2)$ are lines and so $S_2\boxminus S_1$ contains at most two elements: thus outside of $S_1$ and each of these element is connected to one of the two extremities of $\Gamma_1$.
    Since $(S_2,S_1)$ is a proper covering, one of the extermities of $S_1$ is an extremity of $S$.
    It follows that $S_2\boxminus S_1$ is the singleton $\{s\}$ where $s$ is the unique element $s\in S\setminus S_2$  connected in $\Gamma$ to an element of $S_2$.
    By Proposition~\ref{P:ConditionASingleLeft}, the covering $(S_2,S_1)$ satisfies Condition A with respect to $\DD^p$ if and only if $\Phi^p(s)=s$.
    If $\DD^p$ is central we necessarily have $\Phi^p(s)=s$ and so Consition A is satisifed with respect to $\DD^p$ in this case.
    As recall in Proposition~\ref{P:DeltaCentral}, under our hypothesis, the only possibility for $\DD^p$ to not be central is whenever $p$ is odd and $\Gamma$ of type~$A_r$.
    Assume that $\Gamma=A_r$ for some $r$ and that $p$ is odd.
    Write $s=\sig\kk$.
    Then $\Phi^p(s)=s$ holds if and only~$r$ is odd and $k$ is equal to $(r-1)/2$.
    Assume $r$ is odd and $s=\sig{(r-1)/2}$.
    The set $S_1$ is then equal to~$\{\sig1,\ldots,\sig{k-1}\}$ or to $\{\sig{k+1},\ldots,\sig{r-1}\}$.
    Since $(S_2,S_1)$ is $\DD^p$ stabilized, but not $\DD^p$ fixed, $S_2$ is the other set. 
    It follows that $s$ does not belong to $S_1\cup S_2$, which is impossible.
    \end{proof}
    
    It remains to deal the cases of $\Gamma =D_r, E_6, E_7$ and $E_8$.
    
    \subsection{Case of $D_r$} 
    Consider $M$ is of type $D_r$, for $r\geq 4$, with the notations of Figure~\ref{F:CoxDiag}. 
    Recall (\cite{BrS}) that the minimal Garside element of $M$ is 
    \begin{equation}
    \label{E:DeltaD}
    \DD_{D_r} = \sig0\sig0'\cdot (\sig1\cdot\sig0\sig0'\cdot\sig1)\cdot \ldots \cdot (\sig{r-2}\ldots \sig1\cdot \sig0\sig0'\cdot \sig1\ldots \sig{r-2})
    \end{equation}
    \begin{lm}
    \label{L:ConditionA:TypeD}
    Assume $\Gamma$ is of type $D_r$ for $r\geq 4$.
    Except in the following cases, the set $S_2\boxminus S_1$ is a singleton :
    \begin{enumerate}
        \item $S_1 = \{\sig1,\ldots,\sig{r-2}\}$ and $S_2 = \{\sig0,\sig0',\sig1,\ldots,\sig\jj\}$ with $1\leq j\leq r-3$,
        \item $r$ is even and  $S_1 = \{\sig0, \sig1,\ldots, \sig\ii\}$ and $S_2 = \{\sig0',\sig1,\ldots,\sig{r-2}\}$ with $1\leq \ii \leq r-3$,
        \item $r$ is even and  $S_1 = \{\sig0', \sig1,\ldots, \sig\ii\}$ and $S_2 = \{\sig0,\sig1,\ldots,\sig{r-2}\}$ with $1\leq \ii \leq r-3$.
    \end{enumerate}
    \end{lm} 
    
    \begin{proof}
    The generators $\sig0,\sig0'$ and $\sig{r-2}$ will play a key role.  
    By $X$ we denote the set~$\{\sig0,\sig0',\sig{r-2}\}$. 
    Since~$(S_2,S_1)$ is a proper covering and $\Gamma(S_1)$ and $\Gamma(S_2)$ are connected, the only possibilities for $S_i\cap X$ are
    \[
    \{\sig0\},\{\sig0'\},\{\sig{r-2}\},\{\sig0,\sig0'\},\{\sig0,\sig{r-2}\},\{\sig0',\sig{r-2}\},
    \]
    with the additional constraint $(S_1 \cap X) \cup (S_2\cap X) = X$.
    All posibilities are listed in the following table. 
    Bold elements are those belonging to $S_i\cap X$, the columns '$\DD$-stab. if $r$ odd ?' specifies if the covering is $\DD$-stabilized whenever $r$ is odd.
    {\footnotesize
    \[
    \begin{array}{|c|c|c|c|c|}
     \hline
    S_1  & S_2  & \text{conditions}  & S_2\boxminus S_1 & \text{$\DD$-stab. if $r$ odd ?}\\
    \hline
    \bm{\sig0}& \bm{\sig0'},\sig1,\ldots,\bm{\sig{r-2}} &  & \sig1  & \text{no}\\
    \hline
    \bm{\sig0},\sig1,\ldots,\sig\ii & \bm{\sig0'},\sig1,\ldots,\bm{\sig{r-2}} & 1\leq \ii \leq r-3  & \sig0', \sig{i+1} & \text{no}\\
    \hline
    \bm{\sig0'}& \bm{\sig0},\sig1,\ldots,\bm{\sig{r-2}} &   & \sig1 & \text{no}\\
    \hline
    \bm{\sig0'},\sig1,\ldots,\sig\ii & \bm{\sig0},\sig1,\ldots,\bm{\sig{r-2}} & 1\leq \ii \leq r-3  & \sig0, \sig{i+1} & \text{no}\\
    \hline 
    \sig1,\ldots,\bm{\sig{r-2}} & \bm{\sig0}, \bm{\sig0'}, \sig1, \ldots ,\sig\jj & 1  \leq j \leq r-3  & \sig0,\sig0' & \text{yes}\\
    
    \hline 
    \sig\ii,\ldots,\bm{\sig{r-2}} & \bm{\sig0}, \bm{\sig0'}, \sig1, \ldots ,\sig\jj & 1 \leq i -1 \leq j \leq r-3  & \sig\iio & \text{yes}\\
    \hline
    \bm{\sig0},\bm{\sig0'}, \sig1, \ldots, \sig\ii & \sig\jj, \ldots, \bm{\sig{r-2}} & 0 \leq j-1 \leq i \leq r-3  & \sig\iip & \text{yes}\\
    \hline
    \bm{\sig0},\bm{\sig0'}, \sig1, \ldots, \sig\ii & \bm{\sig0}, \sig1,\ldots, \bm{\sig{r-2}} & 1 \leq i \leq r-3 & \sig\iip & \text{no}\\
    \hline
    \bm{\sig0},\bm{\sig0'}, \sig1, \ldots, \sig\ii & \bm{\sig0'}, \sig1,\ldots, \bm{\sig{r-2}} & 1 \leq i \leq r-3  & \sig\iip & \text{no}\\
    \hline
    \bm{\sig0}, \sig1, \ldots, \bm{\sig{r-2}} & \bm{\sig0'}, \sig1, \ldots, \sig\jj & 1\leq j \leq r-3  & \sig0' & \text{no} \\
    \hline
    \bm{\sig0}, \sig1, \ldots, \bm{\sig{r-2}} & \bm{\sig0'}, \bm{\sig0}, \sig1, \ldots, \sig\jj & 1\leq j \leq r-3 & \sig0' & \text{no}\\
    \hline
    \bm{\sig0}, \sig1, \ldots, \bm{\sig{r-2}} & \bm{\sig0'}, \sig1, \ldots, \bm{\sig{r-2}} &  & \sig0' & \text{yes}\\
    \hline
    \bm{\sig0'}, \sig1, \ldots, \bm{\sig{r-2}} & \bm{\sig0}, \sig1, \ldots, \sig\jj & 1\leq j \leq r-3  & \sig0 & \text{no} \\
    \hline
    \bm{\sig0'}, \sig1, \ldots, \bm{\sig{r-2}} & \bm{\sig0'}, \bm{\sig0}, \sig1, \ldots, \sig\jj & 1\leq j \leq r-3  & \sig0 & \text{no}\\
    \hline
    \bm{\sig0'}, \sig1, \ldots, \bm{\sig{r-2}} & \bm{\sig0}, \sig1, \ldots, \bm{\sig{r-2}} &  & \sig0 & \text{yes} \\
    \hline
    \end{array}
    \]
    }
    \qedhere
    \end{proof}
    
    \begin{prp}
    \label{P:ConditionA:TypeD}
    Assume $\Gamma$ is of type $D_r$ for $r\geq 4$.
    Except perhaps for cases $(ii)$ and~$(iii)$ of Lemma~\ref{L:ConditionA:TypeD}, we have $\SLDA{S}{\DD^p}{*} = \Phi^p(S_2\boxminus S_1)$.
    \end{prp}
    
    \begin{proof}
    Assume $(S_2,S_1)$ is not one of the exceptional coverings listed in Lemma~\ref{L:ConditionA:TypeD}.
    In this case, Lemma~\ref{L:ConditionA:TypeD} implies that $S_2\boxminus S_1$ is a singleton. 
    From $\emptyset \not=\SLDA{S}{\DD^p}{*} \subseteq \Phi^p(S_2\boxminus S_1)$, we obtain $\SLDA{S}{\DD^p}{*} = \Phi^p(S_1\boxminus S_1)$. 
    Assume $S_1 = \{\sig1,\ldots,\sig{r-2}\}$ and $S_2=\{\sig0,\sig0',\sig1,\ldots,\sig\jj\}$ with $1\leq j \leq r-3$.
    Let $\eta$ be the word homomorphism exchanging the letters $\sig0$ and $\sig0'$ that keep fixed the other letters. 
    By symmetry of the graph~$\Gamma_{D_r}$, $\eta$ leads to a morphism on $M$.
    Put $[\DD^p]_{(S_2,S_1)}= (g_n, \ldots, g_1)$ the alternating normal form of $\DD^p$ relatively to $(S_2,S_1)$.
    Since $S_2$ and $S_1$ are fixed by~$\eta$, we have 
    \[
    [\eta(\DD^p)]_{(S_2,S_1)}=[\eta(\DD^p)]_{(\eta(S_2),\eta(S_1))}=(\eta(g_n),\ldots,\eta(g_1)).
    \]
    Since $\DD$ is the lcm of all generators, we have $\eta(\DD) = \DD$.
    It follows $\eta(\DD^p)=\DD^p$ and thus we obtain $\eta(g_n) = g_n$. 
    In particular $\SLDA{S}{\DD^p}{*}\cap \{\sig0,\sig0'\}$ is empty or cardinal two.
    From 
    \[
    \emptyset \not = \SLDA{S}{\DD^p}{*} \subseteq \Phi^p(S_2\boxminus S_1) = \Phi^p(\{\sig0,\sig0'\}) = \{\sig0,\sig0'\},
     \]
     we obtain $\SLDA{S}{\DD^p}{*}=\{\sig0,\sig0'\}=\Phi^p(S_2\boxminus S_1)$.
    Note that a similar argument is impossible for case $(ii)$ and $(iii)$ of Lemma~\ref{L:ConditionA:TypeD} since $S_1$ and $S_2$ are not stable under $\eta$.
    \end{proof}
    
    \begin{cor}
    	Assume $\Gamma$ is of type $D_r$ for $r\geq 4$.
    	\begin{enumerate}
    		\item For $r$ even, except perhaps for cases $(ii)$ and $(iii)$ of Lemma~\ref{L:ConditionA:TypeD}, Condition A is satified with respect to $\DD^p$ for any proper covering $(S_2,S_1)$.
    		\item For $r$ odd and $p$ even, except possibly for cases $(ii)$ and $(iii)$ of Lemma~\ref{L:ConditionA:TypeD}, Condition A is satisfied with respect to $\DD^p$ for any proper covering $(S_2,S_1)$.
    		\item For $r$ odd and $p$ odd, Condition A is satisfied if and only if
    		\begin{itemize}
    			\item $S_1 = \{\sig\ii,\ldots,\sig{r-2}\}$ and $S_2= \{\sig0,\sig0',\sig1,\ldots,\sig j\}$ with $0 \leq i -1 \leq j \leq r-3$,
    			\item $S_1= \{\sig0,\sig0',\sig1,\ldots,\sig i\}$ and $S_2 = \{\sig\jj,\ldots,\sig{r-2}\}$ with $0 \leq j-1 \leq i \leq r-3$.
    		\end{itemize}
    	\end{enumerate}
    	
    \end{cor}
    
    \begin{proof}
    Assume $\DD^p$ is central in $M$. By Propositions~\ref{P:Garside} and \ref{P:DeltaCentral},  either $r$ is even, or $r$ is odd with $p$ even.  
    Since $\DD^p$ central, any covering is $\DD^p$-regular,  and we are in position to apply Proposition~\ref{P:ConditionA:TypeD} and Corollary~\ref{C:ConditionAIff}. This proves case (i)  and (ii). Assume $\DD^p$ is not central in $M$. That is,  both $r$ and $p$ are odd. 
    	From the table given in proof of Lemma~\ref{L:ConditionA:TypeD} we obtain that the list of $\DD^p$-regular coverings $(S_2, S_1)$ is
    	{\footnotesize
    		\[
    		\begin{array}{|c|c|c|c|c|}
    			\hline
    			S_1  & S_2  & \text{conditions}  & S_2\boxminus S_1 \\
    			\hline
    			\sig1,\ldots,\sig{r-2} & \sig0, \sig0', \sig1, \ldots ,\sig\jj & 1  \leq j \leq r-3  & \sig0,\sig0' \\			
    			\hline 
    			\sig\ii,\ldots,\sig{r-2} & \sig0, \sig0', \sig1, \ldots ,\sig\jj & 1 \leq i -1 \leq j \leq r-3  & \sig\iio \\
    			\hline
    			\sig0,\sig0', \sig1, \ldots, \sig\ii & \sig\jj, \ldots, \sig{r-2} & 0 \leq j-1 \leq i \leq r-3  & \sig\iip \\
    			\hline
    			\sig0 \sig1, \ldots, \sig{r-2} & \sig0', \sig1, \ldots, \sig{r-2} &  & \sig0' \\
    			\hline
    			\sig0', \sig1, \ldots, \sig{r-2} & \sig0, \sig1, \ldots, \sig{r-2} &  & \sig0  \\
    			\hline
    		\end{array}
    		\]
    	}
From Corollary~\ref{C:ConditionAIff} we know that $(S_2,S_1)$ satisfies Condition A for $\DD^p$ if and only if $\DD^p$ fixes $(S_2,S_1)$.
    	Since $\DD^p$ swaps $\sig0$ and $\sig0'$ and fixes the other generators, the coverings on rows 4 and 5 do not satisfy Condition~A, whereas the overings on rows 1, 2 or 3  satisfy Condition~A. This provides the expected result.
    \end{proof}
    	
For a regular covering in type $D$, the only remaining open cases for Condition A  are the cases $(ii)$ and $(iii)$ in Lemma~\ref{L:ConditionA:TypeD}. So, up to exchanging $\sig0$ and $\sig0$, it remains only one case.  
 Computer assisted experiments suggest: 
    \begin{conj}
    Assume $r$ is even, $\Gamma$ of type $D_r$, $S_1=\{\sig0',\sig1,\ldots,\sig\ii\}$ with $1\leq i \leq r-3$ and $S_2=\{\sig0,\sig1,\ldots,\sig{r-2}\}$.
    We have 
    
    \begin{itemize}
    \item for $i = 1$,   $\SLDA{S}{\DD}{*} = \{\sig0,\sig2\} = \Phi(S_2\boxminus S_1)$ and, therefore,  Condition A is satisfied whaterver $p\geq 1$. 
    \item for $i\geq 2$, $\SLDA{S}{\DD}{*} = \{\sig0\}\subsetneq \Phi(S_2\boxminus S_1) = \{\sig0, \sig\iip\}$.
    \end{itemize}
    We also conjecture that Condition A is also satisfied for $i\geq 2$ (Corollary~\ref{C:ConditionAIff} cannot be applied here).
    \end{conj}
    
    \subsection{Case of type $E_6$}
    
    Consider $M$ is of type $E_6$ with the notations of Figure~\ref{F:CoxDiag}.
    An explicit enumeration gives that there is $54$ proper and irreducible coverings $(S_2,S_1)$, from which 10 are regular.
    For $42$ coverings, from which 8 are regular, the set~$S_2\boxminus S_1$ is a singleton and we can apply Proposition~\ref{P:ConditionASingleLeft}.  Among the $12$ remaining coverings, $4$ satisfy $\SLDA{S}{\DD}{*}=\Phi(S_2\boxminus S_1)$, from which the $2$ remaining regular coverings, and we can apply Corollary~\ref{C:ConditionAIff} to treat them.  As a consequence, the $10$ regular coverings satisfy Condition $A$ for $\Delta^p$ whatever $p\geq 1$ and the $36$ other ones  satisfy Condition~A with respect to $\DD^p$ for $p$ even. 
    
    It remains the $8$ following coverings with an open status and for which Condition A seem to be satisfied with respect to $\DD^2$.
    \[
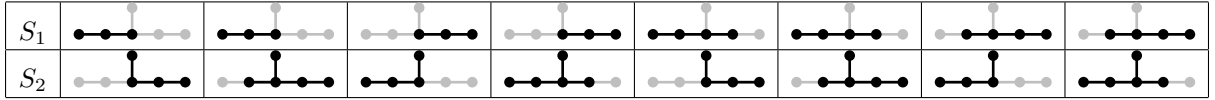
    \begin{array}{|c|c|c|c|c|c|c|c|c|}
    \hline
    S_1 
    & 	
    \begin{tikzpicture}[x=10pt,y=10pt]
    	\fill[lightgray](3,0) circle (0.2);
    	\fill[lightgray](4,0) circle (0.2);
    	\fill[lightgray](2,1) circle (0.2);
    	\draw[lightgray, line width=1](0,0) -- (4,0);
    	\draw[lightgray, line width=1](2,0) -- (2,1);
    	\draw[line width=1](0,0) -- (2,0);
    	\fill(0,0) circle (0.2);
    	\fill(1,0) circle (0.2);
    	\fill(2,0) circle (0.2);
    \end{tikzpicture}
    &
    \begin{tikzpicture}[x=10pt,y=10pt]
    	\fill[lightgray](3,0) circle (0.2);
    	\fill[lightgray](4,0) circle (0.2);
    	\fill[lightgray](2,1) circle (0.2);
    	\draw[lightgray, line width=1](0,0) -- (4,0);
    	\draw[lightgray, line width=1](2,0) -- (2,1);
    	\draw[line width=1](0,0) -- (2,0);
    	\fill(0,0) circle (0.2);
    	\fill(1,0) circle (0.2);
    	\fill(2,0) circle (0.2);
    \end{tikzpicture}
    &
    \begin{tikzpicture}[x=10pt,y=10pt]
    	\fill[lightgray](0,0) circle (0.2);
    	\fill[lightgray](1,0) circle (0.2);
    	\fill[lightgray](2,1) circle (0.2);
    	\draw[lightgray, line width=1](0,0) -- (4,0);
    	\draw[lightgray, line width=1](2,0) -- (2,1);
    	\draw[line width=1](2,0) -- (4,0);
    	\fill(2,0) circle (0.2);
    	\fill(3,0) circle (0.2);
    	\fill(4,0) circle (0.2);
    \end{tikzpicture}
    &
    \begin{tikzpicture}[x=10pt,y=10pt]
    	\fill[lightgray](0,0) circle (0.2);
    	\fill[lightgray](1,0) circle (0.2);
    	\fill[lightgray](2,1) circle (0.2);
    	\draw[lightgray, line width=1](0,0) -- (4,0);
    	\draw[lightgray, line width=1](2,0) -- (2,1);
    	\draw[line width=1](2,0) -- (4,0);
    	\fill(2,0) circle (0.2);
    	\fill(3,0) circle (0.2);
    	\fill(4,0) circle (0.2);
    \end{tikzpicture}
    &
    \begin{tikzpicture}[x=10pt,y=10pt]
    	\fill[lightgray](4,0) circle (0.2);
    	\fill[lightgray](2,1) circle (0.2);
    	\draw[lightgray, line width=1](0,0) -- (4,0);
    	\draw[lightgray, line width=1](2,0) -- (2,1);
    	\draw[line width=1](0,0) -- (3,0);
    	\fill(0,0) circle (0.2);
    	\fill(1,0) circle (0.2);
    	\fill(2,0) circle (0.2);
    	\fill(3,0) circle (0.2);
    \end{tikzpicture}
    &
    \begin{tikzpicture}[x=10pt,y=10pt]
    	\fill[lightgray](4,0) circle (0.2);
    	\fill[lightgray](2,1) circle (0.2);
    	\draw[lightgray, line width=1](0,0) -- (4,0);
    	\draw[lightgray, line width=1](2,0) -- (2,1);
    	\draw[line width=1](0,0) -- (3,0);
    	\fill(0,0) circle (0.2);
    	\fill(1,0) circle (0.2);
    	\fill(2,0) circle (0.2);
    	\fill(3,0) circle (0.2);
    \end{tikzpicture}
    &
    \begin{tikzpicture}[x=10pt,y=10pt]
    	\fill[lightgray](0,0) circle (0.2);
    	\fill[lightgray](2,1) circle (0.2);
    	\draw[lightgray, line width=1](0,0) -- (4,0);
    	\draw[lightgray, line width=1](2,0) -- (2,1);
    	\draw[line width=1](1,0) -- (4,0);
    	\fill(1,0) circle (0.2);
    	\fill(2,0) circle (0.2);
    	\fill(3,0) circle (0.2);
    	\fill(4,0) circle (0.2);	
    \end{tikzpicture}
    &
    \begin{tikzpicture}[x=10pt,y=10pt]
    	\fill[lightgray](0,0) circle (0.2);
    	\fill[lightgray](2,1) circle (0.2);
    	\draw[lightgray, line width=1](0,0) -- (4,0);
    	\draw[lightgray, line width=1](2,0) -- (2,1);
    	\draw[line width=1](1,0) -- (4,0);
    	\fill(1,0) circle (0.2);
    	\fill(2,0) circle (0.2);
    	\fill(3,0) circle (0.2);
    	\fill(4,0) circle (0.2);
    \end{tikzpicture}\\
    \hline
    S_2&
       \begin{tikzpicture}[x=10pt,y=10pt]
    	\fill[lightgray](0,0) circle (0.2);
    	\fill[lightgray](1,0) circle (0.2);
    	\draw[lightgray, line width=1](0,0) -- (4,0);
    	\draw[line width=1](2,0) -- (2,1);
    	\draw[line width=1](2,0) -- (4,0);
    	\fill(2,1) circle (0.2);
    	\fill(2,0) circle (0.2);
    	\fill(3,0) circle (0.2);
    	\fill(4,0) circle (0.2);
    \end{tikzpicture}
    	&
    \begin{tikzpicture}[x=10pt,y=10pt]
    	\fill[lightgray](0,0) circle (0.2);
    	\draw[lightgray, line width=1](0,0) -- (4,0);
    	\draw[line width=1](2,0) -- (2,1);
    	\draw[line width=1](1,0) -- (4,0);
    	\fill(2,1) circle (0.2);
    	\fill(1,0) circle (0.2);
    	\fill(2,0) circle (0.2);
    	\fill(3,0) circle (0.2);
    	\fill(4,0) circle (0.2);
    \end{tikzpicture}
    	&
    \begin{tikzpicture}[x=10pt,y=10pt]
    	\fill[lightgray](3,0) circle (0.2);
    	\fill[lightgray](4,0) circle (0.2);
    	\draw[lightgray, line width=1](0,0) -- (4,0);
    	\draw[line width=1](2,0) -- (2,1);
    	\draw[line width=1](0,0) -- (2,0);
    	\fill(2,1) circle (0.2);
    	\fill(0,0) circle (0.2);
    	\fill(1,0) circle (0.2);
    	\fill(2,0) circle (0.2);
    \end{tikzpicture}
    	&
    \begin{tikzpicture}[x=10pt,y=10pt]
    	\fill[lightgray](4,0) circle (0.2);
    	\draw[lightgray, line width=1](0,0) -- (4,0);
    	\draw[line width=1](2,0) -- (2,1);
    	\draw[line width=1](0,0) -- (3,0);
    	\fill(2,1) circle (0.2);
    	\fill(0,0) circle (0.2);
    	\fill(1,0) circle (0.2);
    	\fill(2,0) circle (0.2);
    	\fill(3,0) circle (0.2);
    \end{tikzpicture}
    	&
    \begin{tikzpicture}[x=10pt,y=10pt]
    	\fill[lightgray](0,0) circle (0.2);
    	\fill[lightgray](1,0) circle (0.2);
    	\draw[lightgray, line width=1](0,0) -- (4,0);
    	\draw[line width=1](2,0) -- (2,1);
    	\draw[line width=1](2,0) -- (4,0);
    	\fill(2,1) circle (0.2);
    	\fill(2,0) circle (0.2);
    	\fill(3,0) circle (0.2);
    	\fill(4,0) circle (0.2);
    \end{tikzpicture}
    	&
    \begin{tikzpicture}[x=10pt,y=10pt]
    	\fill[lightgray](0,0) circle (0.2);
    	\draw[lightgray, line width=1](0,0) -- (4,0);
    	\draw[line width=1](2,0) -- (2,1);
    	\draw[line width=1](1,0) -- (4,0);
    	\fill(1,0) circle (0.2);
    	\fill(2,1) circle (0.2);
    	\fill(2,0) circle (0.2);
    	\fill(3,0) circle (0.2);
    	\fill(4,0) circle (0.2);
    \end{tikzpicture}
    	&
    \begin{tikzpicture}[x=10pt,y=10pt]
    	\fill[lightgray](3,0) circle (0.2);
    	\fill[lightgray](4,0) circle (0.2);
    	\draw[lightgray, line width=1](0,0) -- (4,0);
    	\draw[line width=1](2,0) -- (2,1);
    	\draw[line width=1](0,0) -- (2,0);
    	\fill(2,1) circle (0.2);
    	\fill(0,0) circle (0.2);
    	\fill(1,0) circle (0.2);
    	\fill(2,0) circle (0.2);
    \end{tikzpicture}
    	&
    \begin{tikzpicture}[x=10pt,y=10pt]
    	\fill[lightgray](4,0) circle (0.2);
    	\draw[lightgray, line width=1](0,0) -- (4,0);
    	\draw[line width=1](2,0) -- (2,1);
    	\draw[line width=1](0,0) -- (3,0);
    	\fill(2,1) circle (0.2);
    	\fill(0,0) circle (0.2);
    	\fill(1,0) circle (0.2);
    	\fill(2,0) circle (0.2);
    	\fill(3,0) circle (0.2);
    \end{tikzpicture}\\
    \hline
    \end{array}	
    \]
    
    \subsection{Case of type $E_7$} Consider $M$ of type $E_7$ with the notations of Figure~\ref{F:CoxDiag}.
    Since $\DD$ is central in this case, any covering is $\DD$-stable.
    Using an explicit enumeration we obtain that there is $78$ proper and irreducible, and so regular,  coverings $(S_2,S_1)$.
  Among them,   exactly $60$  have the property that  $S_2\boxminus S_1$ is a singleton and so, by Proposition~\ref{P:ConditionASingleLeft},  satisfy Condition $A$ for $\DD^p$ whatever $p\geq 1$.
    Among the $18$ remaining coverings,~$11$ satisfy $\SLDA{S}{\DD}{*}=\Phi(S_2\boxminus S_1)$ and we can apply Corollary~\ref{C:ConditionAIff} to treat them.    It remains 7 coverings such that $\SLDA{S}{\DD}{*}\subsetneq\Phi(S_2\boxminus S_1)$, among which 
    $3$  do not satisfy Condition A with respect to $\DD$.
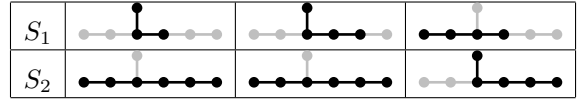
    \[
    \begin{array}{|c|c|c|c|}
    \hline
    	S_1&
    	\begin{tikzpicture}[x=10pt,y=10pt]
    		\draw[line width=1](2, 1) -- (2, 0);
    		\draw[lightgray, line width=1](0, 0) -- (1, 0);
    		\draw[lightgray, line width=1](1, 0) -- (2, 0);
    		\draw[line width=1](2, 0) -- (3, 0);
    		\draw[lightgray, line width=1](3, 0) -- (4, 0);
    		\draw[lightgray, line width=1](4, 0) -- (5, 0);
    		\fill(2, 1) circle (0.2);
    		\fill[lightgray](0, 0) circle (0.2);
    		\fill[lightgray](1, 0) circle (0.2);
    		\fill(2, 0) circle (0.2);
    		\fill(3, 0) circle (0.2);
    		\fill[lightgray](4, 0) circle (0.2);
    		\fill[lightgray](5, 0) circle (0.2);
    	\end{tikzpicture}
    	&
    	\begin{tikzpicture}[x=10pt,y=10pt]
    		\draw[line width=1](2, 1) -- (2, 0);
    		\draw[lightgray, line width=1](0, 0) -- (1, 0);
    		\draw[lightgray, line width=1](1, 0) -- (2, 0);
    		\draw[line width=1](2, 0) -- (3, 0);
    		\draw[line width=1](3, 0) -- (4, 0);
    		\draw[lightgray, line width=1](4, 0) -- (5, 0);
    		\fill(2, 1) circle (0.2);
    		\fill[lightgray](0, 0) circle (0.2);
    		\fill[lightgray](1, 0) circle (0.2);
    		\fill(2, 0) circle (0.2);
    		\fill(3, 0) circle (0.2);
    		\fill(4, 0) circle (0.2);
    		\fill[lightgray](5, 0) circle (0.2);
    	\end{tikzpicture}
    	&
    	\begin{tikzpicture}[x=10pt,y=10pt]
    		\draw[lightgray, line width=1](2, 1) -- (2, 0);
    		\draw[line width=1](0, 0) -- (1, 0);
    		\draw[line width=1](1, 0) -- (2, 0);
    		\draw[line width=1](2, 0) -- (3, 0);
    		\draw[lightgray, line width=1](3, 0) -- (4, 0);
    		\draw[lightgray, line width=1](4, 0) -- (5, 0);
    		\fill[lightgray](2, 1) circle (0.2);
    		\fill(0, 0) circle (0.2);
    		\fill(1, 0) circle (0.2);
    		\fill(2, 0) circle (0.2);
    		\fill(3, 0) circle (0.2);
    		\fill[lightgray](4, 0) circle (0.2);
    		\fill[lightgray](5, 0) circle (0.2);
    	\end{tikzpicture}
    	\\ \hline S_2&
    	\begin{tikzpicture}[x=10pt,y=10pt]
    		\draw[lightgray, line width=1](2, 1) -- (2, 0);
    		\draw[line width=1](0, 0) -- (1, 0);
    		\draw[line width=1](1, 0) -- (2, 0);
    		\draw[line width=1](2, 0) -- (3, 0);
    		\draw[line width=1](3, 0) -- (4, 0);
    		\draw[line width=1](4, 0) -- (5, 0);
    		\fill[lightgray](2, 1) circle (0.2);
    		\fill(0, 0) circle (0.2);
    		\fill(1, 0) circle (0.2);
    		\fill(2, 0) circle (0.2);
    		\fill(3, 0) circle (0.2);
    		\fill(4, 0) circle (0.2);
    		\fill(5, 0) circle (0.2);
    	\end{tikzpicture}
    	&
    	\begin{tikzpicture}[x=10pt,y=10pt]
    		\draw[lightgray, line width=1](2, 1) -- (2, 0);
    		\draw[line width=1](0, 0) -- (1, 0);
    		\draw[line width=1](1, 0) -- (2, 0);
    		\draw[line width=1](2, 0) -- (3, 0);
    		\draw[line width=1](3, 0) -- (4, 0);
    		\draw[line width=1](4, 0) -- (5, 0);
    		\fill[lightgray](2, 1) circle (0.2);
    		\fill(0, 0) circle (0.2);
    		\fill(1, 0) circle (0.2);
    		\fill(2, 0) circle (0.2);
    		\fill(3, 0) circle (0.2);
    		\fill(4, 0) circle (0.2);
    		\fill(5, 0) circle (0.2);
    	\end{tikzpicture}
    	&
    	\begin{tikzpicture}[x=10pt,y=10pt]
    		\draw[line width=1](2, 1) -- (2, 0);
    		\draw[lightgray, line width=1](0, 0) -- (1, 0);
    		\draw[lightgray, line width=1](1, 0) -- (2, 0);
    		\draw[line width=1](2, 0) -- (3, 0);
    		\draw[line width=1](3, 0) -- (4, 0);
    		\draw[line width=1](4, 0) -- (5, 0);
    		\fill(2, 1) circle (0.2);
    		\fill[lightgray](0, 0) circle (0.2);
    		\fill[lightgray](1, 0) circle (0.2);
    		\fill(2, 0) circle (0.2);
    		\fill(3, 0) circle (0.2);
    		\fill(4, 0) circle (0.2);
    		\fill(5, 0) circle (0.2);
    	\end{tikzpicture}\\
    	\hline
    \end{array}
    \]
    Note that the second covering does not satisfy Condition A for $\DD^2$ too.
    The remaing four coverings below have an open status for $\DD$ and for $\DD^2$, regarding  Condition A. It seems for all of them that Condition A is satisfied with respect to $\DD$ and to $\DD^2$. 
    \[
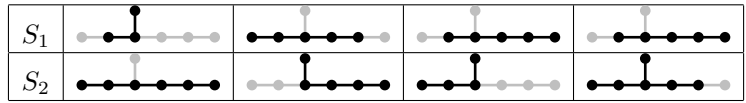
    \begin{array}{|c|c|c|c|c|}
    	\hline
    	S_1&
    	\begin{tikzpicture}[x=10pt,y=10pt]
    		\draw[line width=1](2, 1) -- (2, 0);
    		\draw[lightgray, line width=1](0, 0) -- (1, 0);
    		\draw[line width=1](1, 0) -- (2, 0);
    		\draw[lightgray, line width=1](2, 0) -- (3, 0);
    		\draw[lightgray, line width=1](3, 0) -- (4, 0);
    		\draw[lightgray, line width=1](4, 0) -- (5, 0);
    		\fill(2, 1) circle (0.2);
    		\fill[lightgray](0, 0) circle (0.2);
    		\fill(1, 0) circle (0.2);
    		\fill(2, 0) circle (0.2);
    		\fill[lightgray](3, 0) circle (0.2);
    		\fill[lightgray](4, 0) circle (0.2);
    		\fill[lightgray](5, 0) circle (0.2);
    	\end{tikzpicture}
    	&
    	\begin{tikzpicture}[x=10pt,y=10pt]
    		\draw[lightgray, line width=1](2, 1) -- (2, 0);
    		\draw[line width=1](0, 0) -- (1, 0);
    		\draw[line width=1](1, 0) -- (2, 0);
    		\draw[line width=1](2, 0) -- (3, 0);
    		\draw[line width=1](3, 0) -- (4, 0);
    		\draw[lightgray, line width=1](4, 0) -- (5, 0);
    		\fill[lightgray](2, 1) circle (0.2);
    		\fill(0, 0) circle (0.2);
    		\fill(1, 0) circle (0.2);
    		\fill(2, 0) circle (0.2);
    		\fill(3, 0) circle (0.2);
    		\fill(4, 0) circle (0.2);
    		\fill[lightgray](5, 0) circle (0.2);
    	\end{tikzpicture}
    	&
    	\begin{tikzpicture}[x=10pt,y=10pt]
    		\draw[lightgray, line width=1](2, 1) -- (2, 0);
    		\draw[lightgray, line width=1](0, 0) -- (1, 0);
    		\draw[line width=1](1, 0) -- (2, 0);
    		\draw[line width=1](2, 0) -- (3, 0);
    		\draw[line width=1](3, 0) -- (4, 0);
    		\draw[line width=1](4, 0) -- (5, 0);
    		\fill[lightgray](2, 1) circle (0.2);
    		\fill[lightgray](0, 0) circle (0.2);
    		\fill(1, 0) circle (0.2);
    		\fill(2, 0) circle (0.2);
    		\fill(3, 0) circle (0.2);
    		\fill(4, 0) circle (0.2);
    		\fill(5, 0) circle (0.2);
    	\end{tikzpicture}
    	&
    	\begin{tikzpicture}[x=10pt,y=10pt]
    		\draw[lightgray, line width=1](2, 1) -- (2, 0);
    		\draw[lightgray, line width=1](0, 0) -- (1, 0);
    		\draw[line width=1](1, 0) -- (2, 0);
    		\draw[line width=1](2, 0) -- (3, 0);
    		\draw[line width=1](3, 0) -- (4, 0);
    		\draw[line width=1](4, 0) -- (5, 0);
    		\fill[lightgray](2, 1) circle (0.2);
    		\fill[lightgray](0, 0) circle (0.2);
    		\fill(1, 0) circle (0.2);
    		\fill(2, 0) circle (0.2);
    		\fill(3, 0) circle (0.2);
    		\fill(4, 0) circle (0.2);
    		\fill(5, 0) circle (0.2);
    	\end{tikzpicture}
    	\\ \hline S_2&
    	\begin{tikzpicture}[x=10pt,y=10pt]
    		\draw[lightgray, line width=1](2, 1) -- (2, 0);
    		\draw[line width=1](0, 0) -- (1, 0);
    		\draw[line width=1](1, 0) -- (2, 0);
    		\draw[line width=1](2, 0) -- (3, 0);
    		\draw[line width=1](3, 0) -- (4, 0);
    		\draw[line width=1](4, 0) -- (5, 0);
    		\fill[lightgray](2, 1) circle (0.2);
    		\fill(0, 0) circle (0.2);
    		\fill(1, 0) circle (0.2);
    		\fill(2, 0) circle (0.2);
    		\fill(3, 0) circle (0.2);
    		\fill(4, 0) circle (0.2);
    		\fill(5, 0) circle (0.2);
    	\end{tikzpicture}
    	&
    	\begin{tikzpicture}[x=10pt,y=10pt]
    		\draw[line width=1](2, 1) -- (2, 0);
    		\draw[lightgray, line width=1](0, 0) -- (1, 0);
    		\draw[lightgray, line width=1](1, 0) -- (2, 0);
    		\draw[line width=1](2, 0) -- (3, 0);
    		\draw[line width=1](3, 0) -- (4, 0);
    		\draw[line width=1](4, 0) -- (5, 0);
    		\fill(2, 1) circle (0.2);
    		\fill[lightgray](0, 0) circle (0.2);
    		\fill[lightgray](1, 0) circle (0.2);
    		\fill(2, 0) circle (0.2);
    		\fill(3, 0) circle (0.2);
    		\fill(4, 0) circle (0.2);
    		\fill(5, 0) circle (0.2);
    	\end{tikzpicture}
    	&
    	\begin{tikzpicture}[x=10pt,y=10pt]
    		\draw[line width=1](2, 1) -- (2, 0);
    		\draw[line width=1](0, 0) -- (1, 0);
    		\draw[line width=1](1, 0) -- (2, 0);
    		\draw[lightgray, line width=1](2, 0) -- (3, 0);
    		\draw[lightgray, line width=1](3, 0) -- (4, 0);
    		\draw[lightgray, line width=1](4, 0) -- (5, 0);
    		\fill(2, 1) circle (0.2);
    		\fill(0, 0) circle (0.2);
    		\fill(1, 0) circle (0.2);
    		\fill(2, 0) circle (0.2);
    		\fill[lightgray](3, 0) circle (0.2);
    		\fill[lightgray](4, 0) circle (0.2);
    		\fill[lightgray](5, 0) circle (0.2);
    	\end{tikzpicture}
    	&
    	\begin{tikzpicture}[x=10pt,y=10pt]
    		\draw[line width=1](2, 1) -- (2, 0);
    		\draw[line width=1](0, 0) -- (1, 0);
    		\draw[line width=1](1, 0) -- (2, 0);
    		\draw[line width=1](2, 0) -- (3, 0);
    		\draw[line width=1](3, 0) -- (4, 0);
    		\draw[lightgray, line width=1](4, 0) -- (5, 0);
    		\fill(2, 1) circle (0.2);
    		\fill(0, 0) circle (0.2);
    		\fill(1, 0) circle (0.2);
    		\fill(2, 0) circle (0.2);
    		\fill(3, 0) circle (0.2);
    		\fill(4, 0) circle (0.2);
    		\fill[lightgray](5, 0) circle (0.2);
    	\end{tikzpicture}\\
    	\hline
    \end{array}
    \].
    
    \subsection{Case of type $E_8$} Consider $M$ of type $E_8$ with the notations of Figure~\ref{F:CoxDiag}.
    Since $\DD$ is central in this case, any covering is $\DD$-stable.
    Using an explicit enumeration we obtain there is $104$ proper and irreducible, and so regular, coverings $(S_2,S_1)$.
    Among them, exactly $80$ have the property that $S_2\boxminus S_1$ is a singleton and so, by Proposition~\ref{P:ConditionASingleLeft}, they satisfy Condition $A$ for $\DD^p$ for any $p\geq 1$.
    Among the $24$ remaining coverings, $14$  satisfy $\SLDA{S}{\DD}{*}=\Phi(S_2\boxminus S_1)$ and we can apply Corollary~\ref{C:ConditionAIff} to treat them.
  
    It remain 10 coverings such that $\SLDA{S}{\DD}{*}\subsetneq\Phi(S_2\boxminus S_1)$.
    Among  them $3$ do not satisfy Condition A with respect to $\DD$:
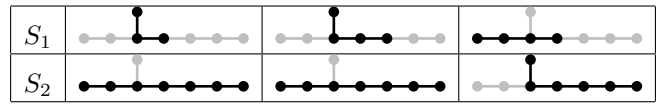
    \[
    \begin{array}{|c|c|c|c|}
    \hline
    S_1&
    \begin{tikzpicture}[x=10pt,y=10pt]
    \draw[line width=1](2, 1) -- (2, 0);
    \draw[lightgray, line width=1](0, 0) -- (1, 0);
    \draw[lightgray, line width=1](1, 0) -- (2, 0);
    \draw[line width=1](2, 0) -- (3, 0);
    \draw[lightgray, line width=1](3, 0) -- (4, 0);
    \draw[lightgray, line width=1](4, 0) -- (5, 0);
    \draw[lightgray, line width=1](5, 0) -- (6, 0);
    \fill(2, 1) circle (0.2);
    \fill[lightgray](0, 0) circle (0.2);
    \fill[lightgray](1, 0) circle (0.2);
    \fill(2, 0) circle (0.2);
    \fill(3, 0) circle (0.2);
    \fill[lightgray](4, 0) circle (0.2);
    \fill[lightgray](5, 0) circle (0.2);
    \fill[lightgray](6, 0) circle (0.2);
    \end{tikzpicture}
    &
    \begin{tikzpicture}[x=10pt,y=10pt]
    \draw[line width=1](2, 1) -- (2, 0);
    \draw[lightgray, line width=1](0, 0) -- (1, 0);
    \draw[lightgray, line width=1](1, 0) -- (2, 0);
    \draw[line width=1](2, 0) -- (3, 0);
    \draw[line width=1](3, 0) -- (4, 0);
    \draw[lightgray, line width=1](4, 0) -- (5, 0);
    \draw[lightgray, line width=1](5, 0) -- (6, 0);
    \fill(2, 1) circle (0.2);
    \fill[lightgray](0, 0) circle (0.2);
    \fill[lightgray](1, 0) circle (0.2);
    \fill(2, 0) circle (0.2);
    \fill(3, 0) circle (0.2);
    \fill(4, 0) circle (0.2);
    \fill[lightgray](5, 0) circle (0.2);
    \fill[lightgray](6, 0) circle (0.2);
    \end{tikzpicture}
    &
    \begin{tikzpicture}[x=10pt,y=10pt]
    \draw[lightgray, line width=1](2, 1) -- (2, 0);
    \draw[line width=1](0, 0) -- (1, 0);
    \draw[line width=1](1, 0) -- (2, 0);
    \draw[line width=1](2, 0) -- (3, 0);
    \draw[lightgray, line width=1](3, 0) -- (4, 0);
    \draw[lightgray, line width=1](4, 0) -- (5, 0);
    \draw[lightgray, line width=1](5, 0) -- (6, 0);
    \fill[lightgray](2, 1) circle (0.2);
    \fill(0, 0) circle (0.2);
    \fill(1, 0) circle (0.2);
    \fill(2, 0) circle (0.2);
    \fill(3, 0) circle (0.2);
    \fill[lightgray](4, 0) circle (0.2);
    \fill[lightgray](5, 0) circle (0.2);
    \fill[lightgray](6, 0) circle (0.2);
    \end{tikzpicture}
    \\ \hline S_2&
    \begin{tikzpicture}[x=10pt,y=10pt]
    \draw[lightgray, line width=1](2, 1) -- (2, 0);
    \draw[line width=1](0, 0) -- (1, 0);
    \draw[line width=1](1, 0) -- (2, 0);
    \draw[line width=1](2, 0) -- (3, 0);
    \draw[line width=1](3, 0) -- (4, 0);
    \draw[line width=1](4, 0) -- (5, 0);
    \draw[line width=1](5, 0) -- (6, 0);
    \fill[lightgray](2, 1) circle (0.2);
    \fill(0, 0) circle (0.2);
    \fill(1, 0) circle (0.2);
    \fill(2, 0) circle (0.2);
    \fill(3, 0) circle (0.2);
    \fill(4, 0) circle (0.2);
    \fill(5, 0) circle (0.2);
    \fill(6, 0) circle (0.2);
    \end{tikzpicture}
    &
    \begin{tikzpicture}[x=10pt,y=10pt]
    \draw[lightgray, line width=1](2, 1) -- (2, 0);
    \draw[line width=1](0, 0) -- (1, 0);
    \draw[line width=1](1, 0) -- (2, 0);
    \draw[line width=1](2, 0) -- (3, 0);
    \draw[line width=1](3, 0) -- (4, 0);
    \draw[line width=1](4, 0) -- (5, 0);
    \draw[line width=1](5, 0) -- (6, 0);
    \fill[lightgray](2, 1) circle (0.2);
    \fill(0, 0) circle (0.2);
    \fill(1, 0) circle (0.2);
    \fill(2, 0) circle (0.2);
    \fill(3, 0) circle (0.2);
    \fill(4, 0) circle (0.2);
    \fill(5, 0) circle (0.2);
    \fill(6, 0) circle (0.2);
    \end{tikzpicture}
    &
    \begin{tikzpicture}[x=10pt,y=10pt]
    \draw[line width=1](2, 1) -- (2, 0);
    \draw[lightgray, line width=1](0, 0) -- (1, 0);
    \draw[lightgray, line width=1](1, 0) -- (2, 0);
    \draw[line width=1](2, 0) -- (3, 0);
    \draw[line width=1](3, 0) -- (4, 0);
    \draw[line width=1](4, 0) -- (5, 0);
    \draw[line width=1](5, 0) -- (6, 0);
    \fill(2, 1) circle (0.2);
    \fill[lightgray](0, 0) circle (0.2);
    \fill[lightgray](1, 0) circle (0.2);
    \fill(2, 0) circle (0.2);
    \fill(3, 0) circle (0.2);
    \fill(4, 0) circle (0.2);
    \fill(5, 0) circle (0.2);
    \fill(6, 0) circle (0.2);
    \end{tikzpicture}
    \\
    \hline
    \end{array}
    \]
    The following 7 coverings have an open status. For all of them Condition A seems to be satisfied with respect to $\DD$ and so to $\DD^2$.
    \begin{align*}
    &\begin{array}{|c|c|c|c|c|c|}
    \hline
    S_1&
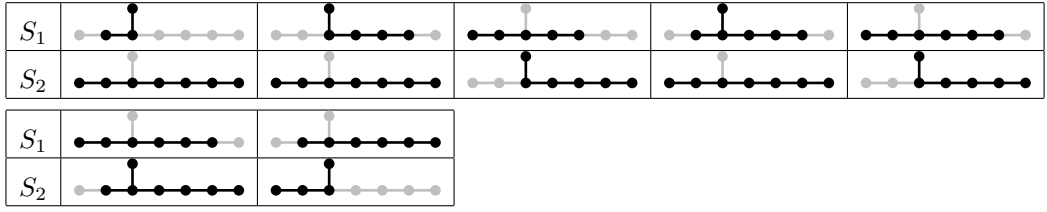
    \begin{tikzpicture}[x=10pt,y=10pt]
    \draw[line width=1](2, 1) -- (2, 0);
    \draw[lightgray, line width=1](0, 0) -- (1, 0);
    \draw[line width=1](1, 0) -- (2, 0);
    \draw[lightgray, line width=1](2, 0) -- (3, 0);
    \draw[lightgray, line width=1](3, 0) -- (4, 0);
    \draw[lightgray, line width=1](4, 0) -- (5, 0);
    \draw[lightgray, line width=1](5, 0) -- (6, 0);
    \fill(2, 1) circle (0.2);
    \fill[lightgray](0, 0) circle (0.2);
    \fill(1, 0) circle (0.2);
    \fill(2, 0) circle (0.2);
    \fill[lightgray](3, 0) circle (0.2);
    \fill[lightgray](4, 0) circle (0.2);
    \fill[lightgray](5, 0) circle (0.2);
    \fill[lightgray](6, 0) circle (0.2);
    \end{tikzpicture}
    &
    \begin{tikzpicture}[x=10pt,y=10pt]
    \draw[line width=1](2, 1) -- (2, 0);
    \draw[lightgray, line width=1](0, 0) -- (1, 0);
    \draw[lightgray, line width=1](1, 0) -- (2, 0);
    \draw[line width=1](2, 0) -- (3, 0);
    \draw[line width=1](3, 0) -- (4, 0);
    \draw[line width=1](4, 0) -- (5, 0);
    \draw[lightgray, line width=1](5, 0) -- (6, 0);
    \fill(2, 1) circle (0.2);
    \fill[lightgray](0, 0) circle (0.2);
    \fill[lightgray](1, 0) circle (0.2);
    \fill(2, 0) circle (0.2);
    \fill(3, 0) circle (0.2);
    \fill(4, 0) circle (0.2);
    \fill(5, 0) circle (0.2);
    \fill[lightgray](6, 0) circle (0.2);
    \end{tikzpicture}
    &
    \begin{tikzpicture}[x=10pt,y=10pt]
    \draw[lightgray, line width=1](2, 1) -- (2, 0);
    \draw[line width=1](0, 0) -- (1, 0);
    \draw[line width=1](1, 0) -- (2, 0);
    \draw[line width=1](2, 0) -- (3, 0);
    \draw[line width=1](3, 0) -- (4, 0);
    \draw[lightgray, line width=1](4, 0) -- (5, 0);
    \draw[lightgray, line width=1](5, 0) -- (6, 0);
    \fill[lightgray](2, 1) circle (0.2);
    \fill(0, 0) circle (0.2);
    \fill(1, 0) circle (0.2);
    \fill(2, 0) circle (0.2);
    \fill(3, 0) circle (0.2);
    \fill(4, 0) circle (0.2);
    \fill[lightgray](5, 0) circle (0.2);
    \fill[lightgray](6, 0) circle (0.2);
    \end{tikzpicture}
    &
    \begin{tikzpicture}[x=10pt,y=10pt]
    \draw[line width=1](2, 1) -- (2, 0);
    \draw[lightgray, line width=1](0, 0) -- (1, 0);
    \draw[line width=1](1, 0) -- (2, 0);
    \draw[line width=1](2, 0) -- (3, 0);
    \draw[line width=1](3, 0) -- (4, 0);
    \draw[line width=1](4, 0) -- (5, 0);
    \draw[lightgray, line width=1](5, 0) -- (6, 0);
    \fill(2, 1) circle (0.2);
    \fill[lightgray](0, 0) circle (0.2);
    \fill(1, 0) circle (0.2);
    \fill(2, 0) circle (0.2);
    \fill(3, 0) circle (0.2);
    \fill(4, 0) circle (0.2);
    \fill(5, 0) circle (0.2);
    \fill[lightgray](6, 0) circle (0.2);
    \end{tikzpicture}
    &
    \begin{tikzpicture}[x=10pt,y=10pt]
    \draw[lightgray, line width=1](2, 1) -- (2, 0);
    \draw[line width=1](0, 0) -- (1, 0);
    \draw[line width=1](1, 0) -- (2, 0);
    \draw[line width=1](2, 0) -- (3, 0);
    \draw[line width=1](3, 0) -- (4, 0);
    \draw[line width=1](4, 0) -- (5, 0);
    \draw[lightgray, line width=1](5, 0) -- (6, 0);
    \fill[lightgray](2, 1) circle (0.2);
    \fill(0, 0) circle (0.2);
    \fill(1, 0) circle (0.2);
    \fill(2, 0) circle (0.2);
    \fill(3, 0) circle (0.2);
    \fill(4, 0) circle (0.2);
    \fill(5, 0) circle (0.2);
    \fill[lightgray](6, 0) circle (0.2);
    \end{tikzpicture}
    \\ \hline S_2&
    \begin{tikzpicture}[x=10pt,y=10pt]
    \draw[lightgray, line width=1](2, 1) -- (2, 0);
    \draw[line width=1](0, 0) -- (1, 0);
    \draw[line width=1](1, 0) -- (2, 0);
    \draw[line width=1](2, 0) -- (3, 0);
    \draw[line width=1](3, 0) -- (4, 0);
    \draw[line width=1](4, 0) -- (5, 0);
    \draw[line width=1](5, 0) -- (6, 0);
    \fill[lightgray](2, 1) circle (0.2);
    \fill(0, 0) circle (0.2);
    \fill(1, 0) circle (0.2);
    \fill(2, 0) circle (0.2);
    \fill(3, 0) circle (0.2);
    \fill(4, 0) circle (0.2);
    \fill(5, 0) circle (0.2);
    \fill(6, 0) circle (0.2);
    \end{tikzpicture}
    &
    \begin{tikzpicture}[x=10pt,y=10pt]
    \draw[lightgray, line width=1](2, 1) -- (2, 0);
    \draw[line width=1](0, 0) -- (1, 0);
    \draw[line width=1](1, 0) -- (2, 0);
    \draw[line width=1](2, 0) -- (3, 0);
    \draw[line width=1](3, 0) -- (4, 0);
    \draw[line width=1](4, 0) -- (5, 0);
    \draw[line width=1](5, 0) -- (6, 0);
    \fill[lightgray](2, 1) circle (0.2);
    \fill(0, 0) circle (0.2);
    \fill(1, 0) circle (0.2);
    \fill(2, 0) circle (0.2);
    \fill(3, 0) circle (0.2);
    \fill(4, 0) circle (0.2);
    \fill(5, 0) circle (0.2);
    \fill(6, 0) circle (0.2);
    \end{tikzpicture}
    &
    \begin{tikzpicture}[x=10pt,y=10pt]
    \draw[line width=1](2, 1) -- (2, 0);
    \draw[lightgray, line width=1](0, 0) -- (1, 0);
    \draw[lightgray, line width=1](1, 0) -- (2, 0);
    \draw[line width=1](2, 0) -- (3, 0);
    \draw[line width=1](3, 0) -- (4, 0);
    \draw[line width=1](4, 0) -- (5, 0);
    \draw[line width=1](5, 0) -- (6, 0);
    \fill(2, 1) circle (0.2);
    \fill[lightgray](0, 0) circle (0.2);
    \fill[lightgray](1, 0) circle (0.2);
    \fill(2, 0) circle (0.2);
    \fill(3, 0) circle (0.2);
    \fill(4, 0) circle (0.2);
    \fill(5, 0) circle (0.2);
    \fill(6, 0) circle (0.2);
    \end{tikzpicture}
    &
    \begin{tikzpicture}[x=10pt,y=10pt]
    \draw[lightgray, line width=1](2, 1) -- (2, 0);
    \draw[line width=1](0, 0) -- (1, 0);
    \draw[line width=1](1, 0) -- (2, 0);
    \draw[line width=1](2, 0) -- (3, 0);
    \draw[line width=1](3, 0) -- (4, 0);
    \draw[line width=1](4, 0) -- (5, 0);
    \draw[line width=1](5, 0) -- (6, 0);
    \fill[lightgray](2, 1) circle (0.2);
    \fill(0, 0) circle (0.2);
    \fill(1, 0) circle (0.2);
    \fill(2, 0) circle (0.2);
    \fill(3, 0) circle (0.2);
    \fill(4, 0) circle (0.2);
    \fill(5, 0) circle (0.2);
    \fill(6, 0) circle (0.2);
    \end{tikzpicture}
    &
    \begin{tikzpicture}[x=10pt,y=10pt]
    \draw[line width=1](2, 1) -- (2, 0);
    \draw[lightgray, line width=1](0, 0) -- (1, 0);
    \draw[lightgray, line width=1](1, 0) -- (2, 0);
    \draw[line width=1](2, 0) -- (3, 0);
    \draw[line width=1](3, 0) -- (4, 0);
    \draw[line width=1](4, 0) -- (5, 0);
    \draw[line width=1](5, 0) -- (6, 0);
    \fill(2, 1) circle (0.2);
    \fill[lightgray](0, 0) circle (0.2);
    \fill[lightgray](1, 0) circle (0.2);
    \fill(2, 0) circle (0.2);
    \fill(3, 0) circle (0.2);
    \fill(4, 0) circle (0.2);
    \fill(5, 0) circle (0.2);
    \fill(6, 0) circle (0.2);
    \end{tikzpicture}\\
    \hline
    \end{array}\\
    &
    \begin{array}{|c|c|c|}
    \hline
    S_1&
    \begin{tikzpicture}[x=10pt,y=10pt]
    \draw[lightgray, line width=1](2, 1) -- (2, 0);
    \draw[line width=1](0, 0) -- (1, 0);
    \draw[line width=1](1, 0) -- (2, 0);
    \draw[line width=1](2, 0) -- (3, 0);
    \draw[line width=1](3, 0) -- (4, 0);
    \draw[line width=1](4, 0) -- (5, 0);
    \draw[lightgray, line width=1](5, 0) -- (6, 0);
    \fill[lightgray](2, 1) circle (0.2);
    \fill(0, 0) circle (0.2);
    \fill(1, 0) circle (0.2);
    \fill(2, 0) circle (0.2);
    \fill(3, 0) circle (0.2);
    \fill(4, 0) circle (0.2);
    \fill(5, 0) circle (0.2);
    \fill[lightgray](6, 0) circle (0.2);
    \end{tikzpicture}
    &
    \begin{tikzpicture}[x=10pt,y=10pt]
    \draw[lightgray, line width=1](2, 1) -- (2, 0);
    \draw[lightgray, line width=1](0, 0) -- (1, 0);
    \draw[line width=1](1, 0) -- (2, 0);
    \draw[line width=1](2, 0) -- (3, 0);
    \draw[line width=1](3, 0) -- (4, 0);
    \draw[line width=1](4, 0) -- (5, 0);
    \draw[line width=1](5, 0) -- (6, 0);
    \fill[lightgray](2, 1) circle (0.2);
    \fill[lightgray](0, 0) circle (0.2);
    \fill(1, 0) circle (0.2);
    \fill(2, 0) circle (0.2);
    \fill(3, 0) circle (0.2);
    \fill(4, 0) circle (0.2);
    \fill(5, 0) circle (0.2);
    \fill(6, 0) circle (0.2);
    \end{tikzpicture}
    \\ \hline S_2&
    \begin{tikzpicture}[x=10pt,y=10pt]
    \draw[line width=1](2, 1) -- (2, 0);
    \draw[lightgray, line width=1](0, 0) -- (1, 0);
    \draw[line width=1](1, 0) -- (2, 0);
    \draw[line width=1](2, 0) -- (3, 0);
    \draw[line width=1](3, 0) -- (4, 0);
    \draw[line width=1](4, 0) -- (5, 0);
    \draw[line width=1](5, 0) -- (6, 0);
    \fill(2, 1) circle (0.2);
    \fill[lightgray](0, 0) circle (0.2);
    \fill(1, 0) circle (0.2);
    \fill(2, 0) circle (0.2);
    \fill(3, 0) circle (0.2);
    \fill(4, 0) circle (0.2);
    \fill(5, 0) circle (0.2);
    \fill(6, 0) circle (0.2);
    \end{tikzpicture}
    &
    \begin{tikzpicture}[x=10pt,y=10pt]
    \draw[line width=1](2, 1) -- (2, 0);
    \draw[line width=1](0, 0) -- (1, 0);
    \draw[line width=1](1, 0) -- (2, 0);
    \draw[lightgray, line width=1](2, 0) -- (3, 0);
    \draw[lightgray, line width=1](3, 0) -- (4, 0);
    \draw[lightgray, line width=1](4, 0) -- (5, 0);
    \draw[lightgray, line width=1](5, 0) -- (6, 0);
    \fill(2, 1) circle (0.2);
    \fill(0, 0) circle (0.2);
    \fill(1, 0) circle (0.2);
    \fill(2, 0) circle (0.2);
    \fill[lightgray](3, 0) circle (0.2);
    \fill[lightgray](4, 0) circle (0.2);
    \fill[lightgray](5, 0) circle (0.2);
    \fill[lightgray](6, 0) circle (0.2);
    \end{tikzpicture}\\
    \hline
    \end{array}
    \end{align*}

    \section{Condition~$B$}
    
    \label{S:ConditionB}
    
    We recall Notation~\ref{N: cadre}.  During this section we fix the following notations:
    
     \begin{nota} We fix $p\geq 1$, an integer, typically $p=1$ or $p=2$. We assume that $(S_2,S_1)$ is  a proper and irreducible covering  of $S$. Contrary to Section~\ref{S:ConditionA}, except if this is explicitly specified, we do not assume the covering is necessarily $\DD^p$- stabilized. \label{N:cov proper}
 \end{nota}    
    
    \begin{df}
    We define
    \begin{itemize}
     \item $\Theta_{S_2,S_1}(p)$ to be the set of elements of $M$ of the form $(\DD^{p}\DD_{S_1}^{-p})^k h$ with $k\geq1$ and $h\in M_{S_1}$.
     \item $\oTheta_{S_2,S_1}(p)$ to be the set $\Theta_{S_2,S_1}(p)\cup M_{S_1}$.
    \end{itemize}
    \end{df}
    
    When no confusion is possible, we write $\Theta(p)$  and $\oTheta(p)$ for $\Theta_{S_2,S_1}(p)$  and $\oTheta_{S_2,S_1}(p)$, respectively. 
    In~\cite{ArP2019}, Condition~$B$ is stated for a  specific value of $p$, that is when $p$ is the minimal integer such that $\DD^p$ is central.
    Here, we give a more general definition of this condition in the context of $\Theta(p)$. We recall that $\DD$-unmovable  elements have been defined just after Notation~\ref{N: cadre}.

    \begin{df} \label{D:ConditionB}
    Consider $(a,b)$ in $M\times M\setminus (\oTheta(p)\times\oTheta(p))$ such that $a,b$ are both $\DD^p$-unmovable (see Definition~\ref{D:DeltaUnmovable}).
    Let $\DD^{pt} c$ be the $\DD^p$-form of the product~$ab$.
    We say  that the pair~$(a,b)$ satisfies Condition~$B(p)$ when
    \begin{enumerate}
    \item either $t \geq 1$ and  $\dpt(c)+\dpt(\DD^{pt}) =  \dpt(a)+\dpt(b)+ \varepsilon$ with $\varepsilon \in \{0,1\}$,
    \item or $t = 0$, that is $c = ab$, and  $\dpt(ab) =  \dpt(a)+\dpt(b) - 1 + \varepsilon $  with $\varepsilon \in \{0,1\}$,
    \item and moreover $\varepsilon = 0$ in $(i)$ or $(ii)$ whenever $a\in \Theta(p)$ or $b\in \Theta(p)$ or $c\in M_{S_1}$.
    \end{enumerate}
    
    We say that the covering $(S_2,S_1)$ satisfies Condition~$B(p)$ if each pair $(a,b)$ of $M\times M\setminus (\oTheta(p)\times\oTheta(p))$ satisfies Condition~$B(p)$.
    \end{df}

    \subsection{Condition B and two criteria}
    Let $T$ be a proper and irreducible subset of $S$. Put $T_i= S_i\cap T$ for $i = 1,2$, and assume $(T_2, T_1)$ is a proper covering of $T$. A natural question is whether Properties~$B(p)$ in $M$ and in $M_T$ are related. 
    This is not easy to answer the question because, in particular,  it is not clear whether $M_T\cap \oTheta_{S_2,S_1}(1) \subseteq \oTheta_{T_2,T_1}(1)$ in general. However, in Lemma~\ref{L:Subcovering} we provide a partial answer that we will use in the next subsection.
    
    \begin{lm}
    \label{L:B1ToBj}~
    
    \begin{enumerate}
    \item If $\DD$ is central, then $\oTheta(p)\subseteq \oTheta(1)$. In any case, $\oTheta(2p)\subseteq \oTheta(2)$.
    \item  If it exists $(a,b)\in M\times M\setminus (\oTheta(p)\times\oTheta(p))$ such that $ab$ is $\DD$-unmovable and
    \[
    \dpt(a)+\dpt(b) - \dpt(ab)\not\in \{0,1\}
    \]
    Then, $a,b$ and $ab$ are $\Delta^p$-unmovable,  and the pair $(a,b)$ does not satisfy Property~$(ii)$ of Defininition~\ref{D:ConditionB}.
    Therefore,  $(S_2,S_1)$ does not satisfy condition~$B(p)$. 
    \end{enumerate}
    \end{lm}
    \begin{proof}
    (i) Let $a$ be in $\oTheta(p)$.
    If $a$ lies in $M_{S_1}$, then it belongs to $\oTheta(q)$ for any $q$.
    So assume this is not the case.
    Write $a = (\DD^p\DD_{S_1}^{-p})^kh$ with $h$ in $M_{S_1}$.
    If $\DD$ is central, then $a = (\DD\DD_{S_1}^{-1})^{p k}h$, and $a$ belongs to~$\oTheta(1)$.
    For  $p = 2q$, since $\DD^2$ is central, we get $a = (\DD^2\DD_{S_1}^{-2})^{q k}h$  and $a$ belongs to $\oTheta(2)$.
    (ii) Let $p\geq 1$ and assume~$(a,b)$ as in the statement.
    Since $ab$ is $\DD$-unmovable, this is also the case for $a$ and $b$.
    Elements $a$, $b$ and $ab$ are then $\DD^p$-unmovable. In particular the $\DD^p$-form of $ab$ is~$c=ab$.
    By hypothesis,  we have $\dpt(c)=\dpt(ab)\not=\dpt(a)+\dpt(b)-1+\varepsilon$ for any $\varepsilon\in\{0,1\}$,
    \end{proof}
    Note that, despite Point~(i) of Lemma~\ref{L:B1ToBj}, there is no obvious implication between Condition $B(p)$ and Condition $B(\ell\times p)$ because the $\DD^p$-form and the $\DD^{\ell\times p}$-form of the product $ab$ may be different. 
    \begin{lm}
    \label{L:Subcovering}
    Let $T$ be a proper irreducible subset of $S$. Put $T_i= S_i\cap T$ for $i = 1,2$, and assume $(T_2, T_1)$ is a proper covering of $T$.  We have:
    \begin{enumerate}
    \item  The $(S_2,S_1)$-alternating normal form and the $(T_2, T_1)$-alternating normal form of any element of $M_T$ coincide;
    \item Any element of $M_T$ is $\DD$-unmovable in $M$;
    \item $T_1$ and $T_2$ are irreducible;
    \item Assume $(a,b)$ lies in $M_T\times M_T$. If $\dpt(a)+\dpt(b) - \dpt(ab)\not\in \{0,1\}$ in $M_T$ for $(T_2,T_1)$,  then $\dpt(a)+\dpt(b) - \dpt(ab)\not\in \{0,1\}$ in $M$ for $(S_1,S_2)$.
    \end{enumerate}
    \end{lm}
    
    \begin{proof}
    (i) Assume $X,Y \subseteq S$. Recall from the Definition~\ref{P:ExistenceTail}, that the right $Y$-tail  in $M$ of an element $g$ in $M_X$ is the unique maximal element of $R(g,M_Y)$ for the right divisibility. Since $g$ lies in $M_X$, its $Y$-tail in $M$ belongs to $M_X$. So this is also its $(Y\cap X)$-tail in $M_X$. Point (i) follows.
    (ii) Every representative word of an element of $M_T$ is written over $T\subsetneq S$, so it cannot be divisible by $\Delta$, since the support of the latter is $S$. 
    (iii) As $S$ is irreducible and of spherical type, its associated Coxeter $\Gamma$ graph has no cycle (see Figure~\ref{F:CoxDiag}).
    So the intersection of two connected subgraphs of $\Gamma$ is a connected subgraph and (iii) holds.
    (iv)  By (i) $a$, $b$ and $ab$ have the same alternating normal form for $(S_2,S_1)$ and for $(T_2,T_1)$. Then for each of them, the depths in $M$ and in $M_T$ are equal.
    \end{proof}
    \begin{prp}[First criteria]   \label{P:ConditionB:Induction} Let $T$ be a proper irreducible subset of $S$. Put $T_i= S_i\cap T$ for $i = 1,2$, and assume $(T_2, T_1)$ is a proper covering of $T$. Assume their exists $(a,b)$ in $M_T\times M_T\setminus (\oTheta_{S_2,S_1}(p)\times\oTheta_{S_2,S_1}(p))$ such that $\dpt(a)+\dpt(b) - \dpt(ab)\not\in \{0,1\}$ in $M_T$ relatively to~$(T_2,T_1)$. Then the covering $(S_2,S_1)$ does not satisfy Condition~$B(p)$.
    \end{prp}
    \begin{proof} By assumption, $(a,b)$ lies in $M\times M\setminus (\oTheta_{S_2,S_1}(p)\times\oTheta_{S_2,S_1}(p))$. Also, $a$, $b$ and $ab$ are $\Delta$-unmovable in $M$, by Lemma~\ref{L:Subcovering}(ii). By  Lemma~\ref{L:Subcovering}(iv),  $\dpt(a)+\dpt(b) - \dpt(ab)\not\in \{0,1\}$ in $M$ relatively to $(S_2,S_1)$. We conclude with Lemma~\ref{L:B1ToBj}(ii).
    \end{proof}
    In order to apply Proposition~\ref{P:ConditionB:Induction}, we need to identify some specific pairs $(a,b)$ in some well suited irreducible subset~$T$ of $S$ that satisfy its hypotheses. This is the objective of the two following parts.
 
   \begin{prp}[Second criteria] \label{P:ConditionB:Crit:TypeB}  Assume $M$ is an irreducible Artin monoid  of spherical type with generating set $S$ such that $\DD$ is central,  and $(S_2,S_1)$ a proper irreducible  covering of $S$ so that $S_1\cap S_2\neq\emptyset$.
    Let $(g_n,\ldots,g_1)$ be the alternating normal form of $\DD$. If $g_n\not\in M_{S_2\setminus S_1}$, then Condition~$B(p)$ does not hold whatever $p\geq 1$.
    \end{prp}
    \begin{proof}Fix  $a,c, c_1,\ldots, c_m$, with $m\geq 1$, as  follow:
    
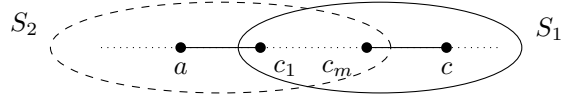
\[
    \begin{tikzpicture}[x=20pt,y=20pt]
    \draw[dotted](0,0) -- (7.5,0);
    \draw(-1,0.45) node[left,] {$S_2$};
    \draw((8,0.35)  node[right] {$S_1$};
    \fill (1.5,0) circle (0.1) node[below=2pt]{$a$};
    \fill (3,0) circle (0.1) node[below right=2pt]{$c_1$};
    \fill (5,0) circle (0.1) node[below left=2pt]{$c_m$};
    \fill (6.5,0) circle (0.1) node[below=2pt]{$c$};
    \draw(1.5,0) -- (3,0);
    \draw(5,0) -- (6.5,0);
    \node[ellipse,
        draw = black,
        minimum width = 3.75cm, 
        minimum height = 1.2cm] (e) at (5.25,0) {};
    \node[ellipse,
    dashed,
        draw = black,
        minimum width = 4.5cm, 
        minimum height = 1.2cm] (e) at (2.25,0) {};    
    \end{tikzpicture}
    \]
    
    As $\DD$ is central, the covering~$(S_2, S_1)$ is $\DD$-fixed.
    From  Lemma~\ref{L:DeltaBreathForFixedCovering} and Lemma~\ref{L:Breadth2Reducible} we obtain that~$n$ is even and different from $2$.
    By Lemma~\ref{L:LeftPowerDelta} we have $\emptyset \not = \SLDA{S}{g}{*}=\SLD{S}{g_n} \subseteq S_2\boxminus S_1 =\{a\}$ and so $\SLD{S}{g_n} = \{a\}$.
    Since $n\geq 4$, we get that $a$ has also to be the unique right divisor of $g_n$ in $S$.
    By assumption, $g_n$ does not lie in $M_{S_2\setminus S_1}$, so we can write $g_n = ah'_nbh_na$ with $ah'_n$ in $M_{S_2\setminus S_1}$, $b$ in $S_2\cap S_1$ and $h_na$ in~$M_{S_2}$.
    We must have $b=c_1$.
    Indeed, if $b=c_i$ for $i>1$ then it commutes with $ah'_n$ implying~$b\in\SLD{S}{g_n}$, which is impossible.
    Let $g = (ah'_n)^{-1}\Delta\Delta_{S_1\setminus S_2}^{-1}$ in $M$.
    Since $\Delta$ is central and $\{a\}\cup (S_1\setminus S_2)$ is not irreducible, we have $gah_n' = ah_n'g$ and $[g] = (c_1h_na,g_{n-1},\ldots,g_2,\Delta_{S_1}\Delta_{S_1\setminus S_2}^{-1})$.
    Define $g_{n+1}$ to be the braid $ac_1^2\cdots c_m^2c^2c_m^2\cdots c_2^2c_1$.
    Then
    \begin{equation}
    \label{E:P:critB:1}
     (a,c_1^2\cdots c_m^2c^2,c_m^2\cdots c_1^2h_na,g_{n-1},\ldots,g_2,\Delta_{S_1}\Delta_{S_1\setminus S_2}^{-1})
     \end{equation}
    is an alternating decomposition of $g_{n+1}g$.
    We claim that, under the assumption of the statement, this is the alternating normal form of $g_{n+1}g$.
    By construction $(c_1h_na,g_{n-1},\ldots,g_2,\Delta_{S_1}\Delta_{S_1\setminus S_2}^{-1})$ is a normal form, and it is immediate to check that $(a,c_1^2\cdots c_m^2c^2,c_m^2\cdots c_1^2h_na,1)$ is a normal form too.
    Hence, the only possibilty for \eqref{E:P:critB:1} not to be a normal form is that for some index $i\in \{n-1,\ldots,1\}$ there exists $b'\in S_{\pi(i)}$ that right divides $ac_1^2\cdots c_m^2c^2c_m^2\cdots c_1^2h_nag_{n-1}\cdots g_{i+1}$ but do not right divides $c_1h_nag_{n-1}\cdots g_{i+1}$.
    By Proposition~\ref{P:ChainDiv}, the braid $c_1h_nag_{n-1}\cdots g_{i+1}$ would be represented by a left $(b'',b')$-chain with $b''$ right dividing $ac_1^2\cdots c_m^2c^2c_m^2\cdots c_1$.
    This would imposes $b'' = c_1$.
    Since $c_1$ left divides $c_1h_nag_{n-1}\cdots g_{i+1}$, the latter cannot be represented by a right $(c_1,b')$-chain by Proposition~\ref{P:ChainDiv}, and so, nor a left $(c_1,b')$-chain by~Proposition~\ref{P:ChainRev}.
    We have a contradiction and then \eqref{E:P:critB:1} is the alternating normal form of $g_{n+1}g$.
    
    On the other hand, let $g'_1\in M_{S_1}$ such that  $\Delta_{S_1} c_1^2\cdots c_m^2c^2c_m^2\cdots c_2^2c_1= g'_1 \Delta_{S_1}$.
    We compute
    \begin{align*}
      g_{n+1}\cdot g\cdot ah'_n &=  g_{n+1} \Delta\Delta_{S_1\setminus S_2}^{-1} = ac_1^2\cdots c_m^2c^2c_m^2\cdots c_2^2c_1\Delta\Delta_{S_1\setminus S_2}^{-1} \\
      & = a \Delta c_1^2\cdots c_m^2c^2c_m^2\cdots c_2^2c_1\Delta_{S_1\setminus S_2}^{-1} &\text{since $\DD$ is central}\\
      & = a \Delta\Delta_{S_1}^{-1}\Delta_{S_1} c_1^2\cdots c_m^2c^2c_m^2\cdots c_2^2c_1\Delta_{S_1\setminus S_2}^{-1} \\
      & =  a \Delta\Delta_{S_1}^{-1}g'_1 \Delta_{S_1}\Delta_{S_1\setminus S_2}^{-1} & \text{by construction of $g_1$} \\
      & = ag_n\cdots g_2g'_1 \Delta_{S_1}\Delta_{S_1\setminus S_2}^{-1} & \text{since $g_1 = \DD_{S_1}$.}
     \end{align*}
    We deduce that $[g_{n+1}gah'_n] = (ag_n,\ldots,g_3,g_2,g'_1\Delta_{S_1}\Delta_{S_1\setminus S_2}^{-1})$.  So, $\brd(g_{n+1}g) = n+2 $,  $\brd(ah'_n) = 2$ and $\brd(g_{n+1}gah'_n) = n$.
    We get $\dpt(g_{n+1}g)+\dpt(ah'_n) =n/2 +2$ and  $\dpt(g_{n+1}gah'_n) = n/2$.
    
    Assume, for a contradiction that $ah'_n$ belongs to $\oTheta(1)$.
    Since $ah'_n$ has breadth $2$ we have $ah'_n\not\in M_{S_1}$ and so it exists $k\geq 1$ and $h\in M_{S_1}$ such that $ah'_n=(\DD\,\DD_{S_1}\inv)^k h$. 
    And so, as $\DD$ is central, $ah'_n=\DD^k\DD_{S_1}^{-k} h$. 
    Since $\DD_{S_1}^k$ and $h$ belongs to $S_1$, the braid $ah'_n$ must have the same breadth than $\DD^k$.
    Using Lemma~\ref{L:Breadth2Reducible} we obtain a contradiction with $\brd(ah'_n) = 2$.
    It follows that $(g_{n+1}g,ah'_n)$ lies on $M\times M\setminus (\oTheta(1)\times \oTheta(1))$
    
    Let $p\geq 1$. 
    Since $\DD$ is central, Lemma~\ref{L:B1ToBj} gives $(g_{n+1}g,ah'_n)$ lies on $M\times M\setminus (\oTheta(p)\times \oTheta(p))$.
    Let us prove that $ah'_n, g_{n+1}g$ and $g_{n+1}gah'_n$ are $\DD$-unmovable.
    This is clear for $ah'_n$. For $g_{n+1}g$   one can deduce it from its alternating normal decompositions: its first term is not right divisible by $\DD_{S_1}$.
    For $g_{n+1}gah'_n$, we have
    \[
     g_{n+1}gah'_n = ac_1^2\cdots c_m^2c^2c_m^2\cdots c_2^2c_1\Delta\Delta_{S_1\setminus S_2}^{-1} =  ac_1^2\cdots c_m^2c^2c_m^2\cdots c_2^2c_1\Delta_{S_1\setminus S_2}^{-1} \Delta
    \]
    The word $ac_1^2\cdots c_m^2c^2c_m^2\cdots c_2^2c_1$ is alone in its equivalence class.
    In particular it is not a right $(c,c)$-chain.
    As $\Delta_{S_1\setminus S_2}^{-1} \Delta$ is not left divisible by $c$, Proposition~\ref{P:ChainDiv} implies that $g_{n+1}gah'_n$ is not left divisible by $c$ and so it is $\DD$-unmovable. 
    Applying Lemma~\ref{L:B1ToBj}, we conclude that  Condition~$B(p)$ does not hold.
    \end{proof}

    \subsection{Induction from type $A$}
    
    In this section we will construct some families of braids $(a_m,b_m)$ of type~$A_r$ with $r\geq 3$ that do not satisfy Condition~$B(p)$.  We use it in the next sections (see Proposition~\ref{P:ConditionB:Induction} for instance).   In the interest of clarity, we define the following notations for $1\leq \ii \leq \jj \leq r$:
    \begin{equation}
    \label{E:sigij}
    \begin{split}
      \siginc\ii\jj = \sig\ii\,\sig\iip \cdots \sig\jj\quad&\text{and}\quad\sigdec\jj\ii = \sig\jj\,\sig\jjo \cdots \sig\ii,\\
      \sigincsq\ii\jj = \sig\ii^2\,\sig\iip^2 \cdots \sig\jj^2\quad&\text{and}\quad\sigdecsq\jj\ii = \sig\jj^2\,\sig\jjo^2 \cdots \sig\ii^2.
    \end{split}
     \end{equation}
    The words that appear in~\eqref{E:sigij} are alone in their equivalence classes and the following relations are immediate.
    \begin{align}
    \label{E:A:PC:INC}
    \siginc\ii\jj \, \sig\kk &\equiv \sig\kkp \, \siginc\ii\jj & \text{for any $1\leq\ii \leq \kk < \jj \leq r$};\\
    \label{E:A:PC:DEC}
    \sigdec\jj\ii\,\sig\kk &\equiv \sig\kko \,\sigdec\jj\ii & \text{for any $1\leq\ii < \kk \leq \jj \leq r$}.
    \end{align}

    \begin{lm}
    \label{L:SigDecLeftChain}
    Let $1\leq i\leq j$ and $1\leq t \leq r$ be integers. Then the word 
    \begin{equation}
    	\label{Sigdec:chain}
    	\text{$\sigdec\jj\ii$ is }
    	\begin{cases}
    		\text{a left $(\sig{t-1},\sig{t})$-chain} & \text{for $i < t \leq j$,} \\ \text{a left $(\sig{t},\sig{t})$-chain} & \text{for $t < i$ or $j < t$.}
    	\end{cases}
    \end{equation}	
    In particular $\sig\ii$ is the unique generator that right divides $\overline{\sigdec\jj\ii}$.
    \end{lm}
    
    \begin{proof}
    If $t<i$ or $t>j$ we conclude using Proposition~\ref{P:SuppChain}.
    For $t\in\{i+1,\ldots,j\}$, we conclude with the following left chain diagram based on Lemma~\ref{L:LcmToChain} and Proposition~\ref{P:SuppChain}:
    \[
     \begin{tikzpicture}
    	\draw[-latex](-0.5,0) -- (1,0) node[midway, below]{$\sigdec{j}{t+1}$};
    	\draw[-latex](1,0) -- (2.5,0) node[midway, below]{$\sig{t}\sig{t-1}$};
    	\draw[-latex](2.5,0) -- (4,0) node[midway, below]{$\sigdec{t-2}i$};
    	\draw[-latex](-0.5,0.75) -- (-0.5,0) node[midway, left]{$\sig{t-1}$};
    	\draw[-latex](1,0.75) -- (1,0) node[midway, right]{$\sig{t-1}$};
    	\draw[-latex](2.5,0.75) -- (2.5,0) node[midway, right]{$\sig{t}$};
    	\draw[-latex](4,0.75) -- (4,0) node[midway, right]{$\sig{t}$};
    	\draw[gray,latex-](-0.25,0.7) -- (0.75,0.7);
    	\draw[gray,latex-](1.25,0.7) -- (2.25,0.7);
    	\draw[gray,latex-](2.75,0.7) -- (3.75,0.7);
    \end{tikzpicture}
    \qedhere
    \]
    \end{proof}
    
    The following lemma gives a more precise statement whenever $t=i-1$.
    \begin{lm}
    	\label{L:SigDec:LeftReverse}
    	For $2\leq i \leq j\leq r$, we have $\sigdec\jj\ii\cdot\sig\iio\inv\curvearrowright_L (\sigdec\jj\iio)\inv\cdot u$ with $u\equiv \sigdec{\jj-1}\iio\cdot \sigdec{\jj}\ii$	
    \end{lm}
    
    \begin{proof}
    	By induction on $\jj-\ii$. For $\ii=\jj$, we have $\sigdec\ii\ii = \sig\ii$,  $\sigdec\ii{\ii-1} = \sig\ii\sig{\ii-1}$    and the left reversing diagram
    	\[
    	\begin{tikzpicture}
    		\draw[-latex](0,0) -- (1.6,0) node[midway, below]{\small$\sig\ii$};
    		\draw[-latex](1.6,1) -- (1.6,0) node[midway, right]{\small$\sig\iio$};
    		\draw[-latex](0,1) -- (0,0.5) node[midway, left]{\small$\sig\ii$};
    		\draw[-latex](0,0.5) -- (0,0) node[midway, left]{\small$\sig\iio$};
    		\draw[-latex](0,1) -- (0.8,1) node[midway, above]{\small$\sig\iio$};
    		\draw[-latex](0.8,1) -- (1.6,1) node[midway, above]{\small$\sig\ii$};
    	\end{tikzpicture}
    	\]
    	and the word  $\sig\iio\sig\ii$ is $\sigdec\iio\iio\cdot \sigdec\ii\ii $.  
    	
  	Assume $\ii<\jj$. we have $\sigdec{j}\ii = \sig\jj\sigdec{j-1}\ii$  and  $\sigdec{j-1}\iio = \sig{\jj-1}\sigdec{j-2}\iio$. By  the induction hypohesis  there exists a word $v$ such that $v\equiv \sigdec{\jj-2}\iio\cdot \sigdec{\jjo}\ii $ with  the left reversing diagram 
    	
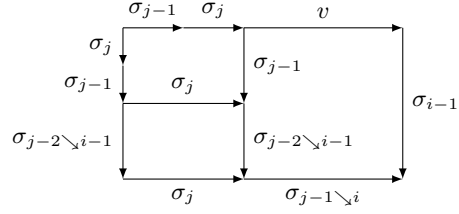
\[
    	\begin{tikzpicture}
    		\draw[-latex](0,0) -- (1.6,0) node[midway, below]{\small$\sig\jj$};
    		\draw[-latex](1.6,0) -- (3.7,0) node[midway, below]{\small$\sigdec\jjo\ii$};
    		\draw[-latex](1.6,1) -- (1.6,0) node[midway, right]{\small$\sigdec{j-2}\iio$};
    		\draw[-latex](1.6,2) -- (1.6,1) node[midway, right]{\small$\sig\jjo$};
    		\draw[-latex](0,1) -- (0,0) node[midway, left]{\small$\sigdec{j-2}\iio$};
    		\draw[-latex](0,1) -- (1.6,1) node[midway, above]{\small$\sig\jj$};
    		\draw[-latex](0,2) -- (0.8,2) node[midway, above]{\small$\sig\jjo$};
    		\draw[-latex](0.8,2) -- (1.6,2) node[midway, above]{\small$\sig\jj$};
    		\draw[-latex](0,2) -- (0,1.5) node[midway, left]{\small$\sig\jj$};
    		\draw[-latex](0,1.5) -- (0,1) node[midway, left]{\small$\sig\jjo$};
    		\draw[-latex](1.6,2) -- (3.7,2) node[midway, above]{\small $v$};
    		\draw[-latex](3.7,2) -- (3.7,0) node[midway, right]{\small $\sig\iio$};
    	\end{tikzpicture}
    	\]
    	and we conclude by considering that $$\sig\jjo\sig\jj v \equiv \sig\jjo\sig\jj \sigdec{\jj-2}\iio\cdot \sigdec{\jjo}\ii \equiv \sig\jjo \sigdec{\jj-2}\iio\cdot \sig\jj \sigdec{\jjo}\ii =  \sigdec{\jjo}\iio\cdot \sigdec{\jj}\ii $$
    \end{proof}
    
    \begin{lm} \label{L:ConditionB:CE:TypeA1} Assume $M$ is of type $A_{k+2}$ for some $k\geq 1$ (see Figure~\ref{F:CoxDiag}) with
    $S_1=\{\sig1,\ldots,\sig\kk\}$ and $S_2=\{\sig2,\ldots,\sig{k+2}\}$.
    Set $u=\sigdec{k+2}1$ and $v=\sigincsq3{k+2}\cdot\sigdecsq\kkp2$.
    For $m\geq 2$, the braids $a=\overline{u}$, $b_m=\overline{v}^m$ and $ab_m$ are $\Delta$-unmovable in $M$, with
    \[
    \dpt(a)=\dpt(b_m)=1\quad\text{and}\quad\dpt(ab_m)=m.
    \]
    \end{lm}
    
    \begin{proof}
    Fix $m\geq 1$.
    Words $u$ and $v^m$ are alone in their equivalence classes, and thereby braids $a$ and~$b_m$ are $\DD$-unmovable.
    Using, in particular, that $v$ is a word over $S_2$, we get that the $(S_2,S_1)$-alternating normal form of $a$ and $b_m$, for any $m\geq 1$,  are
     \[
     [a] = (\sig{k+2}\sig{k+1},\overline{\sigdec\kk1})\quad\text{and}\quad [b_m]=(\overline{v}^{m-1}\cdot\overline{\sigincsq3{k+2}}\cdot\sig{k+1}^2, \overline{\sigdecsq{k}2}),
     \]
     giving $\dpt(a)=\dpt(b_m)=1$.
    Let $w$ be the word $\sigincsq2{k-1}\cdot\sigdecsq\kk1$ and $c$ be $\overline{w}$.
    Using~\eqref{E:A:PC:DEC}, we get~$uv \equiv wu$ together with the relation
    \begin{align*}
    u v^m &\equiv w^{m-1}\, u\,  v&\\
    &\equiv w^{m-1}\,  \sigdec{k+2}1 \, (\sigincsq3{\kk+2}\,\sig\kkp^2\,\sigdecsq\kk2) &\text{by definition of $u$ and $v$}\\
    &\equiv w^{m-1}\, (\sigincsq2{\kk+1}\cdot\sig\kk^2) \ \sigdec{k+2}1 \  (\sigdecsq\kk2) &{\text{by \eqref{E:A:PC:DEC}}}\\
    &\equiv w^{m-1}\, (\sigincsq2{\kk+1}\cdot\sig\kk^2) \, (\sig{k+2}\sig{k+1})\ \sigdec{k}1 \, \sigdecsq\kk2\\
    &\equiv w^{m-1}\, \underbrace{(\sigincsq2{\kk+1}\cdot\sig\kk^2)\,(\sig{k+2}\sig{k+1})}_{w_d}\  \underbrace{\sigdecsq{\kk-1}1\, \sigdec{k}1}_{w_e}  &{\text{by \eqref{E:A:PC:DEC}}}
    \end{align*}
    We define $d$ and $e$ to be the braids represented by $w_d$ and $w_e$, respectively.
    By construction we have $d\in M_{S_2}$ and $e\in M_{S_1}$.
    We now prove that the $S_1$-tail of $c^{m-1}d$ is trivial.
    Let us decompose $d$ as $d_1d_2$ where $d_1$ and $d_2$ are respectively represented by $w_1=\sigincsq2{k+1}\sig\kk$ and $w_2=\sig\kk\sig{k+1}^2\sig\kk^2\sig{k+2}\sig{k+1}$.
    We can check that the word~$w^{m-1}\, w_1$ is alone in its equivalence class giving
    \begin{equation}
     \label{E:L:CEA1:1}
    R(c^{m-1}\, d_1, S) = \{\sig\kk\}.
    \end{equation}
    Proposition~\ref{P:SuppChain} implies that $w_2$ is a left $(\sig\ii,\sig\ii)$-chain for $1\leq i \leq k-1$.
    From the following left chain diagram, based on Lemma~\ref{L:LcmToChain} and Proposition~\ref{P:SuppChain}
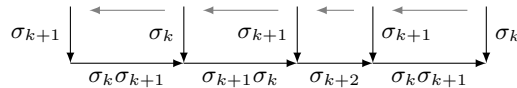
    \[
    \small
    \begin{tikzpicture}
    \draw[-latex](0,0) -- (1.5,0) node[midway, below]{$\sig\kk\sig\kkp$};
    \draw[-latex](1.5,0) -- (3,0) node[midway, below]{$\sig\kkp\sig\kk$};
    \draw[-latex](3,0) -- (4,0) node[midway, below]{$\sig{k+2}$};
    \draw[-latex](4,0) -- (5.5,0) node[midway, below]{$\sig\kk\sig\kkp$};
    
    \draw[-latex](0,0.75) -- (0,0) node[midway, left]{$\sig\kkp$};
    \draw[-latex](1.5,0.75) -- (1.5,0) node[midway, left]{$\sig\kk$};
    \draw[-latex](3,0.75) -- (3,0) node[midway, left]{$\sig\kkp$};
    \draw[-latex](4,0.75) -- (4,0) node[midway, right]{$\sig\kkp$};
    \draw[-latex](5.5,0.75) -- (5.5,0) node[midway, right]{$\sig\kk$};
    \draw[gray,latex-](0.25,0.7) -- (1.25,0.7);
    \draw[gray,latex-](1.75,0.7) -- (2.75,0.7);
    \draw[gray,latex-](3.25,0.7) -- (3.75,0.7);
    \draw[gray,latex-](4.25,0.7) -- (5.25,0.7);
    \end{tikzpicture}
    \]
    we obtain that $w''_2 = \sig\kk\sig{k+1}^2\sig\kk\sig{k+2}\sig\kk\sig{k+1}$ is a left $(\sig\kkp,\sig\kk)$-chain.
    From $w'_2\equiv w_2$, Proposition~\ref{P:ChainDiv} and equality~\eqref{E:L:CEA1:1} imply that the braid $c^{m-1}\cdot d$ is not right-divisible by any element of $S_1$ and so $\tail{S_1}{c^{m-1}\cdot d}= 1$.
    It follows that for $m\geq 2$, the $(S_2,S_1)$-alternating normal form of $ab_m$ is
    \[
     [a b_m] = (\overline{\sigincsq2\kkp},\, \overline{\sigdecsq\kk1} )^{m-1}\cdot(\overline{\sigincsq2{k+1}}\cdot\sig\kk^2\cdot\sig{k+2}\sig{k+1},\,\overline{\sigdecsq\kko1}\,\overline{\sigdec\kk1})
    \]
    implying $\dpt(a b_m) = m$.
    
    It remains to prove that $ab_m$ is $\Delta$-unmovable in $M$.
    Again, from~\eqref{E:A:PC:DEC}, we have $ab_m = c^ma$.
    Using Proposition~\ref{P:SuppChain} and Lemma~\ref{L:LcmToChain} we obtain the following left chain diagram
    \[
     \begin{tikzpicture}
    \draw[-latex](1,0) -- (2.5,0) node[midway, below]{$\sig{k+2}\sig\kkp$};
    \draw[-latex](2.5,0) -- (4,0) node[midway, below]{$\sigdec{k}1$};
    \draw[-latex](1,0.75) -- (1,0) node[midway, left]{$\sig\kkp$};
    \draw[-latex](2.5,0.75) -- (2.5,0) node[midway, left]{$\sig{k+2}$};
    \draw[-latex](4,0.75) -- (4,0) node[midway, right]{$\sig{k+2}$};
    \draw[gray,latex-](1.25,0.7) -- (2.25,0.7);
    \draw[gray,latex-](2.75,0.7) -- (3.75,0.7);
    \end{tikzpicture}
    \]
    proving that $u$ is a left $(\sig\kkp,\sig{k+2})$-chain.
    Since the word $w^m$ ends with $\sig1$ and is the unique representative word of $c^m$, the braid $c^m$ is not right divisible by $\sig\kkp$.
    Hence Proposition~\ref{P:ChainDiv} imposes that $ab_m$ is not right divisible by~$\sig\kkp$.
    This implies in turn that $ab_m$ is not right-divisible by $\DD$, and thereby that it is $\DD$-unmovable.
    \end{proof}
    
    \begin{lm} \label{L:ConditionB:CE:TypeA2} Assume $M$ is of type $A_{k+2}$ for some $k\geq 1$ together with $S_2 = \{\sig1,\ldots,\sig\kk\}$ and $S_1 = \{\sig2,\ldots,\sig{\kk+2}\}$.
    Set $u=\sig1\sigincsq2\kkp\sigdecsq\kk2\sig1$.
    For $m\geq 1$, the braids $a_m= \overline{u^m\,\sigdec{k+2}2}$, $b = \sig1$ and $a_m\,b$ are $\DD$-unmovable and we have
    \[
     \dpt(a_m)=m+1,\quad \dpt(b)= 1\quad \text{and} \quad \dpt(a_mb)=1.
    \]
    \end{lm}
    
    \begin{proof}
    We denote by $c$ the braid represented by $u$.
    Fix $m\geq 1$.
    We immediately check that $u^m$ is alone in its equivalence class and, so, is  the only element of $S$ that right divides $c^m$ is $\sig1$.
    It follows that the~$(S_2,S_1)$-alternating normal form of $a_m$ is
    \[
     [a_m] = (\sig1)\cdot(\overline{\sigincsq2\kkp},\,\overline{\sigdecsq\kk1})^{m-1}\cdot (\overline{\sigincsq2\kkp},\,\overline{\sigdecsq\kk2\sig1},\,\overline{\sigdec{k+2}2}).
    \]
    We deduce that $a_m$  is $\DD$-unmovable,  $\brd(a_m)=2m+2$ and so $\dpt(a_m)=m+1$.
     We have $[b]=(\sig1,1)$ and so $\dpt(b)=1$ and $b$ is obviously $\DD$-unmovable, 
    Using \eqref{E:A:PC:DEC}, we obtain:
    \begin{align*}
    a_mb &=\left(\sig1\,\overline{\sigincsq2\kkp}\,\overline{\sigdecsq\kk2\sig1}\right)^m\cdot\overline{\sigdec{k+2}1}\\
    &=\underbrace{\overline{\sig{k+2}\sig\kkp}}_{\in M_{S_1}}\,\cdot\,\underbrace{\overline{\sigdec\kk1}}_{\in M_{S_2}}
    \cdot \underbrace{\left(\sig2\,\overline{\sigincsq3{k+2}}\,\overline{\sigdecsq\kkp3}\,\sig2\right)^m}_{\in M_{S_1}},
    \end{align*}
    and so $[a_mb] = \big(\sig{k+2}\sig\kkp,\, \overline{\sigdec\kk1}, \,(\sig2\,\overline{\sigincsq3{k+2}}\,\overline{\sigdecsq\kkp3}\,\sig2)^m\big)$. Clearly,  $\Delta_{S_1}$ doesnot right divide the first term, that is $(\sig2\,\overline{\sigincsq3{k+2}}\,\overline{\sigdecsq\kkp3}\,\sig2)^m$. Therefore $\DD$ cannot right-divide $a_mb$. Thus,  the braids $a_m$, $b$ and $a_mb$ are not left divisible by $\DD$ and so they are $\DD$-unmovable.
     \end{proof}
    
    Note that in the two above lemmas, in the case $k = 1$ the set $S_1\cap S_2$ is empty. For $k\geq 2$, it is of type~$A_{k-1}$. For convenience, in the following statements we say that a set is of type $A_0$ whenever it is empty.  
    
    \begin{prp}
    \label{P:ConditionB:SubtypeA}
    Consider an irreducible subset $T$ of $S$ and set $T_i= S_i\cap T$ for $i = 1,2$.  Assume $T$ is of type $A_{k+2}$ for some $k\geq 1$, $(T_2, T_1)$ is a proper covering of $T$ and $T_1\cap T_2$ is of type $A_{k-1}$. 
    Then the covering $(S_2, S_1)$ does not satisfy Condition~$B(p)$ whatever $p\geq 1$.
    \end{prp}
    
    \begin{proof}
    Write $T = \{\sig1,\sig2,\ldots, \sig{k+2}\}$ with $\sig\ii\sig\jj = \sig\jj\sig\ii$ for $|\ii-\jj| >1$ and $\sig\ii\sig\iip\sig\ii=\sig\iip\sig\ii\sig\iip$ for $1\leq i\leq k+1$ and $T'=T_1\cap T_2$.
    If $k = 1$, then $T' = \emptyset$. 
    So either $T_1 = \{\sig1\}$ and $T_2 = \{\sig2,\sig3\}$, or $T_2 = \{\sig1\}$ and $T_1 = \{\sig2,\sig3\}$. Assume $k\geq 2$. Suppose $T'$ is equal to $\{\sig1,\ldots,\sig{k-1}\}$.
    Then, one of the sets $T_1\setminus T_2$ or $T_2\setminus T_1$ does not contain~$\sig\kk$, let's say $T_1\setminus T_2$. Since $T_1\setminus T_2$ must contains at least $\sig{k+1}$ or~$\sig{k+2}$, we obtain a contradiction with the fact that $T_1$ is irreducible (see Lemma~\ref{L:Subcovering}~$(iii)$). 
    If $T_2\setminus T_1$ does not contain~$\sig\kk$ we obtain a similar contradiction, and a similar argument gives that~$T'$ is different from $\{\sig4,\ldots,\sig{k+2}\}$. 
    So the only two possibilities are $T' = \{\sig2,\ldots,\sig\kk\}$ or $T' = \{\sig3,\ldots,\sig{k+1}\}$.
    Up to replace $\sig\ii$ with $\sig{k+3-\ii}$,  we can assume $T'=\{\sig2,\ldots,\sig{k}\}$.
    Coming back to the general case $k\geq 1$, we have then 2 cases to consider: $T_1 = \{\sig1\}\cup T'$ and $T_2 = T'\cup\{\sig{k+1},\sig{k+2}\}$ or  $T_1 = T'\cup\{\sig{k+1},\sig{k+2}\}$ and  $T_2 = \{\sig1\}\cup T'$.
    
    \emph{First case}. Assume $T_1 = \{\sig1\}\cup T'$ and $T_2 = T'\cup\{\sig{k+1},\sig{k+2}\}$. 
    Consider $a=\overline{u}$ and $b_3=\overline{v}^3$ as in Lemma~\ref{L:ConditionB:CE:TypeA1}.
    The element $\DD^p\DD_{S_1}^{-p}$ is left divisible by $\sig{k+2}$.
    Words $u$ and $v_3$ are alone in their equivalence classes.
    In particular $a$ and $b_3$ are not right divisible by $\sig{k+2}$. So,  $\DD^p\DD_{S_1}^{-p}$ does not left divide $a$ nor~$b_3$.
    Since $a$ and $b_3$ do not belong to $M_{S_1}$ we obtain $(a,b_3)\in M_T\times M_T\setminus\oTheta_{S_2,S_1}(p)\times \oTheta_{S_2,S_1}(p)$.
    If $T = S$, we conclude with Lemma~\ref{L:ConditionB:CE:TypeA1} and Lemma~\ref{L:B1ToBj}.  If $T$ is a proper subset of $S$, then $ab_3$ is $\DD$-unmovable by Lemma~\ref{L:Subcovering} (ii) and we conclude using Lemma~\ref{L:ConditionB:CE:TypeA1} and Proposition~\ref{P:ConditionB:Induction}.
    
    \emph{Second case}.  Assume $T_1 = T'\cup\{\sig{k+1},\sig{k+2}\}$ and $T_2 = \{\sig1\}\cup T'$. Consider $a_2$ and $b$ as in Lemma~\ref{L:ConditionB:CE:TypeA2}. 
    In particular $a_2,b$ and $a_2b$ are $\DD$-unmovable. 
    Since $\DD^p\DD_{S_1}^{-p}$ is right divisible by $\sig1$ and this is not the case for $a_2$, the element $a_2$ is not right divisible by $\DD^p\DD_{S_1}^{-p}$.
    Hence, as $a_2\not\in M_{S_1}$, we have $(a_2,b)\not \in (M_T\times M_T)\setminus \oTheta_{S_2,S_1}(p)\times \oTheta_{S_2,S_1}(p)$. 
    If $T = S$, we conclude with Lemma~\ref{L:ConditionB:CE:TypeA2} and Lemma~\ref{L:B1ToBj}.  If $T$ is a proper subset of $S$, we conclude using Lemma~\ref{L:ConditionB:CE:TypeA2} and Proposition~\ref{P:ConditionB:Induction}.
    \end{proof}
    
    \begin{cor}
    \label{C:ConditionB:TypeA}
    Assume $M$ is of type $A_{k+2}$ for  some $k\geq 1$ and $S_1\cap S_2$ is of type $A_{k-1}$. 
    Then the covering $(S_2,S_1)$ does not satisfy Condition~$B(p)$ whatever $p\geq 1$.
    \end{cor}
    \begin{proof}
      We apply Proposition~\ref{P:ConditionB:SubtypeA} with $T = S$. 
    \end{proof}
    
    \subsection{The cases of small types}
    
    An Artin monoid $M$ with generating set $S$ is said to be of \emph{small type} whenever $m(s,t)\leq 3$ for all $(s,t)\in S^2$.
    Irreducible Artin monoids that are both of spherical type and of small type are those of type $A_r$, $D_r$, $E_6$, $E_7$ or $E_8$. For this section, we assume that $M$ is an irreducible spherical type Artin monoid of small type. As the considered Coxeter graph $\Gamma_S$ is connected and without cycle, for any two vertices $s,t$ of $\Gamma_S$ their exists a unique path in $\Gamma_S$ connecting $s$ to $t$.

\subsubsection{Co-abelian coverings}      
    \begin{df}
    A subset $T$ of $S$ is said to be \emph{co-abelian} if $M_{S\setminus T}$ is abelian.
    A covering $(S_2, S_1)$ of $S$ is  \emph{co-abelian} if $S_1$ and $S_2$ are both co-abelian.
    \end{df}
    
    Note that a maximal proper subset of $S$ is necessary co-abelian but the converse is not true in general.
  
    \begin{lm}
    \label{L:CoAbelian}
    Assume $T$ is a non co-abelian subset of $S$ and $M_T$ is irreducible. 
    Then their exist $s,t\in S\setminus T$ such that $m(s,t)\geq 3$ and $M_{T\cup\{s\}}$ is irreducible.
    \end{lm}
    \begin{proof}
     Let $T$ be a non co-abelian subset of $S$ with $M_T$ irreducible. 
     We denote by $\Gamma_{C_1},\ldots,\Gamma_{C_\ell}$ the connected components of $S\setminus T$. 
     Since $T$ is not co-abelian,  one of the $C_i$, let's say $C_1$, contains at least two elements. 
     As $\Gamma_S$ is connected, it exists $s\in C_1$ such that $\Gamma_{T\cup\{s\}}$ is connected. 
     Since $\Gamma_{C_1}$ is connected with at least two vertices, it exists $t\in C_1$ such that $m(s,t)\geq 3$.
    \end{proof}
    
    \begin{prp}
    \label{P:SmallCoAbelian}
    Let $M$ be an irreducible spherical type Artin monoid of small type with generating set~$S$ and $(S_2,S_1)$ be a proper and irreducible covering of $S$. 
    The covering $(S_2,S_1)$ may satisfy condition~$B(p)$ for some $p\geq 1$ only if it is co-abelian.
    \end{prp}
    
    \begin{proof}
    Let $M$ and $(S_2, S_1)$ as in the statement and assume the latter is not co-abelian. Write $\{i,j\} = \{1,2\}$ so that $S_i$ is not co-abelian. 
    By Lemma~\ref{L:CoAbelian} it exists $s$ and $t$ in $S_j\setminus S_i$ such that $m(s,t) \geq 3$ and $\Gamma_{S_i\cup\{s\}}$ is connected.   
    Since $\Gamma_S$ is connected, it exists $\sig1 \in S_i\setminus S_j$ such that $\Gamma_{S_j\cup\{\sig1\}}$ is connected. 
    As~$\Gamma_S$ contains no cycle, the unique path ($\sig1,\ldots,\sig\kk,\sig\kkp= s$) from $\sig1$ to $s$ in $\Gamma_S$ contains only vertices of $\left(S_j\cup\{\sig1\}\right)\cap \left(S_i\cup\{s\}\right)$. 
    Put $T = \{\sig1,\ldots,\sig{k+1},t\}$, $T_1=T\cap S_1$ and $T_2=T\cap S_2$.
    Because $M$ is of small type, $T$ has to be of type $A_{k+2}$ with $k\geq 1$ and we have $T_1\cap T_2= \{\sig2,\ldots,\sig\kk\}$ which is of type~$A_{k-1}$. 
    Since $\sig1\in (T\cap S_\ii)\setminus (T\cap S_\jj)$ and $s=\sig\kkp\in(T\cap S_\jj)\setminus (T\cap S_\ii)$, the covering $(T_2,T_1)$ of $T$ is proper.
    We conclude with Proposition~\ref{P:ConditionB:SubtypeA}. 
    \end{proof}
    
    As a consequence of Proposition~\ref{P:SmallCoAbelian},  we get 
    
    \begin{cor}
      Assume $M$ is of type $A_r$.  If the covering $(S_2,S_1)$ satisfies condition~$B(p)$,  then $\{S_1,S_2\} = \{S\setminus\{\sig1\}, S\setminus\{\sig\rr\}\}$:
    \[
    \{S_1,S_2\}=\big\{
         \begin{tikzpicture}[x=15pt,y=15pt,baseline=-5pt]
    		\draw[lightgray,line width=1](0,0) -- (1,0);
    		\draw[line width=1](1,0) -- (2,0);
    		\draw[dashed, line width=1](2,0) -- (4,0);
    		\draw[line width=1](4,0) -- (6,0);
    		\fill[lightgray](0,0) circle (0.15);
    		\fill(1,0) circle (0.15);
    		\fill(2,0) circle (0.15);	
    		\fill(4,0) circle (0.15);
    		\fill(5,0) circle (0.15);
    		\fill(6,0) circle (0.15);
    		\draw (0,0) node[below] {\small$\sig1$};	
    		\draw (1,0) node[below] {\small$\sig2$};	
    		\draw (5,0) node[below] {\small$\sig\rro$};	
    		\draw (6.2,0) node[below] {\small$\sig\rr$};
    	\end{tikzpicture}
    	,\quad
    	\begin{tikzpicture}[x=15pt,y=15pt,baseline=-5pt]
    		\draw[line width=1](0,0) -- (2,0);
    		\draw[dashed, line width=1](2,0) -- (4,0);
    		\draw[line width=1](4,0) -- (5,0);
    		\draw[lightgray, line width=1](5,0) -- (6,0);
    		\fill(0,0) circle (0.15);
    		\fill(1,0) circle (0.15);
    		\fill(2,0) circle (0.15);	
    		\fill(4,0) circle (0.15);
    		\fill(5,0) circle (0.15);
    		\fill[lightgray](6,0) circle (0.15);
    		\draw (0,0) node[below] {\small$\sig1$};	
    		\draw (1,0) node[below] {\small$\sig2$};	
    		\draw (5,0) node[below] {\small$\sig\rro$};	
    		\draw (6.2,0) node[below] {\small$\sig\rr$};
    	\end{tikzpicture}
    	\big\}
    \]
    Moreover, in this case, the covering $(S_2,S_1)$ is $\DD^p$-regular and Condition~$B(2)$ is satisfied.
    \end{cor}

    \begin{proof}
    Since $\Gamma_{S_1}$ and $\Gamma_{S_2}$ are both connected and $S_1$, $S_2$ are different form $S$, $\sig1$ and $\sig{r}$ cannot lie both in $S_1$ nor in $S_2$.
    Up to exchanging $\sig\ii$ and  $\sig{\rr-\ii +1}$ assume that $\sig1\in S_1\setminus S_2$ and $\sig{r}\in S_2\setminus S_1$.
    It exists $1\leq \ii\leq r-1$ and $2\leq \jj\leq r$ such that $S\setminus S_2 =\{\sig1,\ldots,\sig\ii\}$ and $S\setminus S_1=\{\sig\jj,\ldots,\sig{r}\}$.
    As $S_1$ is not co-abelian for $\jj\leq r-1$ and $S_2$ not co-abelian for $\ii\geq 2$, Proposition~\ref{P:SmallCoAbelian} implies that $(S_2,S_1)$ satisfies condition~$B(k)$ for some $k\geq 1$ only if $S\setminus S_1=\{\sig{r}\}$ and $S\setminus S_2=\{\sig1\}$. 
   In the latter case we immediately check that the covering $(S_2,S_1)$ is $\DD^p$-regular.
  From~\cite{ArP2019},  $(S_2,S_1)$ satisfies Condition~$B(2)$ in this case.
    \end{proof}
    
    It follows from the above corollary that to expect a not yet known ordering on braids groups of type~$A$ using the approach of Diego Arcis and Luis Paris, one has to consider coverings $(S_2,S_1)$  that are not $\DD^p$-regular. For types $D$, $E_6$, $E_7$ and $E_8$, the list of co-abelian regular coverings  $(S_2,S_1)$  is given by the following list:
    \subsubsection{Type $E$}
    Up to exchanging $S_1$ and $S_2$, in types  $E_6$, $E_7$ or $E_8$, the  co-abelian proper and irreducible coverings are the $5$ following ones (in type $E_6$,  the second and the third covering are equal) . 
    \[
    S_i = 
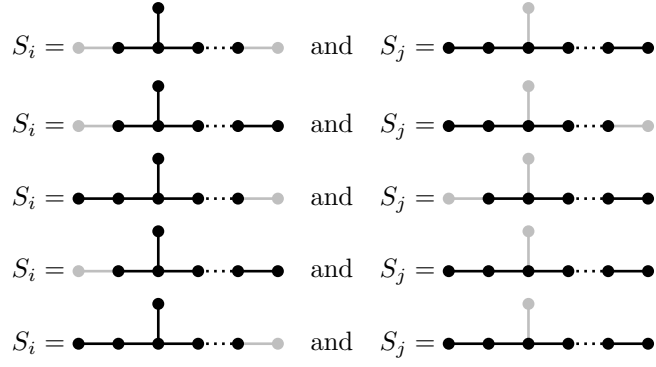
    \begin{tikzpicture}[x=15pt,y=15pt]
    \fill[lightgray](0,0) circle (0.15);
    \fill[lightgray](5,0) circle (0.15);
    \draw[lightgray, line width=1](0,0) -- (1,0);
    \draw[lightgray, line width=1](4,0) -- (5,0);
    \draw[line width=1](2,0) -- (2,1);
    \draw[line width=1](1,0) -- (3,0);
    \draw[dotted,line width=1](3,0) -- (4,0);
    \fill(1,0) circle (0.15);
    \fill(2,0) circle (0.15);
    \fill(2,1) circle (0.15);
    \fill(3,0) circle (0.15);
    \fill(4,0) circle (0.15);
    \end{tikzpicture}
    \quad\text{and}\quad
    S_j = 
    \begin{tikzpicture}[x=15pt,y=15pt]
    \fill[lightgray](2,1) circle (0.15);
    \draw[lightgray, line width=1](2,0) -- (2,1);
    \draw[line width=1](0,0) -- (3,0);
    \draw[line width=1](4,0) -- (5,0);
    \draw[dotted,line width=1](3,0) -- (4,0);
    \fill(0,0) circle (0.15);
    \fill(1,0) circle (0.15);
    \fill(2,0) circle (0.15);
    \fill(3,0) circle (0.15);
    \fill(4,0) circle (0.15);
    \fill(5,0) circle (0.15);
    \end{tikzpicture}
    \]
    \[
    S_i = 
    \begin{tikzpicture}[x=15pt,y=15pt]
    \fill[lightgray](0,0) circle (0.15);
    \draw[lightgray, line width=1](0,0) -- (1,0);
    \draw[line width=1](4,0) -- (5,0);
    \draw[line width=1](2,0) -- (2,1);
    \draw[line width=1](1,0) -- (3,0);
    \draw[dotted,line width=1](3,0) -- (4,0);
    \fill(1,0) circle (0.15);
    \fill(2,0) circle (0.15);
    \fill(2,1) circle (0.15);
    \fill(3,0) circle (0.15);
    \fill(4,0) circle (0.15);
    \fill(5,0) circle (0.15);
    \end{tikzpicture}
    \quad\text{and}\quad
    S_j = 
    \begin{tikzpicture}[x=15pt,y=15pt]
    \fill[lightgray](2,1) circle (0.15);
    \fill[lightgray](5,0) circle (0.15);
    \draw[lightgray, line width=1](2,0) -- (2,1);
    \draw[line width=1](0,0) -- (3,0);
    \draw[lightgray,line width=1](4,0) -- (5,0);
    \draw[dotted,line width=1](3,0) -- (4,0);
    \fill(0,0) circle (0.15);
    \fill(1,0) circle (0.15);
    \fill(2,0) circle (0.15);
    \fill(3,0) circle (0.15);
    \fill(4,0) circle (0.15);
    \end{tikzpicture}
    \]
    \[
    S_i = 
    \begin{tikzpicture}[x=15pt,y=15pt]
    \fill[lightgray](5,0) circle (0.15);
    \draw[lightgray, line width=1](4,0) -- (5,0);
    \draw[line width=1](2,0) -- (2,1);
    \draw[line width=1](0,0) -- (3,0);
    \draw[dotted,line width=1](3,0) -- (4,0);
    \fill(0,0) circle (0.15);
    \fill(1,0) circle (0.15);
    \fill(2,0) circle (0.15);
    \fill(2,1) circle (0.15);
    \fill(3,0) circle (0.15);
    \fill(4,0) circle (0.15);
    \end{tikzpicture}
    \quad\text{and}\quad
    S_j = 
    \begin{tikzpicture}[x=15pt,y=15pt]
    \fill[lightgray](0,0) circle (0.15);
    \fill[lightgray](2,1) circle (0.15);
    \draw[lightgray, line width=1](0,0) -- (1,0);
    \draw[lightgray, line width=1](2,0) -- (2,1);
    \draw[line width=1](1,0) -- (3,0);
    \draw[line width=1](4,0) -- (5,0);
    \draw[dotted,line width=1](3,0) -- (4,0);
    
    \fill(1,0) circle (0.15);
    \fill(2,0) circle (0.15);
    \fill(3,0) circle (0.15);
    \fill(4,0) circle (0.15);
    \fill(5,0) circle (0.15);
    \end{tikzpicture}
    \]
    \[
    S_i = 
    \begin{tikzpicture}[x=15pt,y=15pt]
    \fill[lightgray](0,0) circle (0.15);
    \draw[lightgray, line width=1](0,0) -- (1,0);
    \draw[line width=1](4,0) -- (5,0);
    \draw[line width=1](2,0) -- (2,1);
    \draw[line width=1](1,0) -- (3,0);
    \draw[dotted,line width=1](3,0) -- (4,0);
    \fill(1,0) circle (0.15);
    \fill(2,0) circle (0.15);
    \fill(2,1) circle (0.15);
    \fill(3,0) circle (0.15);
    \fill(4,0) circle (0.15);
    \fill(5,0) circle (0.15);
    \end{tikzpicture}
    \quad\text{and}\quad
    S_j = 
    \begin{tikzpicture}[x=15pt,y=15pt]
    \fill[lightgray](2,1) circle (0.15);
    \draw[lightgray, line width=1](2,0) -- (2,1);
    \draw[line width=1](0,0) -- (3,0);
    \draw[line width=1](4,0) -- (5,0);
    \draw[dotted,line width=1](3,0) -- (4,0);
    \fill(0,0) circle (0.15);
    \fill(1,0) circle (0.15);
    \fill(2,0) circle (0.15);
    \fill(3,0) circle (0.15);
    \fill(4,0) circle (0.15);
    \fill(5,0) circle (0.15);
    \end{tikzpicture}
    \]
    \[
    S_i = 
    \begin{tikzpicture}[x=15pt,y=15pt]
    \fill[lightgray](5,0) circle (0.15);
    \draw[lightgray, line width=1](4,0) -- (5,0);
    \draw[line width=1](2,0) -- (2,1);
    \draw[line width=1](0,0) -- (3,0);
    \draw[dotted,line width=1](3,0) -- (4,0);
    \fill(0,0) circle (0.15);
    \fill(1,0) circle (0.15);
    \fill(2,0) circle (0.15);
    \fill(2,1) circle (0.15);
    \fill(3,0) circle (0.15);
    \fill(4,0) circle (0.15);
    \end{tikzpicture}
    \quad\text{and}\quad
    S_j = 
    \begin{tikzpicture}[x=15pt,y=15pt]
    \fill[lightgray](2,1) circle (0.15);
    \draw[lightgray, line width=1](2,0) -- (2,1);
    \draw[line width=1](0,0) -- (3,0);
    \draw[line width=1](4,0) -- (5,0);
    \draw[dotted,line width=1](3,0) -- (4,0);
    \fill(0,0) circle (0.15);
    \fill(1,0) circle (0.15);
    \fill(2,0) circle (0.15);
    \fill(3,0) circle (0.15);
    \fill(4,0) circle (0.15);
    \fill(5,0) circle (0.15);
    \end{tikzpicture}
    \]
    
   In types $E_7$ or $E_8$, $\DD$ is central  and  all  are $\DD^p$-regular  whathever $p\geq 1$.  In type $E_6$ they are all $\DD^{p}$ regular for $p$ even ;  Among the five, only the three first ones are $\DD^{p}$ regular for $p$ odd. 
   
    \subsubsection{Type $D$.}
    
   In type $D$, up to exchanging $S_1$ and $S_2$ and the  graph automorphism  that exchanges $\sig0$, $\sig0'$ and fixes the other generators,  there is 4  co-abelian proper and irreducible coverings. 
   
    \[
    S_i = 
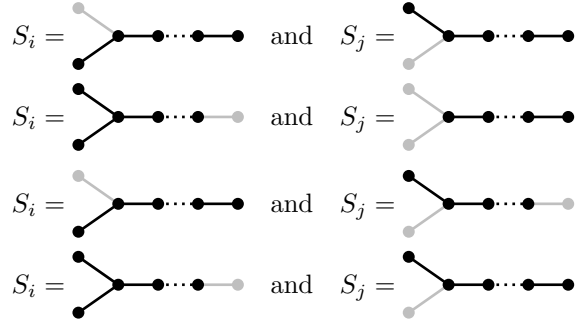
    \begin{tikzpicture}[x=15pt,y=15pt,baseline=-3pt]
    \draw[lightgray,line width=1](0,0.7) -- (1,0);
    \fill(0,-0.7) circle (0.15);
    \fill[lightgray](0,0.7) circle (0.15);
    \fill(4,0) circle (0.15);
    \fill(1,0) circle (0.15);
    \fill(3,0) circle (0.15);
    \fill(2,0) circle (0.15);
    \draw[line width=1](0,-0.7) -- (1,0) -- (2,0);
    \draw[dotted, line width=1](2,0) -- (3,0);
    \draw[line width=1](3,0) -- (4,0);
    \end{tikzpicture}
    \quad\text{and}\quad
    S_j = 
    \begin{tikzpicture}[x=15pt,y=15pt,baseline=-3pt]
    \draw[lightgray,line width=1](0,-0.7) -- (1,0);
    \fill(0,0.7) circle (0.15);
    \fill[lightgray](0,-0.7) circle (0.15);
    \fill(4,0) circle (0.15);
    \fill(1,0) circle (0.15);
    \fill(3,0) circle (0.15);
    \fill(2,0) circle (0.15);
    \draw[line width=1](0,0.7) -- (1,0) -- (2,0);
    \draw[dotted, line width=1](2,0) -- (3,0);
    \draw[line width=1](3,0) -- (4,0);
    \end{tikzpicture}
    \]
    \[
    S_i = 
    \begin{tikzpicture}[x=15pt,y=15pt,baseline=-3pt]
    \draw[lightgray,line width=1](0,0.7) -- (1,0);
    \draw[lightgray,line width=1](3,0) -- (4,0);
    \fill(0,-0.7) circle (0.15);
    \fill(0,0.7) circle (0.15);
    \fill[lightgray](4,0) circle (0.15);
    \fill(1,0) circle (0.15);
    \fill(3,0) circle (0.15);
    \fill(2,0) circle (0.15);
    \draw[line width=1](0,0.7) -- (1,0);
    \draw[line width=1](0,-0.7) -- (1,0) -- (2,0);
    \draw[dotted, line width=1](2,0) -- (3,0);
    \end{tikzpicture}
    \quad\text{and}\quad
    S_j = 
    \begin{tikzpicture}[x=15pt,y=15pt,baseline=-3pt]
    \draw[lightgray,line width=1](0,-0.7) -- (1,0);
     \draw[lightgray,line width=1](0,0.7) -- (1,0);
    \fill[lightgray](0,0.7) circle (0.15);
    \fill[lightgray](0,-0.7) circle (0.15);
    \fill(4,0) circle (0.15);
    \fill(1,0) circle (0.15);
    \fill(3,0) circle (0.15);
    \fill(2,0) circle (0.15);
    \draw[line width=1](1,0) -- (2,0);
    \draw[dotted, line width=1](2,0) -- (3,0);
    \draw[line width=1](3,0) -- (4,0);
    \end{tikzpicture}
    \]
   
      \[
    S_i = 
    \begin{tikzpicture}[x=15pt,y=15pt,baseline=-3pt]
    \draw[lightgray,line width=1](0,0.7) -- (1,0);
    \fill(0,-0.7) circle (0.15);
    \fill[lightgray](0,0.7) circle (0.15);
    \fill(4,0) circle (0.15);
    \fill(1,0) circle (0.15);
    \fill(3,0) circle (0.15);
    \fill(2,0) circle (0.15);
    \draw[line width=1](0,-0.7) -- (1,0) -- (2,0);
    \draw[dotted, line width=1](2,0) -- (3,0);
    \draw[line width=1](3,0) -- (4,0);
    \end{tikzpicture}
    \quad\text{and}\quad
    S_j = 
    \begin{tikzpicture}[x=15pt,y=15pt,baseline=-3pt]
    \draw[lightgray,line width=1](0,-0.7) -- (1,0);
    \draw[lightgray,line width=1](3,0) -- (4,0);
    \fill(0,0.7) circle (0.15);
    \fill[lightgray](0,-0.7) circle (0.15);
    \fill[lightgray](4,0) circle (0.15);
    \fill(1,0) circle (0.15);
    \fill(3,0) circle (0.15);
    \fill(2,0) circle (0.15);
    \draw[line width=1](0,0.7) -- (1,0) -- (2,0);
    \draw[dotted, line width=1](2,0) -- (3,0);
    \end{tikzpicture}
    \]
    \[
    S_i = 
    \begin{tikzpicture}[x=15pt,y=15pt,baseline=-3pt]
    \draw[lightgray,line width=1](0,0.7) -- (1,0);
    \draw[lightgray,line width=1](3,0) -- (4,0);
    \fill(0,-0.7) circle (0.15);
    \fill(0,0.7) circle (0.15);
    \fill[lightgray](4,0) circle (0.15);
    \fill(1,0) circle (0.15);
    \fill(3,0) circle (0.15);
    \fill(2,0) circle (0.15);
    \draw[line width=1](0,0.7) -- (1,0);
    \draw[line width=1](0,-0.7) -- (1,0) -- (2,0);
    \draw[dotted, line width=1](2,0) -- (3,0);
    \end{tikzpicture}
    \quad\text{and}\quad
    S_j = 
    \begin{tikzpicture}[x=15pt,y=15pt,baseline=-3pt]
    \draw[lightgray,line width=1](0,-0.7) -- (1,0);
    \fill(0,0.7) circle (0.15);
    \fill[lightgray](0,-0.7) circle (0.15);
    \fill(4,0) circle (0.15);
    \fill(1,0) circle (0.15);
    \fill(3,0) circle (0.15);
    \fill(2,0) circle (0.15);
    \draw[line width=1](0,0.7) -- (1,0) -- (2,0);
    \draw[dotted, line width=1](2,0) -- (3,0);
    \draw[line width=1](3,0) -- (4,0);
    \end{tikzpicture}
    \]
 They are all $\DD^p$-regular for $p$ even.  For $p$ odd, thare are all   $\DD^p$-regular in type $D_{2n+1}$.  Only the first two covering are $\DD^p$-regular in type $D_{2n}$ . 
    \subsection{Type $I_2$}
  Consider type  $I_2(m)$ with the notations in Figure~\ref{F:CoxDiag}. Up to exchanging $\sig1$ and $\sig2$,  there is only one  proper covering $(S_2,S_1)$ of $S$ ;  it is co-abelian and $\DD^p$-regular whatever $p$.  
    In \cite{ArP2019}, the following result is proved:
    \begin{prp} 
    \begin{enumerate}
    \item If $m$ is even, then $(S_2,S_1)$ satisfies Condition~$B(1)$.
    \item If $m$ is odd, then $(S_2,S_1)$ satisfies Condition~$B(2)$.
    \end{enumerate}
    \end{prp}
    \subsection{Type $B$}\label{S:ConditonB_typeB} 
    
    In this section, $M$ is assumed to be an  Artin monoid of type $B_r$ with $r\geq 3$.   We use the notation of Figure~\ref{F:CoxDiag}.    The case $r = 2$ is cover by the previous section on type $I_2$.
    
     By assumption, one of the sets $S_1,S_2$ is $\{\sig0,\ldots, \sig\kko\}$ and the other is  $ \{\sig{r-j},\ldots,\sig{r-1}\}$,  with  $1\leq k,\,j \leq r-1$ and $r\leq j+k$.  As $\DD$ is central,  $(S_2,S_1)$ is $\DD^p$-regular whatever $p\geq 1$.   As in the small cases, in most cases Condition~$B(p)$ does not hold whatever $p\geq 1$. 
      
   \begin{prp}
    \label{P:typeBCoAb}
    Let $M$ be an irreducible spherical type Artin monoid of type $B_r$ with $r\geq 3$ and the notation of the introduction. Set $S = \{\sig0,\ldots,\sig\rro\}$ and assume $(S_2,S_1)$ is a proper irreducible  covering of $S$. 
    The covering $(S_2,S_1)$ may satisfy condition~$B(p)$ for some $p\geq 1$ only in the following cases.
    \begin{enumerate}
    \item $r = 3$;
    \item  $r\geq 4$ and   $\{S_1, S_2\} = \bigg\{ \{\sig0,\ldots, \sig{r-2}\} , \{\sig2,\ldots, \sig{\rro}\}\bigg\}$.
    \item  $r\geq 4$ and  $\{S_1, S_2\} = \bigg\{ \{\sig1,\ldots, \sig\rro\} ,\{\sig0,\ldots, \sig\kko\}\bigg\}$ such that $k\geq 2$ and $r \equiv 1$ mod $k$.  
    \end{enumerate}
    \end{prp}
    \begin{proof} 
    Assume $r\geq 4$.
    We have $S=S_1\cup S_2$. Let $i\in\{1,2\}$ such that $\sig0\in S_i$.
    Let $j\in\{1,2\}$ such that $\{i,j\}=\{1,2\}$.
    Since $S_i$ is irreducible and different from $S$ we have $\sig\rro\in S_j\setminus S_i$.
    A similar argument give $\sig0\in S_i\setminus S_j$
    Again, since $S_i$ and $S_j$ are irreducible, it exists $1\leq \kk,\ell \leq r-1$ such that 
    \[
    S_i=\{\sig0,\ldots, \sig\kko\} \quad \text{and} \quad S_j=\{\sig{r-\ell},\ldots, \sig{r-1}\}.
    \]
    Consider $T=\{\sig1,\ldots,\sig\rro\}$ together with $T_i= T \cap S_i$ and $T_j=T\cap S_j$. 
    The set $T$ is of type $A_{r-1}$. 
    The couple $(T_2,T_1)$ is a proper and irreducible covering of $T$ if and only of $k\geq 2$ and $\ell\leq r-2$. 
    
    Assume $k\geq 2$ and $\ell\leq r-2$.
    Let $T'=T_1\cap T_2$ and $t$ be the cardinal of $T'$.
    It follows that $T'$ is of type~$A_t$. 
    Since $(T_2,T_1)$ is proper, we must have $t\leq r-3$.
    If $t\leq r-4$ we use Proposition~\ref{P:ConditionB:SubtypeA} to conclude that $(S_2,S_1)$ does not satisfy Condition $B(p)$ whatever $p\geq 1$. 
    The only way to have $t=r-3$ in this case is with $k=r-1$ and $\ell=r-2$, which corresponds to case $(ii)$.
    
    Assume $\ell=r-1$. Then $S_j=\{\sig1,\ldots,\sig\rro\}$.
    Let $(g_n,\ldots,g_1)$ be the alternating normal form of $\DD$ relatively to $(S_2,S_1)$.
    Applying Lemma~\ref{L:ConditionB:ANF:Dellta} below  we get that $g_n$ belongs to $M_{S_2\setminus S_1}$ if and only if $k\geq 2 $ and $r \equiv 1 \mod k$.
    We conclude $(ii)$ with Proposition~\ref{P:ConditionB:Crit:TypeB}.
    \end{proof}     
      
 The remaining of Section~\ref{S:ConditonB_typeB} is devoted to Lemma~\ref{L:ConditionB:ANF:Dellta} and its proof.   We need first to extend the notations of \eqref{E:sigij} in type B for an integer $0\leq \jj \leq \rro$
    \begin{equation}
    \label{E:sigij:B}
    \begin{split}
      \siginc0\jj = \sig0\,\sig1 \cdots \sig\jj\quad&\text{and}\quad\sigdec\jj0 = \sig\jj\,\sig\jjo \cdots \sig0,\\
      \sigincsq0\jj = \sig0^2\,\sig1^2 \cdots \sig\jj^2\quad&\text{and}\quad\sigdecsq\jj0 = \sig\jj^2\,\sig\jjo^2 \cdots \sig0^2.
    \end{split}
     \end{equation}
    Again the words of \eqref{E:sigij:B} are alone in their equivalence classes.
    There iss no counterpart of \eqref{E:A:PC:INC} and \eqref{E:A:PC:DEC} for~$\ii=0$ since $\sig0\sig1\sig0\not=\sig1\sig0\sig1$.
    
    Moreover, from \cite{BrS}, we obtain the following characterisation of Garside element of irreducible parabolic submonoids of $B_r$:
    \begin{align}
    \label{E:GarB}
    \DD_{\{\sig0,\ldots,\sig\ii\}} &= \overline{(\siginc0\ii)^\iip}=\overline{(\sigdec\ii0)^\iip}&\text{for $0\leq\ii\leq\rro$}\\
    \label{E:GarA}
    \DD_{\{\sig\ii,\ldots,\sig\jj\}} &= \overline{(\siginc\jj\jj\,\siginc\jjo\jj\cdots\siginc\ii\jj)} = \overline{(\sigdec\jj\ii\,\sigdec\jj\iip\cdots \sigdec\jj\jj)}&\text{for $1\leq \ii\leq\jj\leq\rro$.}
    \end{align}

    \begin{lm} 
    \label{L:ConditionB:ANF:Dellta}
    Assume $r\geq 4$.
    Fix $1\leq k \leq r-1$ and consider one of the following two cases :
    \begin{itemize}
    \item $S_1 =  \{\sig1,\ldots, \sig\rro\}$ and $S_2 =  \{\sig0,\ldots, \sig\kko\}$.
    \item $S_1 =  \{\sig0,\ldots, \sig\kko\}$  and $S_2 =  \{\sig1,\ldots, \sig\rro\}$.
    \end{itemize}
    Let $q$ and $s\in\{0,\ldots,k-1\}$ be the quotient and the remainder of the euclidean division of $r$ by $k$.
    We also define $\hat{s}$ to be $k$ for $s=0$ and $s$ otherwise.
    \begin{enumerate}
     \item We have $\brd(\Delta) = 2q$ for $s = 0$ and $\brd(\Delta) = 2q+2$ otherwise.
     \item Let $(g_n,\ldots, g_1)$ be the alternating normal form of $\DD$.
     \begin{itemize}
      \item In the first case, we have $g_n = \siginc0{\hat{s}-1}\cdot\siginc0{\hat{s}-2}\cdot\ldots\cdot\siginc01\cdot\siginc00$.
      \item In the second case, we have $g_n = \sigdec\kk{k-\hat{s}+1}\cdot\ldots\cdot\sigdec{k+\hat{s}-2}{k-1} \cdot \sigdec{k+\hat{s}-1}\kk$
     \end{itemize}
    \end{enumerate}
    \end{lm}
    
    \begin{proof}  \emph{First case:} $ S_1 = \{\sig1,\ldots, \sig\rro\}$ and $S_2 =  \{\sig0,\ldots, \sig\kko\}$.
    A direct computation using commutations gives, for $1\leq m \leq r$:
    \begin{align*}
    (\siginc0\rro)^m
    & = \big(\siginc0{r-2} \, \siginc\rro\rro\big) \cdots \big(\siginc0{r-m} \, \siginc{r-m + 1}\rro\big) \cdot \big(\siginc0{r-m-1} \, \siginc{r-m}\rro\big)\\
    & \equiv \big(\siginc0{r-2} \, \siginc\rro\rro\big) \cdots \big(\siginc0{r-m} \, \siginc0{r-m-1}\big) \cdot \big(\siginc{r-m + 1}\rro\,  \siginc{r-m}\rro\big)\\
    & \equiv ~\dots \\
    & \equiv \big(\siginc0{r-2} \cdots \siginc0{r-m}\,\siginc0{r-m-1}\big)\cdot\big(\siginc\rro\rro\cdots\siginc{r-m+1}\rro\siginc{r-m}\rro\big)
    \end{align*}
    which, using \eqref{E:GarA} becomes 
    \begin{equation}
    \label{formulepour5.16cas2}
    \overline{(\siginc0\rro)^m} = \overline{(\siginc0{r-2} \cdots \siginc0{r-m}\,\siginc0{r-m-1})} \cdot \DD_{\{\sig{r-m},\ldots,\sig\rro\}}.
    \end{equation}
    In particular for $m = r - 1$ and relations \eqref{E:GarB} we obtain 
    \begin{equation}
    \label{ref:decompdelta}  \DD = \overline{(\siginc0\rro\,(\siginc0\rro)^{r-1})} = \overline{(\siginc0\rro \, \siginc0{r-2} \cdots \siginc01\,\siginc00)} \cdot \DD_{S_1}.
    \end{equation}
     For $0 \leq \ii \leq \jj \leq \rr$ we set
     \begin{align*}
     \xi_{\jj\searrow\ii} &= \siginc0\jjo\,\siginc0{\jj-2} \cdots \siginc0\iip\,\siginc0\ii,\\
     \nu_{\jj\searrow\ii} &= \siginc{\jj-\ii}\jjo\,\siginc{\jj-\ii-1}{\jj-2}\cdots \siginc2\iip\,\siginc1\ii,
     \end{align*}
    with the convention that $\xi_{j\searrow j}$ and $\nu_{j\searrow j}$ are the empty word.
    With these notations, \eqref{formulepour5.16cas2} and \eqref{ref:decompdelta} become
    \[
    \overline{(\siginc0\rro)^m}  = \overline{\xi_{r-1\searrow r-m-1}} \cdot \DD_{\{\sig{r-m},\ldots,\sig{\rro}\}} \quad\text{and}\quad \Delta = \overline{\xi_{r\searrow0}} \cdot\DD_{S_1}.
     \]
    Straightforward computations based on definitions and commutations give
    \begin{equation}
    \label{E:TauNu}
     \xi_{j\searrow t}\, \xi_{t \searrow i} = \xi_{j \searrow i} \quad\textrm{and}\quad \xi_{j\searrow t} \equiv \xi_{j-t\searrow 0}\cdot \nu_{j\searrow t}\qquad\text{for $0\leq\ii\leq t\leq \jj\leq r$}.
    \end{equation}
    Since $\DD$ is square free, by Proposition~\ref{P:Garside}(i), and $\Delta = \overline{\xi_{r\searrow 0}} \DD_{S_1}$, the unique right-divisor of $\xi_{r\searrow0}$ is~$\sig0$.
    For the same reason, for any $j$, the unique right-divisor of $\overline{\xi_{j\searrow0}}$ is $\sig0$ (just consider the subgroup of type~$B$ generated by $\{\sig0,\ldots, \sig{\jj-1}\}$).
    In the same way, for $1\leq \ii < \jj\leq r$, the word $\xi_{\jj\searrow \ii}$ is a left $(\sig{\ell+\jj-\ii},\sig\ell)$-chain for any $\ell$ in $\{0,\ldots,\ii-1\}$.
    This is a direct consequence of the fact that $\siginc0{t}$ is a left~$(\sig{\ell+1},\sig\ell)$-chain for any $\ell\in\{0,\ldots,t-1\}$ as illustrated by the following diagram:
    \[
     \begin{tikzpicture}
    \draw[-latex](-0.5,0) -- (1,0) node[midway, below]{$\siginc0{\ell-1}$};
    \draw[-latex](1,0) -- (2.5,0) node[midway, below]{$\sig{\ell}\sig{\ell+1}$};
    \draw[-latex](2.5,0) -- (4,0) node[midway, below]{$\siginc{\ell+2}t$};
    \draw[-latex](-0.5,0.75) -- (-0.5,0) node[midway, left]{$\sig{\ell+1}$};
    \draw[-latex](1,0.75) -- (1,0) node[midway, right]{$\sig{\ell+1}$};
    \draw[-latex](2.5,0.75) -- (2.5,0) node[midway, right]{$\sig\ell$};
    \draw[-latex](4,0.75) -- (4,0) node[midway, right]{$\sig\ell$};
    \draw[gray,latex-](-0.25,0.7) -- (0.75,0.7);
    \draw[gray,latex-](1.25,0.7) -- (2.25,0.7);
    \draw[gray,latex-](2.75,0.7) -- (3.75,0.7);
    \end{tikzpicture}
    \]
    Hence, no elements of $S_1$ right divides $\overline{\xi_{j\searrow k}}$.
    Since $\overline{\nu_{j\searrow k}}$ belongs to $M_{S_1}$ for all $j$, and $\overline{\xi_{k\searrow 0}}$ lies in $M_{S_2}$, we  have $\tail{S_1}{\overline{\xi_{\jj\searrow\kk}}} = \overline{\nu_{\jj\searrow\kk}}$ and $\tail{S_2}{\overline{\xi_{\jj\searrow 0}}} = \overline{\xi_{\kk\searrow 0}}$.
    With $2q$ successive applications of \eqref{E:TauNu} we obtain
    \begin{align*}
    \xi_{r\searrow 0} = \xi_{r\searrow k}\xi_{k\searrow 0} &\equiv (\xi_{r-k\searrow 0}\nu_{r\searrow k})\,\xi_{k\searrow0} = (\xi_{r-k\searrow k}\xi_{k\searrow 0})\,\nu_{r\searrow k}\xi_{k\searrow 0}\\
    &\equiv (\xi_{r-2k\searrow 0} \nu_{r-k\searrow k})\,\xi_{k\searrow 0} \nu_{r\searrow k}\xi_{k\searrow 0} \\
    &\equiv\,\cdots\\
    & \equiv \begin{cases}\xi_{s\searrow 0}\nu_{s+k\searrow k}\xi_{k\searrow0}\ \cdots\  \nu_{r-k\searrow k}\xi_{k\searrow 0}\, \nu_{r\searrow k}\xi_{k\searrow 0}\,\Delta_{S_1} & \text{if $s>0$,} \\
    \phantom{\xi_{s\searrow 0}\nu_{s+k\searrow k}}\xi_{k\searrow 0}\ \cdots\ \nu_{r-k\searrow k}\xi_{k\searrow 0}\, \nu_{r\searrow k}\xi_{k\searrow 0}\,\Delta_{S_1} & \text{if $s=0$.}
    \end{cases}
    \end{align*}
    It follows that 
    \[
    \ND{\Delta} = \begin{cases}(\xi_{s\searrow 0},\,\nu_{s+k\searrow k},\,\xi_{k\searrow0},\, \ldots,\,  \nu_{r-k\searrow k},\,\xi_{k\searrow 0},\, \nu_{r\searrow k},\,\xi_{k\searrow0},\,\Delta_{S_1}) & \text{if $s>0$}, \\
    \hspace{6.5em}(\xi_{k\searrow 0},\, \ldots,\,  \nu_{r-k \searrow k},\,\xi_{k\searrow0},\, \nu_{r\searrow k},\,\xi_{k\searrow0},\,\Delta_{S_1}) & \text{if $s=0$}.
    \end{cases}
    \]
    We conclude this case with  $\dpt(\DD) = 2q+2$ for $s>0$ and $\dpt(\DD) =2q$ for $s=0$ and we observe that the relation $g_n = \xi_{\hat{s}\searrow0}$ holds. \\

    \noindent \emph{Second case:} $S_1 =  \{\sig0,\ldots, \sig\kko\}$  and $S_2 =  \{\sig1,\ldots, \sig\rro\}$. 
    From \eqref{E:GarB} we have
    \begin{equation}
    \label{L:Case2:DD:1}
    	\DD = \overline{(\sigdec\rro0)^\rr}=  \overline{(\sigdec\rro0)^{r-k}}\,\overline{(\sigdec{\rro}0)^k}
    \end{equation}
    We have  $\sigdec\rro0 = \sigdec\rro{k-m+1}\,\sigdec{k-m}0$.
    We now exhibit $\DD_{S_1}$ as a right divisor of $(\sigdec\rro0)^k$.
    By induction on $m$, for $1\leq m\leq k$, we have
    \begin{equation}
    \label{E:DD:TypeB}
    (\sigdec\rro0)^m \equiv \big(\sigdec{r-m}{k-m+1}\cdots \sigdec{\rr-2}{\kk-1}\,\sigdec{\rro}{\kk} \big)(\sigdec\kko0)^{m}.
    \end{equation}
    Indeed, we have
    \begin{align*}
    (\sigdec{\rro}0)^m& = \sigdec{\rro}0\,(\sigdec{\rro}0)^{m-1} \\
    &\equiv\big(\sigdec{\rro}{k-m+1}\sigdec{k-m}0\big)\big(\sigdec{r-m+1}{k-m+2}\cdots\sigdec{\rr-2}{\kk-1}\,\sigdec{\rro}\kk \big)(\sigdec{k-1}0)^{m-1}&\textrm{  (by induction)}\\
    &\equiv\sigdec{\rro}{k-m+1}\big(\sigdec{r-m+1}{k-m+2}\cdots \sigdec{\rr-2}{\kk-1}\, \sigdec{\rro}\kk\big)\sigdec{k-m}0(\sigdec{k-1}0)^{m-1}   &\textrm{  (by commutation)}\\
    &\equiv\sigdec{\rro}{r-m+1}\underbrace{\big(\sigdec{r-m}{k-m+1}\,\sigdec{r-m+1}{k-m+2}\cdots\sigdec{\rro}\kk\big)}_\alpha\,\sigdec{k-m}0\,(\sigdec{k-1}0)^{m-1}.
      \end{align*}
    Using commutations we obtain
    \begin{align*}
     \alpha &= (\sig{r-m}\cdot\sigdec{r-m-1}{k-m+1})\,(\sig{r-m+1}\cdot\sigdec{r-m}{k-m+2})\cdots(\sig\rro\cdot\sigdec{\rr-2}\kk)\\
     &\equiv \siginc{r-m}{r-1} \cdot (\sigdec{r-m-1}{k-m+1}\,\sigdec{r-m}{k-m+2}\cdots\sigdec{\rr-2}\kk)\\
     &\equiv \siginc{r-m}{r-1} \, \siginc{r-m-1}{r-2}\cdot (\sigdec{r-m-2}{k-m+1}\,\sigdec{r-m-1}{k-m+2}\cdots\sigdec{\rr-3}\kk)\\
     &\equiv\dots \\
    &\equiv \siginc{r-m}{r-1} \, \siginc{r-m-1}{r-2} \cdots \siginc{k-m+1}k.
    \end{align*}
    Substituting this last expression of $\alpha$ in this of $(\sigdec{r-1}0)^m$ we get
    \begin{align*}
     (\sigdec{\rro}0)^m& \equiv \sigdec{\rro}{r-m+1}\, (\siginc{r-m}{r-1} \, \siginc{r-m-1}{r-2} \cdots \siginc{k-m+1}k) \, \sigdec{k-m}0\,(\sigdec{k-1}0)^{m-1}\\
     &\equiv (\siginc{r-m}{r-1} \, \siginc{r-m-1}{r-2} \cdots \siginc{k-m+1}k)\, \sigdec{k-1}{k-m}\,\sigdec{k-m}0\,(\sigdec{k-1}0)^{m-1} & \text{by \eqref{E:A:PC:INC}}\\
      &\equiv (\siginc{r-m}{r-1} \, \siginc{r-m-1}{r-2} \cdots \siginc{k-m+1}k)\, (\sigdec{k-1}0)^{m} \\
      &\equiv (\sigdec{r-m}{k-m+1}\,\sigdec{r-m+1}{k-m+2}\cdots\sigdec{\rro}\kk)\, (\sigdec{k-1}0)^{m}. & \text{restoring $\alpha$}.
     \end{align*}
    Equation~\eqref{E:DD:TypeB} for $m=k$ together with \eqref{E:GarA} transfom \eqref{L:Case2:DD:1} into
    \begin{equation}
    \label{forumulepour5.16}
    \DD = \overline{(\sigdec\rro0)^{r-k}}\cdot  \overline{(\sigdec{r-k}1\cdots\sigdec{\rro}k)}\cdot \Delta_{S_1}.
    \end{equation}
    We now focus on $(\sigdec\rro0)^{r-k}$.
    Since Artin's relations are the same reading left to right or right to left, equation~\eqref{formulepour5.16cas2} remains valid by mirror symmetry:
    \begin{align*}
    	\overline{(\sigdec\rro0)^m} = \DD_{\{\sig{r-m},\ldots,\sig{r-1}\}} \cdot\overline{(\sigdec{r-m-1}0\sigdec{r-m}0\cdots\sigdec{r-2}0)},
    \end{align*}
    and for $m = r-k$ we obtain
    \begin{equation}
    	\label{E:Delta2}
    	\begin{aligned}
    	\overline{(\sigdec\rro0)^{r-k}} &=  \DD_{\{\sig\kk,\ldots,\sig\rro\}}\cdot \overline{(\sigdec{k-1}0\cdots\sigdec{r-2}0})  \\
    	& = \DD_{\{\sig\kkp,\ldots,\sig\rro\}}\cdot\overline{(\siginc{k}\rro\,\sigdec{k-1}0\cdots\sigdec{r-2}0)}.
    	\end{aligned}
    \end{equation}
    Using commutations we obtain
    \[
    \siginc{k}\rro\,(\sigdec{k-1}0\cdots\sigdec{r-2}0) \equiv (\sig\kk\cdot\sigdec{k-1}0)\cdots(\sig\rro\cdot\sigdec{r-2}0)
    =\sigdec\kk0\cdots\sigdec\rro0.
    \]
    For $0\leq\ii\leq\jj<\rr$, we define the word
    \[
    \muinc\ii\jj = \sigdec\ii0\cdots \sigdec\jj0,
    \]
    so that \eqref{E:Delta2} becomes:
    \begin{equation}
    \label{E:Delta3}
    \overline{(\sigdec\rro0)^{r-k}} = \DD_{S_1^\perp}\cdot\overline{\muinc\kk\rro}.
    \end{equation}
    Gathering all pieces together we obtain
    
    \begin{equation}
    \begin{aligned}
    \DD & =\overline{(\sigdec\rro0)^{r-k}}\cdot  \overline{(\sigdec{r-k}1\cdots\sigdec{\rro}k)}\cdot \Delta_{S_1} & \text{by \eqref{forumulepour5.16}}\\
    & = \DD_{S_1^\perp}\cdot \overline{(\muinc\kk\rro)} \cdot  \overline{(\sigdec{r-k}1\cdots\sigdec{\rro}k)}\cdot \Delta_{S_1} & \text{by \eqref{E:Delta3}}\\
    & = \overline{( \muinc\kk\rro)} \cdot  \overline{(\sigdec{r-k}1\cdots\sigdec{\rro}k)}\cdot \Delta_{S_1} \cdot \DD_{S_1^\perp} &\text{as $\DD$ is central}\\
    & = \overline{(\muinc\kk\rro)} \cdot  \overline{(\sigdec{r-k}1\cdots\sigdec{\rro}k)}\cdot \Delta_{S_1^\perp} \cdot \DD_{S_1} &\text{by commutation.}
    \end{aligned}
    \end{equation}

    As $\DD_{S_1}$ lies in $M_{S_1}$ and $\overline{(\sigdec{r-k}1\cdots\sigdec{\rro}k)}\cdot \DD_{S_1^\perp}$ belongs to $M_{S_2}$ we have almost obtain an $(S_2,S_1)$-alternating decomposition of $\DD$.
    By very definition of $\muinc\ii\jj$ and the relation $r =q\times k + s$, we have
    \begin{equation}
    	\label{E:DecompMu}
    	\muinc\kk\rro = \muinc\kk{(\kk+s-1)} \cdot  \prod_{i=1}^{q-1} \muinc{(\ii\kk+s)}{((\ii+1)\kk+s-1)}.
    \end{equation}
    For $\kk\leq \ii\leq \rro$ and $1\leq t \leq r-i$, we define
    \[
    \nuinc\ii{i+t-1} = \sigdec{i}{k-t+1}\,\sigdec{i+1}{k-t+2}\,\cdots\,\sigdec{i+t-2}{k-1}\,\sigdec{i+t-1}{k},
    \]
    such that, using commutations, we obtain
    \begin{align*}
    	\muinc\ii{i+t-1} &= \sigdec\ii0\, \sigdec{i+1}0 \,\cdots \, \sigdec{i+t-2}0 \,\sigdec{i+t-1}0\\
    	&= \sigdec\ii{k-t+1}\,\sigdec{k-t}0\cdot \sigdec{i+1}{k-t+2}\,\sigdec{k-t+1}0\,\cdots\,\sigdec{i+t-2}{k-1}\,\sigdec{k-2}0\cdot\sigdec{i+t-1}{k}\,\sigdec{k-1}0\\
    	&\equiv(\sigdec{i}{k-t+1}\,\sigdec{i+1}{k-t+2}\,\cdots\,\sigdec{i+t-2}{k-1}\,\sigdec{i+t-1}{k})\cdot(\sigdec{k-t}0\,\sigdec{k-t+1}0\,\cdots\,\sigdec{k-2}0\,\sigdec{k-1}0).\\
    	&\equiv \nuinc\ii{i+t-1}\cdot\muinc{k-t}{k-1}.
    \end{align*}
    By construction $\overline{\muinc{k-t}{k-1}}$ and $\overline{\nuinc\ii{i+t-1}}$ belong respectively to $M_{S_1}$ and $M_{S_2}$. 
    It follows that 
    \begin{equation}
    	(\overline{\nuinc\kk{k+s-1}}, \overline{\muinc{k-s}{k-1}}) \cdot \left(\prod_{i=1}^{q-1} (\overline{\nuinc{ik+s}{(i+1)k+s-1}}, \overline{\muinc{0}{k-1}})\right)\cdot \left(\overline{(\sigdec{r-k}1\cdots\sigdec{\rro}k)}\cdot \Delta_{S_1^\perp},\DD_{S_1}\right)
    \end{equation}
    is an $(S_2,S_1)$-alternating decomposition of $\DD$.
    For $s\not =0$ we denote it by $(h_m,\ldots,h_1)$ with $m = 2q+2$.
    Assume $s=0$. Then  we have $\overline{\nuinc\kk{k+s-1}}=\overline{\muinc{k-s}{k-1}}=1$.
    Then the previous $(S_2,S_1)$-alternating decomposition without this two leftmost terms is also denotes $(h_m,\ldots, h_1)$ but with $m=2q$.
    
    We now prove that, in both cases, $(h_m,\ldots,h_1)$ is normal.
    By Definition~\ref{D:AlternatingDecomposition} it is sufficient to establish 
    \begin{equation}
    	\label{E:Tau}
    	h_\ell = \tail{M_{\pi(\ell)}}{h_m\cdots h_\ell}\qquad\text{for $\ell\in\{1,\ldots,m\}$.}
    \end{equation}
    The equality $\tail{S_2}{h_m}=h_m$ is immediate since $h_m\in M_{S_2}$.
    The only element of $S$ that right divides $h_m$ is $\sig\kk$. 
    This is a direct consequence of the folowing more general and precise result.
    Let $\kk\leq \ii\leq \rro$, $1\leq t \leq r-i$ and $j\in\{0,\ldots,r-1\}$ with $j\not=k$. 
    Then the word 
    \begin{equation}
    	\label{E:NuInc:Chain}
    \text{$\nuinc{i}{i+t-1}$ is a left }
    \begin{cases}
     \text{$(\sig\jj,\sig\jj)$-chain} & \text{for $j < k - t + 1$ or $j > i + t -1$,}\\
     \text{$(\sig{j-t},\sig\jj)$-chain} & \text{for $j\in\{k+1,\ldots,i+t-1\}$,}\\
     \text{$(\sig{i+t+j-k},\sig\jj)$-chain} & \text{for $j\in\{k-t+1,\ldots,k-1\}$.}
    \end{cases}
    \end{equation}
    
    The first case is immediate by Proposition~\ref{P:SuppChain},  since the support of $\nuinc{i}{i+t-1}$ is $\{\sig{k-t+1},\ldots,\sig{i+t-1}\}$.
    For $j\in\{k+1,\ldots,i+t-1\}$, Lemma~\ref{L:SigDecLeftChain} gives the folowing diagram
    \[
    \begin{tikzpicture}
    	\draw[-latex](0,0) -- (2,0) node[midway, below]{\small$\sigdec\ii{k-t+1}$};
    	\draw[dashed](2,0) -- (4,0);
    	\draw[-latex](4,0) -- (6,0) node[midway, below]{\small$\sigdec{\ii+t-2}{k-1}$};
    	\draw[-latex](6,0) -- (8,0) node[midway, below]{\small$\sigdec{\ii+t-1}{k}$};
    	\draw[-latex](8,0.75) -- (8,0) node[midway, right]{\small$\sig\jj$};
    	\draw[-latex](6,0.75) -- (6,0) node[midway, right]{\small$\sig{j-1}$};
    	\draw[-latex](4,0.75) -- (4,0) node[midway, right]{\small$\sig{j-2}$};
    	\draw[-latex](2,0.75) -- (2,0) node[midway, left]{\small$\sig{j-t+1}$};
    	\draw[-latex](0,0.75) -- (0,0) node[midway, left]{\small$\sig{j-t}$};
    	\draw[gray,latex-](0.25,0.7) -- (1.75,0.7);
    	\draw[gray,latex-](4.25,0.7) -- (5.75,0.7);
    	\draw[gray,latex-](6.25,0.7) -- (7.75,0.7);
    \end{tikzpicture}
    \]
    Assume $j\in\{k-t+1,\ldots,k-1\}$.
    Write $j$ as $k-\ell$ with $\ell\in\{1,\ldots,t-1\}$.
    As the support of $\sigdec{i+t-\ell+1}{k-\ell+2}\,\cdots\,\sigdec{i+t-1}{k}$ does not contain $\sig{k-\ell}$, it is a left $(\sig{k-\ell},\sig{k-\ell})$-chain.
    By Lemma~\ref{L:SigDec:LeftReverse} the word $\sigdec{i+t-\ell}{k-\ell+1}\,\sig\jj^{-1}$ left reverses in $\sigdec{i+t-\ell}{k-\ell}^{-1}u$ where $u$ is a positive word and so 
    \[
    \sigdec{i+t-\ell-1}{k-\ell}\,\sigdec{i+t-\ell}{k-\ell+1}\,\sig{k-\ell}^{-1},
    \]
    left reverses in $\sig{i+t-\ell}\inv u$. 
    In particular, $\sigdec{i+t-\ell-1}{k-\ell}\,\sigdec{i+t-\ell}{k-\ell+1}$ is a left $(\sig{i+t-\ell}^{-1},\sig{k-\ell})$-chain
    The word $\sigdec\ii{k-t+1}\cdots\sigdec{i+t-\ell-2}{k-\ell-1}$ is a left $(\sig{i+t-\ell},\sig{i+t-\ell})$-chain as its support does not contain $\sig{i+t-\ell}$.
    Gathering all the parts we obtain the following diagram
    \[
    \begin{tikzpicture}
    \draw[-latex](0,0) -- (2.5,0) node[midway, below]{\small $\sigdec{i}{k-t+1}$};
    \draw[dashed](2.5,0) -- (3.5,0);
    \draw[-latex](3.5,0) -- (6,0) node[midway, below]{\small $\sigdec{i+t-\ell-1}{k-\ell}$};
    \draw[-latex](6,0) -- (8.5,0) node[midway, below]{\small $\sigdec{i+t-\ell}{k-\ell+1}$};
    \draw[dashed](8.5,0) -- (9.5,0);
    \draw[-latex](9.5,0) -- (11,0) node[midway, below]{\small $\sigdec{i+t-1}{k}$};
    \draw[-latex](0,0.75) -- (0,0) node[midway, left]{\small $\sig{i+t-\ell}$};
    \draw[-latex](2.5,0.75) -- (2.5,0) node[midway, left]{\small $\sig{i+t-\ell}$};
    \draw[-latex](3.5,0.75) -- (3.5,0) node[midway, right]{\small $\sig{i+t-\ell}$};
    \draw[-latex](8.5,0.75) -- (8.5,0) node[midway, right]{\small $\sig{k-\ell}$};
    \draw[-latex](9.5,0.75) -- (9.5,0) node[midway, right]{\small $\sig{k-\ell}$};
    \draw[-latex](11,0.75) -- (11,0) node[midway, right]{\small $\sig{k-\ell}$};
    \draw[gray,latex-](0.25,0.7) -- (2.25,0.7);
    \draw[gray,latex-](3.75,0.7) -- (8.25,0.7);
    \draw[gray,latex-](9.75,0.7) -- (10.75,0.7);
    \end{tikzpicture}
    \]
    proving that $\nuinc\ii{i+t-1}$ is a left $(\sig{i+t+j-k}, \sig\jj)$-chain as expected.
    A direct consequence of \eqref{E:NuInc:Chain} with Proposition~\ref{P:ChainDiv} is that 
    \begin{equation}
    	\label{E:EEE}
    	\text{$\sig\kk$ is the unique element of $S$ that right divides $\overline{\nuinc{i}{i+t-1}}$.}
    \end{equation}
    In particular $\tail{S_1}{h_m\cdot h_{m-1}} = h_{m-1}$.
    Since $\tail{S_1}{\DD}=\DD_{S_1} = h_1$ we have \eqref{E:Tau} for $\ell=1$.
    It remains to prove it for $\ell=2,\ldots,m-2$. 
    We will use a backward induction on $\ell$.
    Assume $\tail{S_{\pi(t)}}{h_m\cdots h_t} = h_t$ for all $t\geq \ell$ and $\ell\in\{3,\ldots,m-1\}$.
    Let us prove \eqref{E:Tau} for $\ell-1$.
    
    \textbf{Case} $\ell-1$ is odd. Then $\pi(\ell-1)=1$ and $h_{\ell-1} = \overline{\muinc0{k-1}}$.
    Assume for a contradiction that $\tail{S_{\pi(\ell)}}{h_m\cdots h_{\ell-1}} \not= h_{\ell-1}$.
    As $h_{\ell-1}\in M_{S_1}$, it exists $a\in S_1$ that right-divides $h_m\cdots h_\ell$.
    If $a$ belongs to~$S_2$ then it divides $h_\ell$ since, by the induction hypothesis $\tail{S_2}{h_m\cdots h_\ell} = h_\ell$.
    Since $h_\ell = \overline{\nuinc{ik+s}{(i+1)k+s-1}}$ for some $i\in\{1,\ldots,q-1\}$ we have $\SRD{S}{h_\ell}=\{\sig\kk\}$ which is not include in $S_1$.
    Hence the only possibility for $a$ is to be in $S_1\setminus S_2 = \{\sig0\}$ and so $a=\sig0$.
    This is impossible as $h_{\ell-1}=\overline{\muinc0{k-1}}$ is left divisible by~$\sig0$ and $\DD=h_m\cdots h_1$ is square free by Proposition~\ref{P:Garside}.
    
    \textbf{Case} $\ell-1$ is even. Then $\pi(\ell-1)=2$ and it exists $i\in\{1,\ldots,q-1\}$ such that $h_{\ell-1} = \overline{\nuinc{ik+s}{(i+1)k+s-1}}$.
    Assume for a contradiction that $\tail{S_2}{h_m\cdots h_{\ell-1}}\not=h_{\ell-1} \in S_2$.
    Then it exists $a\in S_2$ that right divides $h_m\cdots h_\ell$. 
    If $a\in S_1$ then it must right divides $h_\ell$ since $\tail{S_1}{h_m\cdots h_\ell}=h_\ell$ by induction hypothesis.
    By construction we have 
    \[
    h_\ell = \begin{cases}
    	\overline{\muinc{k-s}{k-1}} & \text{for $\ell = m -1$ and $s\not =0$,}\\
    	\overline{\muinc{0}{k-1}} & \text{otherwise.}
    	\end{cases}
    \]
    Let $\sig\jj\in S_2\cap S_1=\{\sig1,\ldots,\sig\kko\}$.
    From the following diagram, obtained by applying Lemma~\ref{L:SigDecLeftChain},
    \[
    \begin{tikzpicture}
    	\draw[-latex](0,0) -- (1.5,0)  node[midway,below]{\small $\sigdec{k-1}1$};
    	\draw[-latex](1.5,0) -- (2.5,0)  node[midway,below]{\small $\sig0$};
    	\draw[-latex](0,0.75) -- (0,0) node[midway,left]{\small $\sig\jjo$};
    	\draw[-latex](1.5,0.75) -- (1.5,0) node[midway,right]{\small $\sig\jj$};
    	\draw[-latex](2.5,0.75) -- (2.5,0) node[midway,right]{\small $\sig\jj$};
    	\draw[gray,latex-](0.25,0.7) -- (1.25,0.7);
    	\draw[gray,latex-](1.75,0.7) -- (2.25,0.7);
    \end{tikzpicture}
    \]
    we get
    \[
    \begin{tikzpicture}
    	\draw[-latex](0,0) -- (1.5,0)  node[midway,below]{\small $\sigdec{k-j+1}0$};
    	\draw[-latex](1.5,0) -- (3,0)  node[midway,below]{\small $\sigdec{k-j+2}0$};
    	\draw[dashed](3,0) -- (4,0);
    	\draw[-latex](4,0) -- (5.5,0)  node[midway,below]{\small $\sigdec{k-2}0$};
    	\draw[-latex](5.5,0) -- (7,0)  node[midway,below]{\small $\sigdec{k-1}0$};
    	\draw[-latex](0,0.75) -- (0,0) node[midway,left]{\small $\sig1$};
    	\draw[-latex](1.5,0.75) -- (1.5,0) node[midway,left]{\small $\sig2$};
    	\draw[-latex](3,0.75) -- (3,0) node[midway,left]{\small $\sig3$};
    	\draw[-latex](4,0.75) -- (4,0) node[midway,right]{\small $\sig{j-2}$};
    	\draw[-latex](5.5,0.75) -- (5.5,0) node[midway,right]{\small $\sig{j-1}$};
    	\draw[-latex](7,0.75) -- (7,0) node[midway,right]{\small $\sig{j}$};
    	\draw[gray,latex-](0.25,0.7) -- (1.25,0.7);
    	\draw[gray,latex-](1.75,0.7) -- (2.75,0.7);	
    	\draw[gray,latex-](4.25,0.7) -- (5.25,0.7);
    	\draw[gray,latex-](5.75,0.7) -- (6.75,0.7);	
    \end{tikzpicture}
    \]
    Using Lemma~\ref{L:SigDec:LeftReverse}, we obtain the following left reversing diagram
    \[
    \begin{tikzpicture}
    	\draw[-latex](0,0) -- (1.5,0) node[midway, below]{\small $\sigdec{k-j-1}1$};
    	\draw[-latex](1.5,0) -- (2,0) node[midway, below]{\small $\sig0$};
    	\draw[-latex](2,0) -- (3.5,0) node[midway, below]{\small $\sigdec{k-j}2$};
    	\draw[-latex](3.5,0) -- (5,0) node[midway, below]{\small $\sig1\,\sig0$};
    	\draw[-latex](2,0.5) -- (3.5,0.5) node[midway, above]{\small $\sigdec{k-j}2$};

    	\draw[-latex](2,0.5) -- (2,0) node[midway, right]{\small $\sig0$};
    	\draw[-latex](3.5,0.5) -- (3.5,0) node[midway, right]{\small $\sig0$};
    	\draw[-latex](2,1.5) -- (2,0.5) node[midway, left]{\small $\sigdec{k-j}1$};
    	\draw[-latex](3.5,1.5) -- (3.5,0.5) node[midway, right]{\small $\sig1$};
    	\draw[-latex](0,1.5) -- (0,0) node[midway, left]{\small $\sig{k-j}$};
    	\draw[-latex](5,1.5) -- (5,0) node[midway, right]{\small $\sig1$};
    
    	\draw[-latex](0,1.5) -- (2,1.5) node[midway, above]{\small $\varepsilon$};
    	\draw[-latex](2,1.5) -- (3.5,1.5) node[midway, above]{\small $u$};
    	\draw[-latex](3.5,1.5) -- (5,1.5) node[midway, above]{\small $\sig0\sig1\sig0$};
    	
    \draw (1.5,0) .. controls (1.5,0.25) and (1.75,0.5) .. (2,0.5) node[midway, left] {\small $\varepsilon$};
    \end{tikzpicture}
    \]
    which implies $\sigdec{k-j}0$ is a left $(\sig0,\sig1)$-chain and $\sigdec{k-j-1}0\sigdec{k-j}0$ is a left $(\sig{k-j},\sig1)$-chain.
    Since $\sig{k-j}$ is not in the support of $\sigdec00\cdots \sigdec{k-j-2}0$, the word $\muinc{0}{k-1}$ is a left $(\sig{k-j},\sig\jj)$-chain as summarized by:
    \[
    \begin{tikzpicture}
    \draw[-latex](-1,0)--(0.5,0) node[midway,below]{\small $\muinc{0}{k-1}$};
    \draw[-latex](-1,0.75)--(-1,0) node[midway, left]{\small $\sig{k-j}$};
    \draw[-latex](0.5,0.75)--(0.5,0) node[midway, right]{\small $\sig{j}$};
    \draw[gray,latex-](-0.75,0.7)--(0.25,0.7);
    \draw(1.55,0.35) node{$=$};
    \draw[-latex](3,0)--(4.5,0) node[midway,below]{\small $\sigdec00$};
    \draw[dashed](4.5,0)--(5.5,0);
    \draw[-latex](5.5,0)--(7,0) node[midway,below]{\small $\sigdec{k-j-2}0$};
    \draw[-latex](7,0)--(8.5,0) node[midway,below]{\small $\sigdec{k-j-1}0$};
    \draw[-latex](8.5,0)--(10,0) node[midway,below]{\small $\sigdec{k-j}0$};
    \draw[dashed](10,0)--(11,0);
    \draw[-latex](11,0)--(12.5,0) node[midway,below]{\small $\sigdec{k-1}0$};
    \draw[-latex](3,0.75) -- (3,0) node[midway,left]{\small $\sig{k-j}$};
    \draw[-latex](4.5,0.75) -- (4.5,0) node[midway, left]{\small $\sig{k-j}$};
    \draw[-latex](5.5,0.75) -- (5.5,0) node[midway, right]{\small $\sig{k-j}$};
    \draw[-latex](7,0.75) -- (7,0) node[midway, right]{\small $\sig{k-j}$};
    \draw[-latex](8.5,0.75) -- (8.5,0) node[midway, right]{\small $\sig0$};
    \draw[-latex](10,0.75) -- (10,0) node[midway, left]{\small $\sig1$};
    \draw[-latex](11,0.75) -- (11,0) node[midway, right]{\small $\sig{j-1}$};
    \draw[-latex](12.5,0.75) -- (12.5,0) node[midway, right]{\small $\sig{j}$};
    \draw[gray,latex-](3.25,0.7)--(4.25,0.7);
    \draw[gray,latex-](5.75,0.7)--(6.75,0.7);
    \draw[gray,latex-](7.25,0.7)--(8.25,0.7);
    \draw[gray,latex-](8.75,0.7)--(9.75,0.7);
    \draw[gray,latex-](11.25,0.7)--(12.25,0.7);
    \end{tikzpicture}
    \]
    For $s\not=0$, the word $\muinc{k-s}{k-1}$ is a suffix of $\muinc0\kko$ and is a left $(\sig{t},\sig\jj)$-chain with $t\in\{0,\ldots,j-1\}\cup\{k-j\}$ depending of the precise value of $s$.
    It follows that $\sig\jj$ is not a right divisor of $h_\ell$ and so $a$ can not belongs to $S_2\cap S_1$.
    
    Assume $a\in S_2\setminus S_1 = \{\sig\kk,\ldots,\sig\rro\}$.
    Since the support of $\muinc{k-s}{k-1}$ and $\muinc{0}{k-1}$ does not contain~$a$, both words are left $(a,a)$-chain.
    Hence, by Proposition~\ref{P:ChainDiv}, $a$ must be a right divisor of $h_{m}\cdots h_{\ell+1}$.
    From $a\in S_2$ and $\tail{S_2}{h_{m}\cdots h_{\ell+1}}=h_{\ell+1}$, $a$ must be a right divisor of 
    \[
    h_{\ell+1} = \begin{cases}
    	\overline{\nuinc{k}{k+s-1}} & \text{for $\ell=m$ and $s\not=0$,} \\
    	\overline{\nuinc{(i-1)k+s}{ik+s}} & \text{otherwise.}
    	\end{cases}
    \]
    By \eqref{E:EEE} we obtain $a=\sig\kk$. 
    Assume $\ell-1 = m-2$, we have 
    \[
    h_m\cdots h_{\ell+1}=h_m \, h_{m-1} = \overline{\muinc{k}{k+s-1}},
    \]
    The following diagram, based on \eqref{Sigdec:chain}, shows that $\muinc{k}{k+s-1}$ is a left $(\sig{k-s},\sig\kk)$-chain.
    \[
    \begin{tikzpicture}
    \draw[-latex](0,0) -- (1.5,0) node[midway, below]{\small $\sigdec\kk0$};
    \draw[-latex](1.5,0) -- (3,0) node[midway, below]{\small $\sigdec{k+1}0$};
    \draw[dashed](3,0) -- (4.5,0);
    \draw[-latex](4.5,0) -- (6,0) node[midway, below]{\small $\sigdec{k+s-2}0$};
    \draw[-latex](6,0) -- (7.5,0) node[midway, below]{\small $\sigdec{k+s-1}0$};
    \draw[-latex](0,0.75) -- (0,0) node[midway,left]{\small $\sig{k-s}$};
    \draw[-latex](1.5,0.75) -- (1.5,0) node[midway,left]{\small $\sig{k-s+1}$};
    \draw[-latex](3,0.75) -- (3,0) node[midway,left]{\small $\sig{k-s+2}$};
    \draw[-latex](4.5,0.75) -- (4.5,0) node[midway,right]{\small $\sig{k-2}$};
    \draw[-latex](6,0.75) -- (6,0) node[midway,right]{\small $\sig{k-1}$};
    \draw[-latex](7.5,0.75) -- (7.5,0) node[midway,right]{\small $\sig{k}$};
    \draw[gray,latex-](0.25,0.7)--(1.25,0.7);
    \draw[gray,latex-](1.75,0.7)--(2.75,0.7);
    \draw[gray,latex-](4.75,0.7)--(5.75,0.7);
    \draw[gray,latex-](6.25,0.7)--(7.25,0.7);
    \end{tikzpicture}
    \]
    It follows that $h_m\,h_{m-1}$ is not right divisible by $\sig\kk$ and so $\tail{S_2}{h_m\,h_{m-1}\,h_{m-2}}=h_{m-2}$.

    Assume $\ell - 1 \leq m-4$. 
    For this sub-case we prove that $\sig\kk$ is a left divisor of $h_{\ell-1}\cdots h_1$ which gives a contradiction since $\DD=h_m\cdots h_1$ is square free by Proposition~\ref{P:Garside}.
    We put $\ell - 1 = 2p$. From $\ell-1\in\{2,\ldots,m-4\}$, we obtain $p\in\{1,\ldots,q-1\}$ for $s\not=0$ and $p\in\{1,\ldots,q-2\}$ for $s=0$.
    In both cases we have $(p+1)k \leq r-1 = qk + s -1$.
    By hypothesis on $S_1$ we know that $\sig{(p+1)k}$ is a left divisor of $\DD_{S_1^\perp}$.
    As $\sig0$ commutes with $\sig{t}$ for $t\geq 2$, a consequence of relation \eqref{E:A:PC:DEC} is
    \begin{equation}
    	\label{E:B:PC:DEC}
    	\sigdec\ii0 \,\sig\jj \equiv \sig\jjo \,\sigdec\ii0 \qquad \text{for any $2\leq \jj \leq \ii$}.
    \end{equation}
    An immediate induction argument gives
    \[
    (\sigdec{r-k}1\ldots \sigdec{r-1}k)\cdot\sig{(p+1)k}\equiv\sig{pk}\cdot(\sigdec{r-k}1\ldots \sigdec{r-1}k)
    \]
    which implies that $\sig{pk}$ is a left divisor of $h_2= \overline{(\sigdec{r-k}1\ldots \sigdec{r-1}k)}\cdot \DD_{S_1^\perp}$.
    By definition of $h_{\ell-1},\ldots,h_3$ and \eqref{E:DecompMu} we have
    \begin{equation}
    h_{\ell-1}\cdots h_3 = \overline{\muinc{((q-p+1)k+s)}{r-1}}.
    \end{equation}
    Once again, by a direct induction based on \eqref{E:B:PC:DEC} together with the fact that $\muinc{((q-p+1)k+s)}{r-1}$ contains~$(p-1)k$ terms of the form $\sigdec{t}{0}$, we obtain
    \begin{equation}
    	(h_\ell\cdots h_3)\cdot \sig{pk} = \sig{k} \cdot (h_\ell\cdots h_3).
    \end{equation}
    We have then obtain that $\sig\kk$ is a left divisor of $h_\ell\cdots h_2$ and so of $h_\ell\cdots h_1$.
    This conclude the proof.
    \end{proof}
        
    \subsection{Remaining types}
    
    Thanks to the results of the previous subsections, it remains to study Condition~$B(p)$ in types $H_3$, $H_4$ and~$F_4$.
    Note that in these types, the Garside element $\DD$ is central.
    It follows that any covering is $\DD^p$-regular.
    
    For some coverings, computer assisted computations, help us to find braids $a$ and $b$ such that   $a$, $b$ and $ab$ are $\DD$-unmovable,  $(a,b)\not\in \oTheta(p)\times \oTheta(p)$ for any $p\geq 1$ and $\dpt(a)+\dpt(b) - \dpt(ab)\not \in\{0,1\}$.  Then, using Lemma~\ref{L:B1ToBj}, we obtain that $(S_2,S_1)$ does not satisfy Condition $B(p)$ whatever $p\geq 1$.
    Our experiments were carried out on the shared computational platform \textsc{Calculco} \cite{Calculco}. We use the notations of Figure~\ref{F:CoxDiag}.
    \subsubsection{Type $H_3$}
    
    The six proper and irreducible coverings of $H_3$ are 
    \[
    \small
    \begin{array}{|c|c|c|c|c|c|c|}
    \hline
    \text{Id} & 1 & 2 & 3 & 4 & 5 & 6 \\
    \hline
    S_1 & 
    \begin{tikzpicture}[x=10pt,y=10pt]
    \fill[lightgray](1,0) circle (0.2);
    \fill[lightgray](2,0) circle (0.2);
    \draw[lightgray, line width=1](0,0) -- (2,0);
    \draw[lightgray] (0.5,0) node[above]{\tiny$5$};
    \fill(0,0) circle (0.2);
    \end{tikzpicture}
    &
    \begin{tikzpicture}[x=10pt,y=10pt]
    \fill[lightgray](0,0) circle (0.2);
    \fill[lightgray](1,0) circle (0.2);
    \draw[lightgray, line width=1](0,0) -- (2,0);
    \draw[lightgray] (0.5,0) node[above]{\tiny$5$};
    \fill(2,0) circle (0.2);
    \end{tikzpicture}
    &
    \begin{tikzpicture}[x=10pt,y=10pt]
    \fill[lightgray](2,0) circle (0.2);
    \draw[line width=1](0,0) -- (1,0);
    \draw[lightgray, line width=1](1,0) -- (2,0);
    \draw(0.5,0) node[above]{\tiny$5$};
    \fill(0,0) circle (0.2);
    \fill(1,0) circle (0.2);
    \end{tikzpicture}
    &
    \begin{tikzpicture}[x=10pt,y=10pt]
    \fill[lightgray](2,0) circle (0.2);
    \draw[line width=1](0,0) -- (1,0);
    \draw[lightgray, line width=1](1,0) -- (2,0);
    \draw(0.5,0) node[above]{\tiny$5$};
    \fill(0,0) circle (0.2);
    \fill(1,0) circle (0.2);
    \end{tikzpicture}
    &
    \begin{tikzpicture}[x=10pt,y=10pt]
    \fill[lightgray](0,0) circle (0.2);
    \draw[line width=1](1,0) -- (2,0);
    \draw[lightgray, line width=1](0,0) -- (1,0);
    \draw[lightgray](0.5,0) node[above]{\tiny$5$};
    \fill(1,0) circle (0.2);
    \fill(2,0) circle (0.2);
    \end{tikzpicture}
    &
    \begin{tikzpicture}[x=10pt,y=10pt]
    \fill[lightgray](0,0) circle (0.2);
    \draw[line width=1](1,0) -- (2,0);
    \draw[lightgray, line width=1](0,0) -- (1,0);
    \draw[lightgray](0.5,0) node[above]{\tiny$5$};
    \fill(1,0) circle (0.2);
    \fill(2,0) circle (0.2);
    \end{tikzpicture}
    \\
    \hline
    S_2 &
    \begin{tikzpicture}[x=10pt,y=10pt]
    \fill[lightgray](0,0) circle (0.2);
    \draw[lightgray, line width=1](0,0) -- (1,0);
    \draw[lightgray] (0.5,0) node[above]{\tiny$5$};
    \fill(1,0) circle (0.2);
    \fill(2,0) circle (0.2);
    \draw[line width=1](1,0) -- (2,0);
    \end{tikzpicture}
    &
    \begin{tikzpicture}[x=10pt,y=10pt]
    \fill[lightgray](2,0) circle (0.2);
    \draw[line width=1](0,0) -- (1,0);
    \draw[lightgray, line width=1](1,0) -- (2,0);
    \draw(0.5,0) node[above]{\tiny$5$};
    \fill(0,0) circle (0.2);
    \fill(1,0) circle (0.2);
    \end{tikzpicture}
    &
    \begin{tikzpicture}[x=10pt,y=10pt]
    \fill[lightgray](0,0) circle (0.2);
    \fill[lightgray](1,0) circle (0.2);
    \draw[lightgray, line width=1](0,0) -- (2,0);
    \draw[lightgray] (0.5,0) node[above]{\tiny$5$};
    \fill(2,0) circle (0.2);
    \end{tikzpicture}
    &
    \begin{tikzpicture}[x=10pt,y=10pt]
    \fill[lightgray](0,0) circle (0.2);
    \draw[line width=1](1,0) -- (2,0);
    \draw[lightgray, line width=1](0,0) -- (1,0);
    \draw[lightgray](0.5,0) node[above]{\tiny$5$};
    \fill(1,0) circle (0.2);
    \fill(2,0) circle (0.2);
    \end{tikzpicture}
    &
    \begin{tikzpicture}[x=10pt,y=10pt]
    \fill[lightgray](1,0) circle (0.2);
    \fill[lightgray](2,0) circle (0.2);
    \draw[lightgray, line width=1](0,0) -- (2,0);
    \draw[lightgray] (0.5,0) node[above]{\tiny$5$};
    \fill(0,0) circle (0.2);
    \end{tikzpicture}
    &
    \begin{tikzpicture}[x=10pt,y=10pt]
    \fill[lightgray](2,0) circle (0.2);
    \draw[line width=1](0,0) -- (1,0);
    \draw[lightgray, line width=1](1,0) -- (2,0);
    \draw(0.5,0) node[above]{\tiny$5$};
    \fill(0,0) circle (0.2);
    \fill(1,0) circle (0.2);
    \end{tikzpicture}
    \\
    \hline
    \end{array}
    \]
    Coverings with $\text{Id}=1$ and $5$ never satisfy Condition $B(p)$ whatever $p\geq 1$.
    Indeed, we have found the following counter-example to Condition $B(p)$ for all $p\geq 1$.
    \[
    \small
    \begin{array}{|c|c|c|c|c|c|c|}
    \hline
    \text{Id}& a & b & \dpt(a) & \dpt(b) & \dpt(ab)\\
    \hline
    1 & \sig1\sig2\sig1\sig2\sig3\sig2\sig1\sig2^2\sig1\sig2\sig1\sig3\sig2\sig1\sig2^2\sig1\sig3\sig2\sig1\sig2 &\sig1 & 8 & 1& 7\\
    \hline
    5 &\sig1\sig2\sig1\sig2\sig1\sig3\sig2\sig1\sig2\sig3\sig2\sig1\sig2\sig2\sig1\sig2\sig1  &\sig2 &6& 1 & 5\\
    \hline
    \end{array}
    \]
    
    It remains $4$ proper and irreducible coverings in types $H_3$ with an unknown status relatively to Condition~$B(p)$.
    
    \begin{prp}
    Assume $M$ is of type $H_3$.  A  $\DD^p$-regular coverings $(S_2,S_1)$  of $S$ that  may satisfy Condition $B(p)$  has  its $\text{Id}$  equal  to $2,3,4$, or $6$ in the above list, that is  $S_1$ or $S_2$ is equal to $\{\sig1,\sig2\}$
    \end{prp}
    
    \subsubsection{Type $H_4$}
    
    The $12$ proper and irreducible coverings of $H_4$ are 
    
    \[
    \small
    \begin{array}{|c|c|c|c|c|c|c|}
    \hline
    \text{Id} & 1 & 2 & 3 & 4 & 5 & 6 \\
    \hline
    S_1 & 
    \begin{tikzpicture}[x=10pt,y=10pt]
    \fill[lightgray](1,0) circle (0.2);
    \fill[lightgray](2,0) circle (0.2);
    \fill[lightgray](3,0) circle (0.2);
    \draw[lightgray, line width=1](0,0) -- (3,0);
    \draw[lightgray] (0.5,0) node[above]{\tiny$5$};
    \fill(0,0) circle (0.2);
    \end{tikzpicture}
    &
    \begin{tikzpicture}[x=10pt,y=10pt]
    \fill[lightgray](2,0) circle (0.2);
    \fill[lightgray](3,0) circle (0.2);
    \draw[line width=1](0,0) -- (1,0);
    \draw[lightgray, line width=1](1,0) -- (3,0);
    \draw(0.5,0) node[above]{\tiny$5$};
    \fill(0,0) circle (0.2);
    \fill(1,0) circle (0.2);
    \end{tikzpicture}
    &
    \begin{tikzpicture}[x=10pt,y=10pt]
    \fill[lightgray](2,0) circle (0.2);
    \fill[lightgray](3,0) circle (0.2);
    \draw[line width=1](0,0) -- (1,0);
    \draw[lightgray, line width=1](1,0) -- (3,0);
    \draw(0.5,0) node[above]{\tiny$5$};
    \fill(0,0) circle (0.2);
    \fill(1,0) circle (0.2);
    \end{tikzpicture}
    &
    \begin{tikzpicture}[x=10pt,y=10pt]
    \fill[lightgray](3,0) circle (0.2);
    \draw[line width=1](0,0) -- (2,0);
    \draw[lightgray, line width=1](2,0) -- (3,0);
    \draw(0.5,0) node[above]{\tiny$5$};
    \fill(0,0) circle (0.2);
    \fill(1,0) circle (0.2);
    \fill(2,0) circle (0.2);
    \end{tikzpicture}
    &
    \begin{tikzpicture}[x=10pt,y=10pt]
    \fill[lightgray](3,0) circle (0.2);
    \draw[line width=1](0,0) -- (2,0);
    \draw[lightgray, line width=1](2,0) -- (3,0);
    \draw(0.5,0) node[above]{\tiny$5$};
    \fill(0,0) circle (0.2);
    \fill(1,0) circle (0.2);
    \fill(2,0) circle (0.2);
    \end{tikzpicture}
    &
    \begin{tikzpicture}[x=10pt,y=10pt]
    \fill[lightgray](3,0) circle (0.2);
    \draw[line width=1](0,0) -- (2,0);
    \draw[lightgray, line width=1](2,0) -- (3,0);
    \draw(0.5,0) node[above]{\tiny$5$};
    \fill(0,0) circle (0.2);
    \fill(1,0) circle (0.2);
    \fill(2,0) circle (0.2);
    \end{tikzpicture}
    \\
    \hline
    S_2 &
    \begin{tikzpicture}[x=10pt,y=10pt]
    \fill[lightgray](0,0) circle (0.2);
    \draw[lightgray, line width=1](0,0) -- (1,0);
    \draw[lightgray] (0.5,0) node[above]{\tiny$5$};
    \fill(1,0) circle (0.2);
    \fill(2,0) circle (0.2);
    \fill(3,0) circle (0.2);
    \draw[line width=1](1,0) -- (3,0);
    \end{tikzpicture}
    &
    \begin{tikzpicture}[x=10pt,y=10pt]
    \fill[lightgray](0,0) circle (0.2);
    \draw[lightgray, line width=1](0,0) -- (1,0);
    \draw[lightgray] (0.5,0) node[above]{\tiny$5$};
    \fill(1,0) circle (0.2);
    \fill(2,0) circle (0.2);
    \fill(3,0) circle (0.2);
    \draw[line width=1](1,0) -- (3,0);
    \end{tikzpicture}
    &
    \begin{tikzpicture}[x=10pt,y=10pt]
    \fill[lightgray](0,0) circle (0.2);
    \fill[lightgray](1,0) circle (0.2);
    \draw[lightgray, line width=1](0,0) -- (2,0);
    \draw[lightgray] (0.5,0) node[above]{\tiny$5$};
    \fill(2,0) circle (0.2);
    \fill(3,0) circle (0.2);
    \draw[line width=1](2,0) -- (3,0);
    \end{tikzpicture}
    &
    \begin{tikzpicture}[x=10pt,y=10pt]
    \fill[lightgray](0,0) circle (0.2);
    \draw[lightgray, line width=1](0,0) -- (1,0);
    \draw[lightgray] (0.5,0) node[above]{\tiny$5$};
    \fill(1,0) circle (0.2);
    \fill(2,0) circle (0.2);
    \fill(3,0) circle (0.2);
    \draw[line width=1](1,0) -- (3,0);
    \end{tikzpicture}
    &
    \begin{tikzpicture}[x=10pt,y=10pt]
    \fill[lightgray](0,0) circle (0.2);
    \fill[lightgray](1,0) circle (0.2);
    \draw[lightgray, line width=1](0,0) -- (2,0);
    \draw[lightgray] (0.5,0) node[above]{\tiny$5$};
    \fill(2,0) circle (0.2);
    \fill(3,0) circle (0.2);
    \draw[line width=1](2,0) -- (3,0);
    \end{tikzpicture}
    &
    \begin{tikzpicture}[x=10pt,y=10pt]
    \fill[lightgray](0,0) circle (0.2);
    \fill[lightgray](1,0) circle (0.2);
    \fill[lightgray](2,0) circle (0.2);
    \draw[lightgray, line width=1](0,0) -- (3,0);
    \draw[lightgray] (0.5,0) node[above]{\tiny$5$};
    \fill(3,0) circle (0.2);
    \end{tikzpicture}
    \\
    \hline
    \end{array}
    \]
    \[
    \small
    \begin{array}{|c|c|c|c|c|c|c|}
    \hline
    \text{Id} & 7 & 8 & 9 & 10 & 11 & 12 \\
    \hline
    S_1 & 
    \begin{tikzpicture}[x=10pt,y=10pt]
    \fill[lightgray](0,0) circle (0.2);
    \draw[lightgray, line width=1](0,0) -- (1,0);
    \draw[lightgray] (0.5,0) node[above]{\tiny$5$};
    \fill(1,0) circle (0.2);
    \fill(2,0) circle (0.2);
    \fill(3,0) circle (0.2);
    \draw[line width=1](1,0) -- (3,0);
    \end{tikzpicture}
    &
    \begin{tikzpicture}[x=10pt,y=10pt]
    \fill[lightgray](0,0) circle (0.2);
    \draw[lightgray, line width=1](0,0) -- (1,0);
    \draw[lightgray] (0.5,0) node[above]{\tiny$5$};
    \fill(1,0) circle (0.2);
    \fill(2,0) circle (0.2);
    \fill(3,0) circle (0.2);
    \draw[line width=1](1,0) -- (3,0);
    \end{tikzpicture}
    &
    \begin{tikzpicture}[x=10pt,y=10pt]
    \fill[lightgray](0,0) circle (0.2);
    \draw[lightgray, line width=1](0,0) -- (1,0);
    \draw[lightgray] (0.5,0) node[above]{\tiny$5$};
    \fill(1,0) circle (0.2);
    \fill(2,0) circle (0.2);
    \fill(3,0) circle (0.2);
    \draw[line width=1](1,0) -- (3,0);
    \end{tikzpicture}
    &
    \begin{tikzpicture}[x=10pt,y=10pt]
    \fill[lightgray](0,0) circle (0.2);
    \fill[lightgray](1,0) circle (0.2);
    \draw[lightgray, line width=1](0,0) -- (2,0);
    \draw[lightgray] (0.5,0) node[above]{\tiny$5$};
    \fill(2,0) circle (0.2);
    \fill(3,0) circle (0.2);
    \draw[line width=1](2,0) -- (3,0);
    \end{tikzpicture}
    &
    \begin{tikzpicture}[x=10pt,y=10pt]
    \fill[lightgray](0,0) circle (0.2);
    \fill[lightgray](1,0) circle (0.2);
    \draw[lightgray, line width=1](0,0) -- (2,0);
    \draw[lightgray] (0.5,0) node[above]{\tiny$5$};
    \fill(2,0) circle (0.2);
    \fill(3,0) circle (0.2);
    \draw[line width=1](2,0) -- (3,0);
    \end{tikzpicture}
    &
    \begin{tikzpicture}[x=10pt,y=10pt]
    \fill[lightgray](0,0) circle (0.2);
    \fill[lightgray](1,0) circle (0.2);
    \fill[lightgray](2,0) circle (0.2);
    \draw[lightgray, line width=1](0,0) -- (3,0);
    \draw[lightgray] (0.5,0) node[above]{\tiny$5$};
    \fill(3,0) circle (0.2);
    \end{tikzpicture}
    \\
    \hline
    S_2 &
    \begin{tikzpicture}[x=10pt,y=10pt]
    \fill[lightgray](1,0) circle (0.2);
    \fill[lightgray](2,0) circle (0.2);
    \fill[lightgray](3,0) circle (0.2);
    \draw[lightgray, line width=1](0,0) -- (3,0);
    \draw[lightgray] (0.5,0) node[above]{\tiny$5$};
    \fill(0,0) circle (0.2);
    \end{tikzpicture}
    &
    \begin{tikzpicture}[x=10pt,y=10pt]
    \fill[lightgray](2,0) circle (0.2);
    \fill[lightgray](3,0) circle (0.2);
    \draw[line width=1](0,0) -- (1,0);
    \draw[lightgray, line width=1](1,0) -- (3,0);
    \draw(0.5,0) node[above]{\tiny$5$};
    \fill(0,0) circle (0.2);
    \fill(1,0) circle (0.2);
    \end{tikzpicture}
    &
    \begin{tikzpicture}[x=10pt,y=10pt]
    \fill[lightgray](3,0) circle (0.2);
    \draw[line width=1](0,0) -- (2,0);
    \draw[lightgray, line width=1](2,0) -- (3,0);
    \draw(0.5,0) node[above]{\tiny$5$};
    \fill(0,0) circle (0.2);
    \fill(1,0) circle (0.2);
    \fill(2,0) circle (0.2);
    \end{tikzpicture}
    &
    \begin{tikzpicture}[x=10pt,y=10pt]
    \fill[lightgray](2,0) circle (0.2);
    \fill[lightgray](3,0) circle (0.2);
    \draw[line width=1](0,0) -- (1,0);
    \draw[lightgray, line width=1](1,0) -- (3,0);
    \draw(0.5,0) node[above]{\tiny$5$};
    \fill(0,0) circle (0.2);
    \fill(1,0) circle (0.2);
    \end{tikzpicture}
    &
    \begin{tikzpicture}[x=10pt,y=10pt]
    \fill[lightgray](3,0) circle (0.2);
    \draw[line width=1](0,0) -- (2,0);
    \draw[lightgray, line width=1](2,0) -- (3,0);
    \draw(0.5,0) node[above]{\tiny$5$};
    \fill(0,0) circle (0.2);
    \fill(1,0) circle (0.2);
    \fill(2,0) circle (0.2);
    \end{tikzpicture}
    &
    \begin{tikzpicture}[x=10pt,y=10pt]
    \fill[lightgray](3,0) circle (0.2);
    \draw[line width=1](0,0) -- (2,0);
    \draw[lightgray, line width=1](2,0) -- (3,0);
    \draw(0.5,0) node[above]{\tiny$5$};
    \fill(0,0) circle (0.2);
    \fill(1,0) circle (0.2);
    \fill(2,0) circle (0.2);
    \end{tikzpicture}
    \\
    \hline
    \end{array}
    \]

    Almost all these coverings do not satisfy Condition $B(p)$ whatever $p\geq 1$ : 
    
    \begin{prp}
    Assume $M$ is of type $H_4$. A  $\DD^p$-regular coverings $(S_2,S_1)$  of $S$ that  may satisfy Condition $B(p)$  has  its $\text{Id}$  equal  to  $5$ or to $8$ in the above list, that is  
    \[
    (S_2,S_1) =(\{\sig1,\sig2,\sig3\},\{\sig3,\sig4\}),\quad\text{or }\quad (S_2,S_1) = (\{\sig2,\sig3,\sig4\},\{\sig1,\sig2\}),
    \]
 
    \end{prp}

    \begin{proof}
    Using Proposition~\ref{P:ConditionB:SubtypeA} with $T=\{\sig2,\sig3,\sig4\}$ 
    we  can show that coverings with $\text{Id}\in\{3,6,11,12\}$ do not satisfy Condition $B(p)$ for any $p\geq 1$.
    Coverings with $\text{Id}\in\{1,7\}$ are treated with Proposition~\ref{P:ConditionB:Induction} and counter examples of type $H_3$ (the then first and the second respectively).
    For coverings with $\text{Id}\in\{9,10\}$ we conclude with Proposition~\ref{P:ConditionB:Crit:TypeB} since we have 
    \begin{align*}
    g_n&=g_{10}=  \sig1\sig2\sig3\sig1\sig2\sig1\not\in M_{S_2\setminus S_1} & \text{for $\text{Id}=9$},\\
    g_n&=g_{10}= \sig2\sig1\sig2\sig3\sig2\sig1\sig2\sig1\not\in M_{S_2\setminus S_1} & \text{for $\text{Id}=10$}.
    \end{align*}
    We have found the following counter-examples for coverings with $\text{Id}$ in $\{2,4\}$:
     \[
     \small
    \begin{array}{|c|c|c|c|c|c|c|}
    \hline
    \text{Id}& a & b & \dpt(a) & \dpt(b) & \dpt(ab)\\
    \hline
    2 &  \sig1\sig2\sig1\sig2\sig3\sig2\sig1\sig2\sig3\sig2\sig3\sig2\sig1\sig2\sig1\sig4\sig3\sig2\sig1\sig2\sig1\sig3\sig2\sig1\sig2&\sig3 & 4 & 1& 3\\
    \hline
    4 &\sig1\sig2\sig1\sig3\sig2\sig1\sig2\sig4\sig3\sig2\sig2\sig1\sig2\sig1\sig3\sig2\sig1\sig2\sig1\sig3\sig2\sig1\sig2\sig4\sig3\sig2\sig1\sig2& \sig4 & 5 & 1 & 4 \\
    & \sig1\sig3\sig2\sig1\sig2\sig3\sig4\sig3\sig2\sig1\sig2\sig1\sig3\sig2\sig1\sig2\sig3\sig4\sig3\sig2\sig1\sig2\sig1\sig3\sig2\sig4\sig3  & & &&\\
    \hline
    \end{array}
    \]
    \end{proof}
    
    \subsubsection{Type $F_4$}
    
    Up to the graph automorphism $\sig\ii\leftrightarrow \sig{5-\ii}$, there is $6$ proper and irreducible coverings $(S_2,S_1)$ of $F_4$:
    \[
    \small
    \begin{array}{|c|c|c|c|c|c|c|}
    \hline
    \text{Id} & 1 & 2 & 3 & 4 & 5 & 6 \\
    \hline 
     S_1 &
    \begin{tikzpicture}[x=10pt,y=10pt]
    \fill[lightgray](0,0) circle (0.2);
    \draw[lightgray, line width=1](0,0) -- (1,0);
    \draw(1.5,0) node[above]{\tiny$4$};
    \fill(1,0) circle (0.2);
    \fill(2,0) circle (0.2);
    \fill(3,0) circle (0.2);
    \draw[line width=1](1,0) -- (3,0);
    \end{tikzpicture}
    &
    \begin{tikzpicture}[x=10pt,y=10pt]
    \fill[lightgray](0,0) circle (0.2);
    \fill[lightgray](1,0) circle (0.2);
    \draw[lightgray, line width=1](0,0) -- (2,0);
    \draw[lightgray] (1.5,0) node[above]{\tiny$4$};
    \fill(2,0) circle (0.2);
    \fill(3,0) circle (0.2);
    \draw[line width=1](2,0) -- (3,0);
    \end{tikzpicture}
    &
    \begin{tikzpicture}[x=10pt,y=10pt]
    \fill[lightgray](0,0) circle (0.2);
    \draw[lightgray, line width=1](0,0) -- (1,0);
    \draw(1.5,0) node[above]{\tiny$4$};
    \fill(1,0) circle (0.2);
    \fill(2,0) circle (0.2);
    \fill(3,0) circle (0.2);
    \draw[line width=1](1,0) -- (3,0);
    \end{tikzpicture}
    &
    \begin{tikzpicture}[x=10pt,y=10pt]
    \fill[lightgray](0,0) circle (0.2);
    \fill[lightgray](1,0) circle (0.2);
    \fill[lightgray](2,0) circle (0.2);
    \draw[lightgray, line width=1](0,0) -- (3,0);
    \draw[lightgray] (1.5,0) node[above]{\tiny$4$};
    \fill(3,0) circle (0.2);
    \end{tikzpicture}
    &
    \begin{tikzpicture}[x=10pt,y=10pt]
    \fill[lightgray](0,0) circle (0.2);
    \fill[lightgray](1,0) circle (0.2);
    \draw[lightgray, line width=1](0,0) -- (2,0);
    \draw[lightgray] (1.5,0) node[above]{\tiny$4$};
    \fill(2,0) circle (0.2);
    \fill(3,0) circle (0.2);
    \draw[line width=1](2,0) -- (3,0);
    \end{tikzpicture}&
    \begin{tikzpicture}[x=10pt,y=10pt]
    \fill[lightgray](0,0) circle (0.2);
    \draw[lightgray, line width=1](0,0) -- (1,0);
    \draw(1.5,0) node[above]{\tiny$4$};
    \fill(1,0) circle (0.2);
    \fill(2,0) circle (0.2);
    \fill(3,0) circle (0.2);
    \draw[line width=1](1,0) -- (3,0);
    \end{tikzpicture}
     \\
     \hline 
     S_2 &
      \begin{tikzpicture}[x=10pt,y=10pt]
    \fill[lightgray](1,0) circle (0.2);
    \fill[lightgray](2,0) circle (0.2);
    \fill[lightgray](3,0) circle (0.2);
    \draw[lightgray, line width=1](0,0) -- (3,0);
    \draw[lightgray] (1.5,0) node[above]{\tiny$4$};
    \fill(0,0) circle (0.2);
    \end{tikzpicture}
    &
    \begin{tikzpicture}[x=10pt,y=10pt]
    \fill[lightgray](2,0) circle (0.2);
    \fill[lightgray](3,0) circle (0.2);
    \draw[line width=1](0,0) -- (1,0);
    \draw[lightgray, line width=1](1,0) -- (3,0);
    \draw[lightgray] (1.5,0) node[above]{\tiny$4$};
    \fill(0,0) circle (0.2);
    \fill(1,0) circle (0.2);
    \end{tikzpicture}
    &
    \begin{tikzpicture}[x=10pt,y=10pt]
    \fill[lightgray](2,0) circle (0.2);
    \fill[lightgray](3,0) circle (0.2);
    \draw[line width=1](0,0) -- (1,0);
    \draw[lightgray, line width=1](1,0) -- (3,0);
    \draw[lightgray] (1.5,0) node[above]{\tiny$4$};
    \fill(0,0) circle (0.2);
    \fill(1,0) circle (0.2);
    \end{tikzpicture}
    &
    \begin{tikzpicture}[x=10pt,y=10pt]
    \fill[lightgray](3,0) circle (0.2);
    \draw[line width=1](0,0) -- (2,0);
    \draw[lightgray, line width=1](2,0) -- (3,0);
    \draw(1.5,0) node[above]{\tiny$4$};
    \fill(0,0) circle (0.2);
    \fill(1,0) circle (0.2);
    \fill(2,0) circle (0.2);
    \end{tikzpicture}
    &
    \begin{tikzpicture}[x=10pt,y=10pt]
    \fill[lightgray](3,0) circle (0.2);
    \draw[line width=1](0,0) -- (2,0);
    \draw[lightgray, line width=1](2,0) -- (3,0);
    \draw(1.5,0) node[above]{\tiny$4$};
    \fill(0,0) circle (0.2);
    \fill(1,0) circle (0.2);
    \fill(2,0) circle (0.2);
    \end{tikzpicture}
    &
    \begin{tikzpicture}[x=10pt,y=10pt]
    \fill[lightgray](3,0) circle (0.2);
    \draw[line width=1](0,0) -- (2,0);
    \draw[lightgray, line width=1](2,0) -- (3,0);
    \draw(1.5,0) node[above]{\tiny$4$};
    \fill(0,0) circle (0.2);
    \fill(1,0) circle (0.2);
    \fill(2,0) circle (0.2);
    \end{tikzpicture}
    \\
    \hline
    \end{array}
    \]
    
    The following Proposition closes the discussion in type $F_4$.
    \begin{prp} Assume $M$ is of type $F_4$.
    No $\DD^p$ regular covering of $S$ satisfies Condition $B(p)$ wharever~$p\geq 1$.
    \end{prp}
    
    \begin{proof}
    For coverings with $\text{Id}\in\{5,6\}$ we conclude with Proposition~\ref{P:ConditionB:Crit:TypeB} since we have 
    \begin{align*}
    g_n&=g_6=\sig2\sig3\sig1\sig2\not\in M_{S_2\setminus S_1}&\text{for $\text{Id}=5$},\\
    g_n&=g_4=\sig1\sig2\sig3\sig2\sig2\not\in M_{S_2\setminus S_1}&\text{for $\text{Id}=6$}.
    \end{align*}
    We have found the following counter-examples for the remaining coverings:
     \[
     \small
    \begin{array}{|c|c|c|c|c|c|c|}
    \hline
    \text{Id}& a & b & \dpt(a) & \dpt(b) & \dpt(ab)\\
    \hline
    1 &  \sig1\sig2\sig1\sig2\sig3\sig2\sig1\sig2\sig3\sig2\sig3\sig2\sig1\sig2\sig1\sig4\sig3\sig2\sig1\sig2\sig1\sig3\sig2\sig1\sig2 &\sig1 & 3 & 1& 2\\
    \hline
    2 &\sig2\sig1\sig4\sig3\sig2\sig1\sig3\sig2\sig3\sig4\sig3\sig2\sig2\sig1\sig3\sig2\sig3\sig4\sig3\sig2\sig1\sig3\sig2\sig3\sig4\sig3\sig2\sig1\sig3\sig2\sig3 & \sig2 & 8 & 1 & 7 \\
    \hline
    3 &\sig3\sig4\sig3\sig2\sig1\sig3\sig2\sig3\sig4\sig3\sig2\sig1\sig3\sig3\sig2\sig3\sig4\sig3\sig2\sig1\sig3\sig2\sig4\sig3& \sig1 & 3 & 1 & 2\\
    \hline  
    4 & \sig2\sig3\sig2\sig1\sig3\sig2\sig3\sig4\sig3\sig2\sig1\sig3\sig4\sig3\sig2\sig3\sig4\sig3\sig2\sig1\sig3\sig2\sig4\sig3\sig2\sig1& \sig2 & 4 & 1& 3\\
    \hline
    \end{array}
    \]
    \end{proof}

    \bibliographystyle{acm}
    \bibliography{biblio}
    \end{document}